\documentclass[reqno]{amsart}
\usepackage[T1]{fontenc}
\usepackage{amssymb,bbm}
\usepackage{amsmath}
\usepackage{amsfonts}
\usepackage{leftidx}

\usepackage{graphicx}
\usepackage{comment}
\usepackage{xcolor}

\usepackage{booktabs}

\newtheorem{theorem}{Theorem}[section]
\newtheorem{lemma}[theorem]{Lemma}
\newtheorem{proposition}[theorem]{Proposition}

\theoremstyle{definition}
\newtheorem{definition}[theorem]{Definition}
\theoremstyle{remark}
\newtheorem{remark}[theorem]{Remark}
\numberwithin{equation}{section}

\newcommand{\N}{\mathbb{N}}
\newcommand{\Z}{\mathbb{Z}}

\newcommand{\C}{\mathbb{C}}

\newcounter{margnotes}

\newcommand{\M}{\mathbb{M}}

\newcommand{\F}{\mathcal{F}}
\newcommand{\G}{\mathcal{G}}

\newcommand{\CP}{\mathcal{P}}

\newcommand{\p}{\mathsf{p}}
\newcommand{\q}{\mathsf{q}}

\newcommand{\imb}{\boldsymbol{\imath}}
\newcommand{\jmb}{\boldsymbol{\jmath}}

\title[Curves of tangencies of foliation pairs]{Curves of tangencies of foliation pairs and normalizing transformations}

	\author[J.A. Jaurez-Rosas]{J.A. Jaurez-Rosas}
\address[a]{Departamento de Matem\'aticas, Facultad de Ciencias, Universidad Nacional Aut\'onoma de 
M\'exico 	(UNAM),  Circuito 
exterior, Ciudad 	Universitaria, 04510, Ciudad de M\'exico, M\'exico}
\email{jessica.jaurez@ciencias.unam.mx}

\author[L. Ortiz-Bobadilla]{L. Ortiz-Bobadilla}
\address[b]{Instituto de Matem\'aticas, Universidad Nacional Aut\'onoma de 
M\'exico 	(UNAM), 	\'Area de la Investigaci\'on Cient\'ifica, Circuito 
exterior, Ciudad 	Universitaria, 04510, Ciudad de M\'exico, M\'exico}
\email{laura@im.unam.mx}

\author[S.M. Voronin]{S.M. Voronin}
\address[b]{Department of Mathematics, Chelyabinsk State University  
ul. Bratiev Kashirinykh 129, 454021, Chelyabinsk, Russian Federation}
\email{voron@csu.ru}
	\subjclass{32S65, 34M35, 34M45, 34M56, 34M50, 32S05, 32B10}
	
 \keywords{ Non dicritical and  dicritical foliation pairs, normal forms, holonomy, analytic invariants, polar curves}
 
\thanks{This work was supported by  
 Papiit (Dgapa UNAM) IN103123}

\begin{document}

\maketitle
 
\begin{abstract}
In this work we give a complete description of the collection of curves of tangencies induced by germs of foliation pairs --\emph{non dicritical and dicritical}-- given by analytic differential equations with degenerated non dicritical and dicritical singularities, satisfying some genericity assumptions. To this purpose we use local models and analytic normalizing transformations. Moreover, for each natural number $k$ we obtain  \emph{k-normal forms} for the normalizing transformations. These normal forms are used to give parametrizations, up to a finite jet, of the branches of the curves of tangencies.

We also prove that under natural genericity assumptions  any germ of analytic curve having pairwise transversal smooth branches is realized as curve of tangencies of a --\emph{non dicritical and dicritical}-- foliation pair. 
\end{abstract}
\vspace{0.3cm}
\centerline{\today}
 
\section{\textbf{\emph{Introduction}}}\label{sec: Introduction}




The aim of this work is to give, under some genericity assumptions, a complete description of the collection of curves of tangencies induced by germs of foliation pairs --\emph{non dicritical and dicritical}-- given by analytic differential equations with degenerated non dicritical and dicritical singularities. To this purpose we use local models and analytic normalizing transformations. 

Local models are simple and informative models that carry analytic information at singular and tangency points for non dicritical and dicritical foliations, respectively. They can be understood as local analytic representatives of foliation pairs.
Local models and normalizing transformations are precisely described in Definition \ref{def: ML_TN}. 

In \cite {[JOV]} a nonlinear Riemann-Hilbert type problem for --\emph{non dicritical and dicritical}-- foliation pairs is solved. Namely, for given \emph{local models} carring monodromy data and involutions, it is proved that a suitable collection of normalizing transformations can be found such that it is possible to realize the local models in a foliation pair --non dicritical and dicritical--. Namely,

\begin{theorem}[Realization of local models of foliation pairs, \cite {[JOV]}]\label{teo_S: realización de modelos locales y t_norm}
Let $\p= (p_{1}, \ldots, p_{n+1})$ and $\q = (q_{1}, \ldots, q_{m})$ be collections of  $n+1$ and $m$ mutually distinct complex numbers such that for any  $k=1,\dots,n+1$ and $j= 1,\dots,m$, $p_{k}\neq q_{j}$. Any collection of $n+1+m$ pairs of generic local models --\emph{non dicritical, dicritical}--  related to a fixed pair $(\p,\q)$, can be realized, through normalizing transformations, by a corresponding foliation pair $(\F, \G)$,
\begin{equation*}
    \F \in \mathcal{N}_{\p} (\mathbf{h})\quad\text{and}\quad\G \in \mathcal{D}_{\q} ( \mathcal{I})\,,
\end{equation*}
where $\mathcal{N}_{\p} (\mathbf{h})$ is the set of germs of non dicritical foliations having strictly analytically equivalent hidden holonomy $\mathbf{h}$, related to the $n+1$ singular points $\p= (p_{1}, \ldots, p_{n+1})$, and $\mathcal{D}_{\q} ( \mathcal{I})$ is the set of germs of dicritical foliations with simple tangencies at the points $\q = (q_{1}, \ldots, q_{m})$, having associated a collection $\mathcal{I} = (I_{1}, \ldots, I_{m})$ of holomorphic involutions  $I_{j} \colon (\C, q_{j}) \rightarrow (\C, q_{j})$ with fixed points $q_{j}$.
\end{theorem}

In \cite {[JOV]} it is proved as well that for a given --\emph{non dicritical, dicritical}--  foliation pair, satisfying some genericity assumptions, it is possible to assign a unique collection of local models and normalizing transformations. A short formulation of the result is the following.

\begin{theorem}[Existence and uniqueness of local models and normalizing transformations, \cite {[JOV]}]
\label{teo_S:Teo_Exis y Unici de ML y TN_Intro}
For any pair of foliations $\F \in \mathcal{N}_{\p}(\mathbf{h})$ and $\G \in \mathcal{D}_{\q} (\mathcal{I})$ it is possible to assign in a unique way, for suitable normalizing transformations, a collection of --\emph{non dicritical, dicritical}-- local models of foliation pairs. 
\end{theorem}

Moreover, the following theorem based on local models is also proved:

\begin{theorem}[Analytic invariants of foliation pairs, \cite {[JOV]}]
\label{Teo_Invariantes de clasificacion_Intro}
A necessary and sufficient condition for the strict analytical equivalence of non dicritical and dicritical foliation pairs in $\mathcal{N}_{\p}(\mathbf{h}) \times \mathcal{D}_{\q} (\mathcal{I})$  is the coincidence of their corresponding collections of local models.
\end{theorem}

These theorems highlight the relevance of the local models and of the collection normalyzing transformations that accompany them.
 The decomposition of foliation pairs in local models allows the accurate tracking of the elements related to the analytic classification invariants of foliation pairs. In order to achieve this purpose, in this work we will give \emph{Normal Forms} for the normalizing transformations.  As we will show, these normal forms will open the way to obtain clean and informative expressions of the curves of tangencies of foliation pairs that we consider. 

\begin{theorem}[$k$--normal form for normalizing transformations]\label{Teorema_Limpieza de TransNorm_Intro}
For each foliation pair $(\F, \G)$ in $\mathcal{N}_{\p}(\mathbf{h}) \times \mathcal{D}_{\q}(\mathcal{I})$ and any $k\geqslant 1$ there exists a foliation pair strictly analytically equivalent to $(\F, \G)$, whose normalizing transformations, up to its $k$--jet, are normalized in order to have explicit expressions that are clear and informative in terms of the local analytic representatives characterizing the foliation pair $(\F, \G)$.
\end{theorem}

 We are interested in the families of curves of tangencies related to each foliation pair $(\F, \G)$ in $\mathcal{N}_{\p}(\mathbf{h}) \times \mathcal{D}_{\q}(\mathcal{I})$. The analytic types of such curves are analytic invariants of foliation pairs. As we will show, the local models and normalizing transformations of foliation pairs, given in Theorems \ref{teo_S:Teo_Exis y Unici de ML y TN_Intro} and \ref{Teorema_Limpieza de TransNorm_Intro}, allow us to describe finite jets of parametrizations of the corresponding curves of tangencies, and in this way, give the possibility of realizing curves as curves of tangencies of foliation pairs in $\mathcal{N}_{\p} (\mathbf{h}) \times \mathcal{D}_{\q} (\mathcal{I})$. These properties are stated in Theorems \ref{Teorema_Parametrizacion_Polar_pi_qj_Intro} and \ref{teo:realizacion curvas como curvas de tangencias}.

\begin{definition}\label{def: curve of tangencies}
    \emph{The curve of tangencies (polar curve) $\mathcal{P}(\F, \G)$ of a given a foliation pair $(\F, \G)$ in $\mathcal{N}_{\p}(\mathbf{h}) \times \mathcal{D}_{\q}(\mathcal{I})$} is a germ of analytic curve $C$ in $(\C^{2}, 0)$ having $m+n+1$ pairwise transversal smooth branches, whose tangent lines are either $y = p_{i}x$ or $y = q_{j} x$. Namely, $C$ is a germ of analytic curve with branches

\begin{equation}
\label{Ec_Curvas_Parametrizaciones_teo_intro}
\begin{split}
C_{p_{i}} & = \{y = p_{i} x + a_{i} x^{2} + \cdots\} \, , 
\hspace{1.3cm}
1 \leqslant i \leqslant n+1 \, ,
\\
C_{q_{j}} & = \{y =  q_{j} x + a_{j + n + 1} x^{2} + \cdots\} \, ,
\hspace{0.5cm}
1 \leqslant j \leqslant m \, .
\end{split}
\end{equation}
\end{definition}
Together with foliation pairs $(\F, \G)$ in $\mathcal{N}_{\p} (\mathbf{h}) \times \mathcal{D}_{\q} (\mathcal{I})$, we consider $(\tilde{\F}, \tilde{\G})$ the corresponding blow-up foliation pair.  

\begin{theorem}[Parametrization of curves of tangencies at singular and tangency points]
\label{Teorema_Parametrizacion_Polar_pi_qj_Intro}
Let $(\F, \G) \in \mathcal{N}_{\p} (\mathbf{h}) \times \mathcal{D}_{\q} (\mathcal{I})$ be a foliation pair whose normalizing transformations are normalized up to its  $k_{0}$-jet  as in Theorem \ref{Teorema_Limpieza de TransNorm_Intro}. 

Let $\pi_{p_{i}}$ and $\pi_{q_{j}}$ be the respective parametrizations by  $x$ of the curves of tangencies $\CP_{p_{i}} (\tilde{\F}, \tilde{\G})$ and $\CP_{q_{j}} (\tilde{\F}, \tilde{\G})$. Then,
\begin{itemize}
    \item [a)]The coefficients $\,\,\mathbf{c}_{\scriptscriptstyle{p_{i}, k}}\, $ of the power series expansion $\pi_{p_{i}} = p_{i} + \textstyle \sum_{r \geqslant 1}  \mathbf{c}_{\scriptscriptstyle{p_{i}, r}}\,x^{r}$\,  of the parametrizations by $x$ of the curve of tangencies $\CP_{p_{i}} (\tilde{\F}, \tilde{\G})$ depend, for $1 \leqslant k \leqslant k_{0}$, 
on the normalizing transformations (up to the $k-1$ derivative) and on the local analytic representatives at all the singular and tangency points.
\item[b)] The coefficients  $\,\,\mathbf{c}_{\scriptscriptstyle{q_{j}, k}}\,$ of the power series expansion $\pi_{q_{j}} = q_{j} + \textstyle \sum_{r \geqslant 1}  \mathbf{c}_{\scriptscriptstyle{q_{j}, r}}\, x^{r}$ of the parametrizations by $x$ of the curve of tangencies $\CP_{q_{j}} (\tilde{\F}, \tilde{\G})$ depend, for $1 \leqslant k \leqslant k_{0}$,  
on the normalizing transformations (up to the $k-1$ derivative) and on the local analytic representatives at all the singular and tangency points.  

 \end{itemize}
 \end{theorem}

 In the study of foliation pairs and their corresponding tangency curves a central question is to determine when it is possible to realize a germ of analytic curve in $(\C^{2}, 0)$ having $m+n+1$ pairwise transversal smooth branches satisfying \eqref{Ec_Curvas_Parametrizaciones_teo_intro}, as curve of tangencies of a foliation pair in $\mathcal{N}_{\p}(\mathbf{h}) \times \mathcal{D}_{\mathsf{q}}(\mathcal{I})$.

To answer this question we state the following theorem.
\begin{theorem}[Realization of curves as curves of tangencies of foliation pairs]
\label{teo:realizacion curvas como curvas de tangencias} 
Under genericity assumptions, any germ of analytic curve $C$ having $m+n+1$ pairwise transversal smooth branches as in \eqref{Ec_Curvas_Parametrizaciones_teo_intro} is realized as a curve of tangencies of a foliation pair in $\mathcal{N}_{\p} (\mathbf{h}) \times \mathcal{D}_{\q} (\mathcal{I})$. 

The genericity assumptions rely on the collection of Camacho--Sad indices of the non dicritical foliations in $\mathcal{N}_{\p} (\mathbf{h})$, and on the collection of quadratic coefficients of the involutions associated with the dicritical foliations in $\mathcal{D}_{\q} (\mathcal{I})$.
\end{theorem}

\subsection{Structure of the work}

In Section \ref{Preliminares} several definitions and statements about local models and normalizing transformations are introduced. In that section,  the accurate statements of Theorems \ref{teo_S: realización de modelos locales y t_norm}, 
\ref{teo_S:Teo_Exis y Unici de ML y TN_Intro} and \ref{Teo_Invariantes de clasificacion_Intro}, proved in \cite {[JOV]} (see Subsection \ref{Resultados de JOV}), are given as well.
 
Section \ref{Transformaciones normalizantes y curvas polares de parejas de foliaciones} is devoted to the proof of a detailed formulation of Theorem \ref{Teorema_Limpieza de TransNorm_Intro} on the $k$--normal forms for normalizing transformations. In Section \ref{Parametrizaciones de polares}  the local analytic representatives as well Theorem \ref{Teorema_Limpieza de TransNorm_Intro} are used in order to prove a precise formulation of Theorem \ref{Teorema_Parametrizacion_Polar_pi_qj_Intro} about the parametrizations (up to a finite jet) of the branches of the curves of tangencies,
 in terms of the analytic invariants of foliation pairs in $\mathcal{N}_{\p} (\mathbf{h}) \times \mathcal{D}_{\q} (\mathcal{I})$. Finally, Section \ref{Realizacion de curvas} is devoted to prove Theorem \ref{teo:realizacion curvas como curvas de tangencias} on the realization of curves as curves of tangencies.


\section{\textbf{\emph{Definitions and statements about local models and normalizing transformations}}} \label{Preliminares}

\setcounter{figure}{0}
\setcounter{equation}{0}

\subsection{Basic definitions}

 We consider differential equations defined by germs of vector fields with degenerate singularity at the origin. Such germs are described in $(\C^2,0)$ by analytic vector fields
\begin{equation}\label{eq: campo vectorial v}
\mathbf{v}(x,y)= P(x,y)\frac{\partial}{\partial{x}}+ Q(x,y)\frac{\partial}{\partial{x}}\,,
\end{equation}
where $P(x, y) = P_{k}(x, y) + P_{k+1}(x, y) + \cdots$, $Q(x, y) = Q_{k}(x, y) + Q_{k+1}(x, y) + \cdots$, being $P_{j}(x,y)$, $Q_{j}(x,y)$ homogeneous polynomials of order $j$.

We consider the germs of foliations $\mathcal{F}$ induced by germs of vector fields $\mathbf{v}$ as in \eqref{eq: campo vectorial v}, i.e., $\mathcal{F} = \mathcal{F}_{\mathbf{v}}$. Two such germs of foliations \textit{$\mathcal{F}_{\mathbf{v}}$, $\mathcal{F}_{\mathbf{w}}$ are strictly analytically equivalent} if there exists a germ of biholomorphism $H$ in $(\C^{2}, 0)$ tangent to the identity (i.e., $\text{D}_{0} H = I$) sending the leaves of the foliation $\mathcal{F}_{\mathbf{v}}$ to the leaves of the foliation $\mathcal{F}_{\mathbf{w}}$, i.e., $H (\mathcal{F}_{\mathbf{v}}) = \mathcal{F}_{\mathbf{w}}$.

We denote by $\mathbf{v_0}$ \textit{the principal part  of $\mathbf{v}$},  $\mathbf{v_0} := (P_k (x,y),Q_k (x,y))$, and define the $k+1$ homogeneous polynomial $R_{k+1}(x,y) := xQ_k (x,y) -yP_k(x,y)$. We distinguish two cases: either $R_{k+1}(x,y)\not\equiv 0$ or $R_{k+1}(x,y)\equiv 0$. The first case, $R_{k+1}(x,y)\not\equiv 0$, is known as \emph{non dicritical case}, while $R_{k+1}(x,y)\equiv 0$ is known as \emph{dicritical case}. We consider as well the corresponding \emph{non dicritical} and \emph{dicritical} foliations, which we  denote by $\mathcal{F}$ and $\mathcal{G}$, respectively.

Given two foliation pairs -- non dicritical and dicritical -- $(\mathcal{F},\mathcal{G})$, $(\mathcal{\hat{F}},\mathcal{\hat{G}})$, we say that \textit{these pairs are strictly analytically equivalent} if there exists a tangent to the identity germ of biholomorphism $H$ in $(\C^{2}, 0)$  such that it sends simultaneously the leaves of the foliation $\mathcal{F}$ to the leaves of the foliation $\mathcal{\hat{F}}$, and the leaves of $\mathcal{G}$ to the leaves of $\mathcal{\hat{G}}$, i.e., 
\begin{equation*}
    H(\mathcal{F},\mathcal{G}) = (\mathcal{\hat{F}}, \mathcal{\hat{G}})\,.
\end{equation*}

 \subsection{Blow--up}

 We recall that \textit{the blow-up of $(0,0)$ in $\C^2$} is the $2$--dimensional complex manifold $\M$ obtained by gluing two copies of $\C^2$ with coordinates (known as standard charts) $(x,u)$, $u=\frac{y}{x}$ for $x\neq 0$ and $(v,y)$, $v=\frac{x}{y}$ for $y\neq 0$, by means of a map $\phi:(x,u)\longrightarrow (v,y)=(u^{-1},xu)$. The projection $\pi: \M\longrightarrow (\C^2,(0,0))$ given in the first chart by $\pi:(x,u)\longrightarrow(x,ux)$ and in the second chart     
 by $\pi:(v,y)\longrightarrow(yv,y)$  The projection $\pi$ thus defined is called the \textit{standard projection}. The Riemann sphere $\mathcal{L}:=\pi^{-1}(0,0)\equiv\C\mathbb{P}^{1}$ obtained by gluing the regions $\{ 0\}\times\C$ and $\C\times\{ 0\}$ by means of the restriction map $\phi|_{\{ 0\}\times\C^{*}}$ is called \textit{the pasted sphere or exceptional divisor of the blow-up} (for more details see, for example, \cite{[V]} or \cite{[ORV 2]}). Throughout this work we will denote by $\M$ the complex manifold obtained by the described procedure.

 Let $\mathbf{v}$ be a germ of vector field  as in (\ref{eq: campo vectorial v}). Since, outside the exceptional divisor $\mathcal{L}$, the projection $\pi$ is a biholomorphism on its image $\C^{2} \smallsetminus \{0\}$, then it is possible to consider the \textit{lifting} of $\mathbf{v}$ to a vector field $\mathbf{\hat{v}}$ in $\mathbb{M}$, outside the exceptional divisor. By multiplying the field $\mathbf{\hat{v}}$ by a suitable power of $x$ the resulting vector field is holomorphically extended around the exceptional divisor $\mathcal{L}$ by a field of directions $\mathbf{\tilde{v}}$ in $\mathbb{M}$, called \textit{the blow-up of $\mathbf{v}$}. 

Being $\mathcal{F}$ the foliation related to the vector field $\mathbf v$, we denote by $\tilde {\mathcal{F}}$ the foliation corresponding to its blow-up $\mathbf {\tilde v}$ in $\mathbb{M}$; it will be called \textit{the blow-up of the foliation $\mathcal{F}$}.


\subsection{Local analytic representatives --local models-- and normalizing transformations} \label{Resultados de JOV}
In what follows, we fix the collections $\p= (p_{1}, \ldots, p_{n+1})$ and $\q = (q_{1}, \ldots, q_{m})$ of $n+1$ and $m$ mutually distinct complex numbers such that for any  $i=1,\dots,n+1$ and $j= 1,\dots,m$, $p_{i}\neq q_{j}$. We will denote $\hat{p}_{i}$, $\hat{q}_{j}$ the points in the sphere $\mathcal{L}$, which in coordinates $(x, u)$ are $(0, p_{i})$ and $(0, q_{j})$, respectively. The punctured sphere $\mathcal{L} \setminus\{\hat{p}_{1}, \ldots,\hat{p}_{n+1}, \hat{q}_{1},\ldots,\hat{q}_{m}\}$ will be denoted by $\mathcal{L}\smallsetminus (\p, \q)$.

There are three different types of local models together with their corresponding normalizing transformations (see \cite {[JOV]}). We start by defining the local model and the normalizing transformations at the punctured sphere $\mathcal{L}\smallsetminus (\p, \q)$. After that, we will give the corresponding local models and normalizing transformations in neighborhoods of singular and tangency points, respectively. 


\begin{definition}[Local analytic representatives (local models) and normalizing transformations]\label{def: ML_TN}
Let $(\mathcal{F}, \mathcal{G})$ be a foliation pair in $\mathcal{N}_{\mathsf{p}} (\mathbf{h}) \times \mathcal{D}_{\mathsf{q}} (\mathcal{I})$; we denote by  $(\tilde{\mathcal{F}}, \tilde{\mathcal{G}})$ the pair of their corresponding blow--ups. Let  $\mathcal{F}_{\scriptscriptstyle \mathbf{\mu}} \in \mathcal{N}_{\mathsf{p}} (\mathbf{h})$ be a non dicritical foliation whose separatrices are given by straight lines $\{y = p_{i} x\}$, and $\mathcal{G}_{\scriptscriptstyle \mathbf{r}}$ be the radial foliation. 

\begin{itemize} 

    \item[a)] \emph{A local model and a normalizing transformation of the foliation pair 
$(\mathcal{F}, \mathcal{G})$ and its corresponding blow--up $(\tilde{\mathcal{F}}, \tilde{\mathcal{G}})$, with respect to the regular points in $\mathcal{L} \smallsetminus (\mathsf{p}, \mathsf{q})$}, is given by 
\begin{equation*}
    \{(\mathcal{F}_{\scriptscriptstyle \mathbf{\mu}}\, ,\mathcal{G}_{\scriptscriptstyle \mathbf{r}})\,, \mathbf{H}_{\scriptscriptstyle \mathbf{n}}\}\,; \quad\mathbf{H}_{\scriptscriptstyle \mathbf{n}}(\mathcal{F}_{\scriptscriptstyle \mathbf{\mu}}\, ,\mathcal{G}_{\scriptscriptstyle \mathbf{r}})=(\tilde{\mathcal{F}}, \tilde{\mathcal{G})}\,,
\end{equation*}
where  $\mathbf{H}_{\scriptscriptstyle \mathbf{n}} $ is a biholomorphism which is defined between open neighborhoods of the punctured sphere $\mathcal{L}\smallsetminus (\mathsf{p}, \mathsf{q})$ in $\mathbb{M}$, such that it takes (in the mentioned neighborhoods) the blow-up, $\tilde{\mathcal{F}_{\scriptscriptstyle \mathbf{\mu}}}$, of foliation $\mathcal{F}_{\scriptscriptstyle \mathbf{\mu}}$ to the non dicritical foliation $\tilde{\mathcal{F}}$, and the blow--up, $\tilde{\mathcal{G}}_{\scriptscriptstyle \mathbf{r}}$, of the radial foliation $\mathcal{G}_{\scriptscriptstyle \mathbf{r}}$ to the dicritical foliation $\tilde{\mathcal{G}}$. Moreover, with respect to the coordinate chart $(x, u)$, the biholomorphism $\mathbf{H}_{\scriptscriptstyle \mathbf{n}}$ satisfies 
\begin{equation*}
\mathbf{H}_{\scriptscriptstyle \mathbf{n}} (x, u) = \bigl(x + \text{O} (x^{2}), \ u + \text{O}(x) \bigl) \, .
\end{equation*}

 \item[b)] \emph{A local model and a normalizing transformation of the foliation pair $(\mathcal{F}, \mathcal{G})$ and its corresponding blow--up $(\tilde{\mathcal{F}}, \tilde{\mathcal{G}})$, at the singularity $(x, u) = (0, p_{i})$}, $i=1,\ldots,n+1$, is given by
\begin{equation*}
    \{(\mathcal{F}_{p_{i}}^l, \mathcal{G}_{p_{i}}), H_{p_{i}}\}\,; \quad H_{p_{i}}(\mathcal{F}_{p_{i}}^l, \mathcal{G}_{p_{i}})=(\tilde{\mathcal{F}}, \tilde{\mathcal{G})}_{p_{i}}\,,
\end{equation*}
where $H_{p_{i}}$ is a biholomorphism from a neighborhood of $(0, p_{i})$ in $\C^{2}$ to a neighborhood of the singularity $\hat{p}_{i} = (0, p_{i})$ in $\mathbb{M}$, such that $H_{p_{i}}$ takes the linear foliation $\mathcal{F}_{p_{i}}^{l}$ induced by the linear vector field 
\begin{equation*}
    v_{p_{i}}^l=\lambda_{i} \, x \tfrac{\partial}{\partial \, x} + (u-p_{i}) \tfrac{\partial}{\partial \, u}
\end{equation*}
to the blow--up $\tilde{\mathcal{F}}$ of the non 
dicritical foliation $\mathcal{F}$ at a neighborhood of the singular point $(0,p_{i})$.  The biholomorphism $H_{p_{i}}$ takes as well the foliation
\begin{equation*}
   \mathcal{G}_{p_{i}} := (u + s_{i} (x) = \text{cst}) \,,
\end{equation*}
where $s_{i} \colon (\C, 0) \rightarrow (\C, 0)$ is a holomorphic function, to the dicritical foliation $\tilde{\mathcal{G}}$. 

Moreover, the restriction of the biholomorphism $H_{p_{i}}$ on  the line $\{x = 0\}$ is the identity map.
 
  \item[c)] \emph{A local model and a normalizing transformation of the foliation pair $(\mathcal{F}, \mathcal{G})$ and its corresponding blow--up $(\tilde{\mathcal{F}}, \tilde{\mathcal{G}})$, at the tangency point $(x, u) = (0, q_{j})$}, $j=1,\ldots,m$, is given by
\begin{equation*}
    \{((x=cst), \mathcal{G}_{q_{j}}), H_{q_{j}}\}\,; \quad H_{q_{j}}((x=cst), \mathcal{G}_{q_{j}})=(\tilde{\mathcal{F}}, \tilde{\mathcal{G})}_{q_{j}}\,,
\end{equation*}
where $H_{q_{j}}$ is a biholomorphism from a neighborhood of $(0, q_{j})$ in $\C^{2}$ to a neighborhood of the tangency point $\hat{q}_{j} = (0, q_{j})$ in $\mathbb{M}$, such that it takes the trivial foliation $(x=\text{cst})$ to the non dicritical foliation $\tilde{\mathcal{F}}$, and takes the foliation 
\begin{equation*}
 \mathcal{G}_{q_{j}} := (z_{j} (x) + g_{j} (u) = \text{cst})\,,   
\end{equation*}
where $z_{j} \colon (\C, 0) \rightarrow (\C, 0)$, $z^{(1)}_{j} (0) \neq 0$, and $g_{j} (u) := (q_{j} - u) (I_{j} (u) - q_{j})$, to the dicritical foliation $\tilde{\mathcal{G}}$. 

Moreover, the restriction of the biholomorphism $H_{q_{j}}$ on  the line $\{x = 0\}$ is the identity map.
 
  \end{itemize}
\end{definition}

In the work \cite {[JOV]}  it is proved that each collection of suitable local models can be realized by a foliation pair (\emph{non dicritical, dicritical}). Namely, the following precise formulation of Theorem \ref{teo_S: realización de modelos locales y t_norm} is proved.

 
\begin{theorem}[Realization of local models, \cite {[JOV]}]\label{teo:Construccion de PdF a partir de PL}
Let $\mathcal{F}_{p_{i}}^l$ be the linear foliation induced by the linear vector field $v_{p_{i}}^{l}(x,u)= \lambda_{i} \, x \tfrac{\partial}{\partial \, x} + (u-p_{i}) \tfrac{\partial}{\partial \, u}$ in the open domain $(\C^{2}, D_{p_{i}})$. 
Let $s_{i}, z_{j} \colon (\C, 0) \rightarrow (\C, 0)$, $z_{j}^{(1)} (0) \neq 0$, be the holomorphic transformations defining foliation $\G_{p_{i}} = (s_{i} (x) + u = \text{cst})$ in $(\C^{2}, D_{p_{i}})$, and foliation $\G_{q_{j}} = (z_{j} (x) + g_{j} (u) = \text{cst})$ in $(\C^{2}, D_{q_{j}})$. 

For small enough disks $D_{p_{i}}, D_{q_{j}} \subseteq \C^{2}$ there exist foliations 
\begin{equation*}
    \F \in \mathcal{N}_{\p} (\mathbf{h})\quad\text{and}\quad\G \in \mathcal{D}_{\q} ( \mathcal{I})
\end{equation*}
fulfilling the following properties (see figure \ref{CuTan_Figura_ML y TN respec a F_ast_L}). 
\begin{itemize}
\item[1.] There exists a biholomorphism $\mathbf{H}_{\scriptscriptstyle \mathbf{n}} \colon \bigl(\M, \mathcal{L} \smallsetminus (\p, \q)\bigl) \rightarrow \bigl(\M, \mathcal{L} \smallsetminus (\p, \q)\bigl)$ taking the foliation pair  $(\tilde{\F}_{\scriptscriptstyle \mu}, \tilde{\mathcal{G}_{\scriptscriptstyle \mathbf{r}}})$ to the foliation pair $(\tilde{\F}, \tilde{\G})$, and satisfying 
\begin{equation*}
\begin{split} 
 \mathbf{H}_{\scriptscriptstyle \mathbf{n}} \bigl\vert_{\mathcal{L}}  = \text{id}_{\mathcal{L}} \, ,
\hspace{0.5cm} 
\mathbf{H}_{\scriptscriptstyle \mathbf{n}} (x, u) = \bigl(x + \text{O} (x^{2}), h_{\scriptscriptstyle \mathbf{n}} (x, u) \bigl) \, .
\end{split}
\end{equation*}

\item[2.] There exists a biholomorphism $H_{p_{i}} \colon (\C^{2}, D_{p_{i}}) \rightarrow (\M, D_{p_{i}})$ taking the foliation pair $(\mathcal{F}_{p_{i}}^l, \G_{p_{i}})$ to the foliation pair $(\tilde{\F}, \tilde{\G})$, and satisfying 
\begin{equation*}
\begin{split} 
H_{p_{i}} \bigl\vert_{\{x = 0\}}  =  \text{id}_{\{x = 0\}} \, ,
\hspace{0.5cm} 
H_{p_{i}} \circ \Psi^{-1}_{i} = \mathbf{H}_{\scriptscriptstyle \mathbf{n}} \circ \xi_{i}   \, .
\end{split}
\end{equation*} 
where $\xi_{i}$, $\Psi_{i}$ are the biholomorphisms in (\ref{Ec_Series de Psi y Xi}). 

\item[3.] There exists a biholomorphism $H_{q_{j}} \colon (\C^{2}, D_{q_{j}}) \rightarrow (\M, D_{q_{j}})$ taking the foliation pair $((x = \text{cst}), \G_{q_{j}})$ to the foliation pair $(\tilde{\F}, \tilde{\G})$, and satisfying 
\begin{equation*}
\begin{split} 
H_{q_{j}} \bigl\vert_{\{x = 0\}}  =  \text{id}_{\{x = 0\}} \, ,
\hspace{0.5cm} 
H_{q_{j}} \circ \Phi^{-1}_{j}  = \mathbf{H}_{\scriptscriptstyle \mathbf{n}} \circ \zeta_{j}  \, .
\end{split}
\end{equation*}
where $\zeta_{j}$, $\Phi_{j}$ are the biholomorphisms in   (\ref{Ec_Series de Phi y Zeta}).  
\end{itemize}
\end{theorem}

\begin{figure}[ht]
\begin{center}
\includegraphics[scale=0.5]{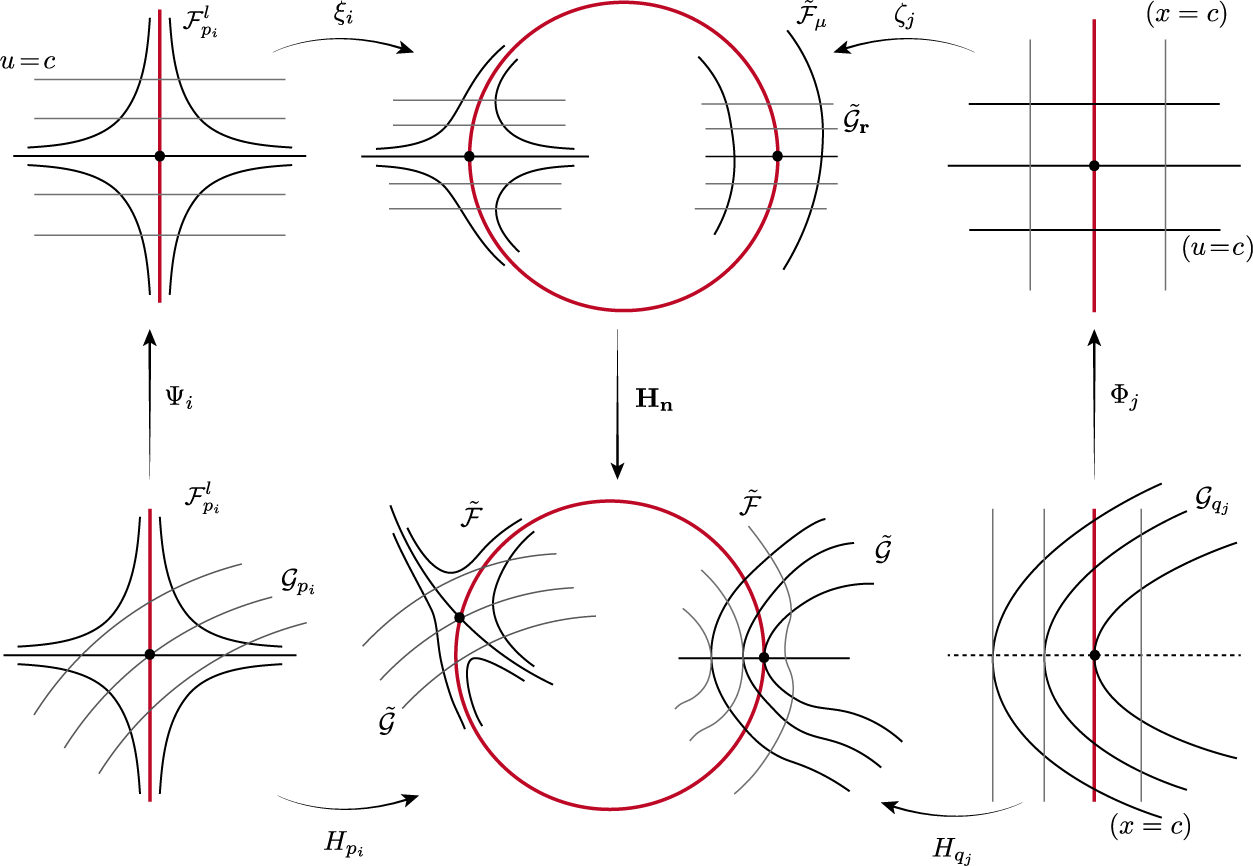}
\caption{\small{Local models and normalizing transformations with respect to $(\F_{\scriptscriptstyle \mu}, \G_{\scriptscriptstyle \mathbf{r}})$. Biholomorphism $\Psi_{i}$ is not defined at the singular point $(0,p_{i})$, biholomorphism $\Phi_{j}$ is not defined at the tangency point $(0,q_{j})$; thus, biholomorphism $ \mathbf{H}_{\scriptscriptstyle \mathbf{n}}$ is not defined neither at singular points nor at tangency points.}}
\label{CuTan_Figura_ML y TN respec a F_ast_L}
\end{center}
\end{figure}

\begin{definition}[Local analytic representatives --local models-- and normalizing transformations with respect to $(\F_{\scriptscriptstyle \mu}\, ,\mathcal{G}_{\scriptscriptstyle \mathbf{r}})$]
\label{Def_ML-TN_Fmu-Gr}
Let $(\mathcal{F}, \mathcal{G})$ be a foliation pair in $\mathcal{N}_{\p} (\mathbf{h}) \times \mathcal{D}_{\q} (\mathcal{I})$. Let us consider local models and normalizing transformatios of the pair $(\mathcal{F}, \mathcal{G})$ in the sense of Definition \ref{def: ML_TN}:
\begin{equation*}
(\F_{\scriptscriptstyle \mu}\, ,\mathcal{G}_{\scriptscriptstyle \mathbf{r}}) \, , \mathbf{H}_{\scriptscriptstyle \mathbf{n}} \ , 
\hspace{0.5cm}
\left\{(\mathcal{F}_{p_{i}}^l, \G_{p_{i}}), H_{p_{i}}\right\} \ ,  
\hspace{0.5cm}
\left\{((x=ct), \G_{q_{j}}), H_{q_{j}}\right\} \ .
\end{equation*}
If, additionally, the normalizing transformations at the singular points and tangency points satisfy the \textit{factorization equations} (see Definition \ref{def: factorization_norm_trans}): 
\begin{equation*}
        H_{p_{i}} = \mathbf{H}_{\scriptscriptstyle \mathbf{n}} \circ \xi_{i} \circ \Psi_{i} \ , 
        \hspace{0.8cm}
         H_{q_{j}} = \mathbf{H}_{\scriptscriptstyle \mathbf{n}} \circ \zeta_{j} \circ \Phi_{j} \ ,
\end{equation*}
we will say that they are \textit{local analytic representatives (local models) and normalizing transformations with respect to $(\F_{\scriptscriptstyle \mu}\, ,\mathcal{G}_{\scriptscriptstyle \mathbf{r}})$}. 
\end{definition}

 The following theorem gives a detailed formulation of Theorem \ref{teo_S:Teo_Exis y Unici de ML y TN_Intro} given in Section \ref{sec: Introduction}.

\begin{theorem}[Existence and uniqueness of local models and normalizing transformations, \cite {[JOV]}]
\label{CuTan_Teo: Existencia_Unicidad de ML y TN}
Let $\F_{\scriptscriptstyle \mu} \in \mathcal{N}_{\p}(\mathbf{h})$ be a non dicritical foliation having the straight lines $\{y = p_{i} x\}$, $1 \leqslant i \leqslant n+1$, as a complete set of separatrices, and let $\G_{\scriptscriptstyle \mathbf{r}}$ be the radial foliation. 

For any foliation pair $(\F,\G)$ where $\F \in \mathcal{N}_{\p}(\mathbf{h})$ and $\G \in \mathcal{D}_{\q} (\mathcal{I})$, there exist unique local models and normalizing transformations with respect to  $(\F_{\scriptscriptstyle \mu}, \G_{\scriptscriptstyle \mathbf{r}})$.
\end{theorem}

Moreover, as it was mentioned in the introduction to this work, the local models are crucial in the analytic classification of foliation pairs in $\mathcal{N}_{\p}(\mathbf{h}) \times \mathcal{D}_{\q} (\mathcal{I})$ (see figure \ref{Figure_Equivan_L}).

\begin{theorem}[Analytic invariants of foliation pairs, \cite {[JOV]}]
\label{Teo_Invariantes de clasificacion}
A necessary and sufficient condition for the strict analytical equivalence of non dicritical and dicritical foliation pairs in $\mathcal{N}_{\p}(\mathbf{h}) \times \mathcal{D}_{\q} (\mathcal{I})$  is the coincidence of their corresponding  collections of local models.
\end{theorem}

\begin{figure}[ht]
\begin{center}
\includegraphics[scale=0.55]{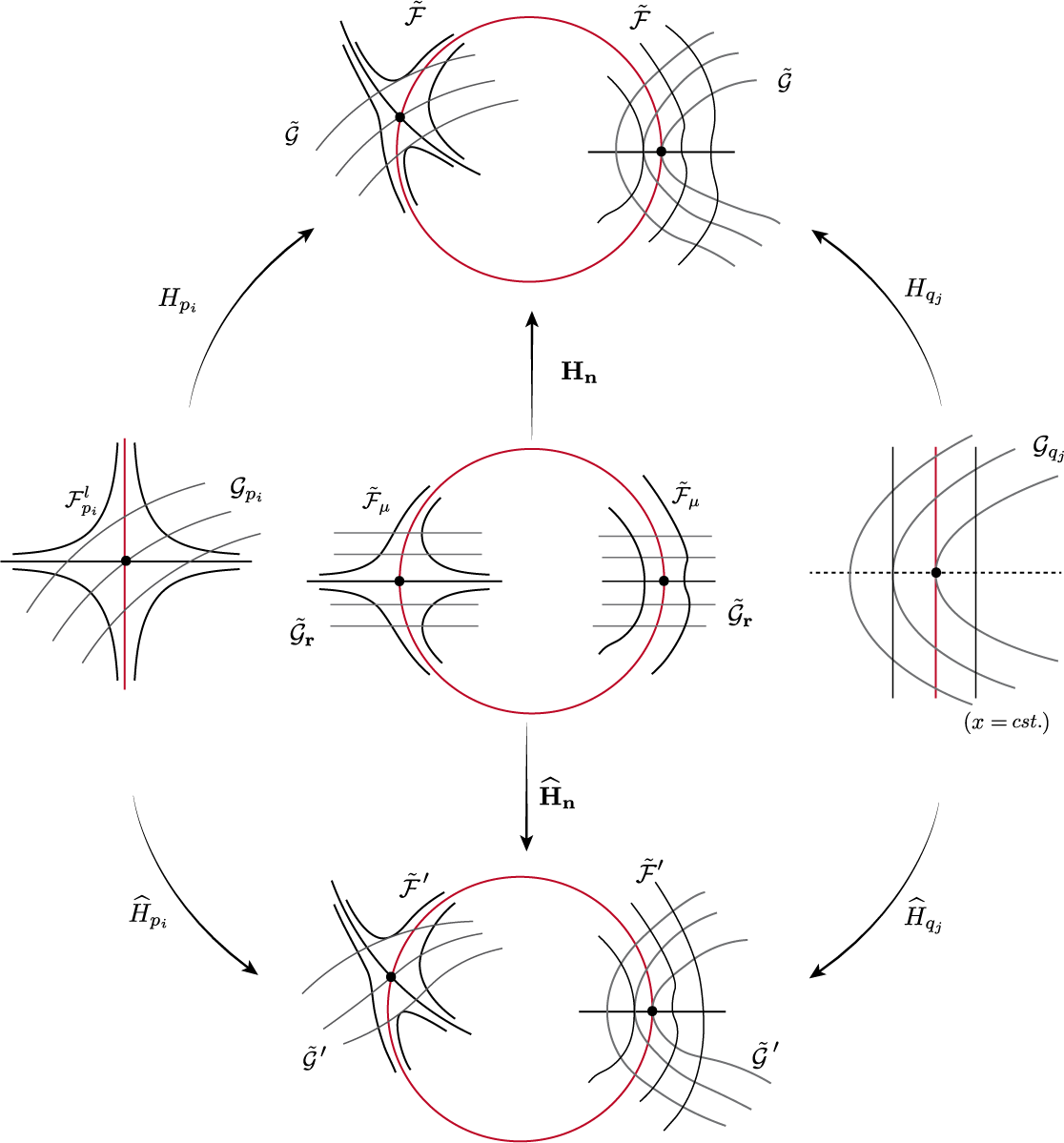}
\caption{\small{Foliation pairs $(\mathcal{F}, \mathcal{G})$ and $(\mathcal{F}', \mathcal{G}')$ are strictly analytically equivalent if and only if they have the same local models with respect to $(\F_{\scriptscriptstyle \mu}, \G_{\scriptscriptstyle \mathbf{r}})$.}}
\label{Figure_Equivan_L}
\end{center}
\end{figure}


As we said before, since we are working with foliation pairs, in what follows we have to deal with different expressions that have to be simplified simultaneously.
 We begin by giving normal forms for the normalizing transformations satisfying that, for each $k\geq 1$, the normalizing transformations  $\mathbf{H}_{\scriptscriptstyle \mathbf{n}} $, $H_{p_{i}}$, $H_{q_{j}}$ 
 with respect to the foliation pair $(\F_{\scriptscriptstyle \mu}, \G_{\scriptscriptstyle \mathbf{r}})$ can be simultaneously simplified, up to order $k$, by a suitable tangent to the identity biholomorphism $\mathcal{H}^{k} \colon (\C^{2}, 0) \rightarrow (\C^{2}, 0)$.

As it was said in the introduction to this work, the normalizing transformations as well as the local models, codify analytic information of the corresponding pair of foliations. For this reason, it is important to obtain expressions of the normalizing transformations that are as informative and simple as possible. The following section is devoted to state and prove a precise formulation of Theorem \ref{Teorema_Limpieza de TransNorm_Intro}, which gives the possibility of having such informative and clean normalizing transformations.

We stress that for 
any foliation pair $(\F, \G)$ in $\mathcal{N}_{\p}(\mathbf{h}) \times \mathcal{D}_{\q}(\mathcal{I})$ 
and any $k\geqslant 1$, there exists a foliation pair strictly analytically equivalent to $(\F, \G)$, whose normalizing transformations up to order $k$ have explicit expressions that are clear and informative in terms of the analytic invariants characterizing the foliation pair $(\F, \G)$.

As we emphasized at the beginning of the introduction, we are interested in the family of curves of tangencies arising from each foliation pair $(\F, \G)$ in $\mathcal{N}_{\p}(\mathbf{h}) \times \mathcal{D}_{\q}(\mathcal{I})$. These curves of tangencies (polar or Jacobian curves) are crucial in the geometric interpretation of the parametric Thom's analytic invariants given in \cite{[ORV 1]} and  \cite{[ORV 2]}. The local models and normalizing transformations of a foliation pair, stated in  Theorem \ref{CuTan_Teo: Existencia_Unicidad de ML y TN},   give the possibility of describing   finite jets of parametrizations of the corresponding curves of tangencies (see Theorem \ref{Teorema_Parametrizacion_Polar_pi_qj_Intro}); by choosing jets of order sufficiently high, this theorem on parametrizations of curves of tangencies will allow us to prove Theorem \ref{teo:realizacion curvas como curvas de tangencias} on realization of analytic curves with pairwise transversal smooth branches as curves of tangencies of foliation pairs.

\label{Pageref_Segunda parte}


\section{\textbf{\emph{Normal forms of normalizing transformations}}}\label{Transformaciones normalizantes y curvas polares de parejas de foliaciones}
\setcounter{equation}{0}

The theory of \emph{Normal Forms} constitutes a powerful treatment in different areas of mathematics, which allows us to work with objects (transformations or vector fields) from a simpler perspective.
Namely, they allow to recover the analytical information (differentiable, formal or topological) of the objects, from their simplest expressions.

In $1987$ Poincaré proved, in his Phd thesis \cite{[P]}, inspiring and useful theorems for local non-resonant vector fields satisfying some assumptions of genericity that are strictly related to the position in the complex plane of the eigenvalues of the linearization matrix at the singular point.  Later on, results of interest were obtained from H.Dulac, C.Siegel, V.I.Arnold, J.Martinet, J-P.Ramis, F.Takens, J.Écalle, S.M.Voronin, Yu.S.Ilyashenko, S.M.Yakovenko, F.Loray, L.Stolovich, E.Paul, and others. In all these results, the linear part of the vector fields at the singular point is not identically zero, although it might not satisfy genericity assumptions. 
 
In the case of diffeomorphisms, there is a certain parallelism in the theory. Thus, there exists analogous results providing key information from the simplest possible expressions.

 In our case, since we are working with foliation pairs, we have to deal with different expressions that have to be simplified simultaneously. 

By definition, the normalizing transformations of a pair of foliations carry the local models of the pair in the pair itself (definition \ref{def: ML_TN}). Therefore, the normalizing transformations carry the curves of tangencies of the local models in the curves of tangencies of the original pair. In this section, except for tangent to the identity coordinate changes, 
\emph{$k$--normal forms} of the normalizing transformations of  foliation pairs in $\mathcal{N}_{\p} (\mathbf{h}) \times \mathcal{D}_{\q} (\mathcal{I})$ are obtained (see Theorem \ref{Teorema_Limpieza de TransNorm_Intro}). As we will show in section \ref{Parametrizaciones de polares} these \emph{$k$--normal forms} allow to express parametrizations of the curves of tangencies of the pair in terms of the analytic data from their local models.

 In subsection \ref{Series de TN} we describe, by means of the \emph{factorization equations} introduced in Definition \ref{def: factorization_norm_trans}, the relations among the power series defining the normalizing transformations. These relations are used in subsection  \ref{Limpieza de TransNorm} in order to achieve, under tangent to the identity coordinate changes, that finite jets of the normalizing transformations are expressed in terms of analytic data of the corresponding local models. In section \ref{Parametrizaciones de polares}   we show how these expressions are used to describe finite jets of parametrizations of the curves of tangencies in terms of the analytic data of the local models corresponding to the original foliation pair. This is done up to changes of coordinates tangent to the identity. It is worth mentioning that, by choosing jets of sufficiently high order, this description of the parametrizations will allow us to study the analytical type of the curves of tangencies.

  Throughout this section we consider foliation pairs $(\F, \G)$ in $\mathcal{N}_{\p} (\mathbf{h}) \times \mathcal{D}_{\q} (\mathcal{I})$ for fixed holonomy $\mathbf{h}$ and fixed collection of involutions $\mathcal{I}$. Moreover, we fix a non dicritical foliation $\F_{\scriptscriptstyle \mu} \in \mathcal{N}_{\p} (\mathbf{h})$ having exclusively the straight lines $\{y = p_{i} x\}$ as invariant branches (separatrices). Let $\tilde{\F}_{\scriptscriptstyle \mu}$ be the blow-up foliation induced by $\F_{\scriptscriptstyle \mu}$. We denote by   $\lambda_{i}$ the corresponding Camacho--Sad indexes  at the singular points $(x, u) = (0, p_{i})$ of foliation $\tilde{\F}_{\scriptscriptstyle \mu}$. As in Definition \ref{def: ML_TN} and Theorem \ref{teo_S: realización de modelos locales y t_norm}, the holomorphic function $g_{j}$ is defined as $g_{j}(u):=(q_{j}-u)(I_{j}(u)-q_{j})$, where  $I_{j} \colon (\C, q_{j}) \rightarrow (\C, q_{j})$ is the $j$-th involution of the collection $\mathcal{I} = (I_{1}, \ldots, I_{m})$.

\subsection{Power series of normalizing transformations of foliation pairs}
\label{Series de TN}

This section is devoted to stablish the expressions of the \emph{factorization equations} of the normalizing transformations $H_{p_{i}}$, $H_{q_{j}}$, defined by the corresponding factorization equations (see Definition \ref{def: factorization_norm_trans})
at the respective points $p_i$, $i=1,\dots, n+1$, and $q_j$, $j=1,\dots, m$,  in terms of their power series. The corresponding statements are given in Lemma \ref{lem: series TN por ecFac en qj} and Lemma \ref{lem: series TN por ecFac en pi}.

Let $\mathbf{H}_{\scriptscriptstyle \mathbf{n}} = (A_{\scriptscriptstyle \mathbf{n}}, B_{\scriptscriptstyle \mathbf{n}})$, $H_{p_{i}} = (A_{p_{i}}, B_{p_{i}})$, $H_{q_{j}} = (A_{q_{j}}, B_{q_{j}})$ be the normalizing transformations of the foliation pair $(\F, \G)$ in $\mathcal{N}_{\p} (\mathbf{h}) \times \mathcal{D}_{\q} (\mathcal{I})$ with respect to the foliation pair $(\F_{\scriptscriptstyle \mu}, \G_{\scriptscriptstyle \mathbf{r}})$ (see Definition \ref{def: ML_TN}). Let us consider the corresponding power series:
\begin{itemize}
    \item [a)] For $\mathbf{H}_{\scriptscriptstyle \mathbf{n}} = (A_{\scriptscriptstyle \mathbf{n}}, B_{\scriptscriptstyle \mathbf{n}})$:

\begin{equation}
\label{Ec_Series de transformaciones normalizantes_n}
A_{\scriptscriptstyle \mathbf{n}}   = \sum_{k \geqslant 1}  A_{\scriptscriptstyle \mathbf{n} ,\scriptscriptstyle {k}} (u) \, x^{k} \, ,
\hspace{1cm}
B_{\scriptscriptstyle \mathbf{n}} = u + \sum_{k \geqslant 1}  B_{\scriptscriptstyle \mathbf{n} ,\scriptscriptstyle {k}} (u) \, x^{k} \, ,
\end{equation}
Defined at the regular points in the neighborhood  $(\M,\mathcal{L}\smallsetminus (\p, \q))$.
Recall that, by definition, the coefficient $A_{\scriptscriptstyle \mathbf{n}\, , 1}$ of $\mathbf{H}_{\scriptscriptstyle \mathbf{n}} $ is equal to $1$.
\vspace{0.3cm}
 \item [b)] For $H_{p_{i}} = (A_{p_{i}}, B_{p_{i}})$:
\begin{equation}
\label{Ec_Series de transformaciones normalizantes_pi}
A_{p_{i}}   = \sum_{k \geqslant 1}  A_{p_{i} ,\scriptscriptstyle {k}} (u) \, x^{k} \, ,
\hspace{1cm}
B_{p_{i}} = u + \sum_{k \geqslant 1}  B_{p_{i} ,\scriptscriptstyle {k}} (u) \, x^{k} \, ,
\end{equation}
Defined in $(\M, D_{p_{i}})$ for the singular points $(0, p_{i})$, $i=1,\dots,n$.
\vspace{0.3cm}
\item[c)] For $H_{q_{j}} = (A_{q_{j}}, B_{q_{j}})$:
\begin{equation}
\label{Ec_Series de transformaciones normalizantes_qj}
A_{q_{j}}   = \sum_{k \geqslant 1}  A_{q_{j} ,\scriptscriptstyle {k}} (u) \, x^{k} \, ,
\hspace{1cm}
B_{q_{j}} = u + \sum_{k \geqslant 1}  B_{q_{j} ,\scriptscriptstyle {k}} (u) \, x^{k} \, .
\end{equation}
\vspace{0.3cm}
Defined in $(\M, D_{q_{j}})$ for the tangency points 
$(0, q_{j})$ $j=1,\dots,m$.
\end{itemize}

\begin{remark} In what follows, for sake of simplicity, we will frequently use the notation 

\begin{equation}
\label{Ec_Series de transformaciones normalizantes}
A_{[\centerdot]}   = \sum_{k \geqslant 1}  A_{[\centerdot] ,\scriptscriptstyle {k}} (u) \, x^{k} \, ,
\hspace{1cm}
B_{[\centerdot]} = u + \sum_{k \geqslant 1}  B_{[\centerdot] ,\scriptscriptstyle {k}} (u) \, x^{k} \, ,
\end{equation}
where $[\centerdot]$ represents $\scriptscriptstyle \mathbf{n}$, $p_{i}$, or $q_{j}$,
be the respective power series defined in the corresponding neighborhoods $(\M,\mathcal{L}\smallsetminus (\p, \q))$ at the regular points,  $(\M, D_{p_{i}})$ at the singular points $(0, p_{i})$, $i=1,\dots,n$, and $(\M, D_{q_{j}})$ at the tangency points 
$(0, q_{j})$ $j=1,\dots,m$, respectively.
By definition, the coefficient $A_{\scriptscriptstyle \mathbf{n}\, , 1}$ of $\mathbf{H}_{\scriptscriptstyle \mathbf{n}} $ is equal to $1$.

\end{remark}

\begin{definition}\label{def: factorization_norm_trans}
    We say that the equalities 
\begin{equation}
\label{eq:ecuaciones de factorizacion}
H_{p_{i}} = \mathbf{H}_{\scriptscriptstyle \mathbf{n}} \circ \xi_{i} \circ \Psi_{i}
\hspace{0.8cm}
\text{and}
\hspace{0.8cm}
H_{q_{j}} = \mathbf{H}_{\scriptscriptstyle \mathbf{n}} \circ \zeta_{j} \circ \Phi_{j}\,,
\end{equation}
\noindent are the \emph{factorization equations of the normalizing transformations $\mathbf{H}_{\scriptscriptstyle \mathbf{n}}$, $H_{p_{i}}$ and $ H_{q_{j}}$} (see figure \ref{Figura_Construccion de PdF a partir de PL}).
\end{definition}

 \begin{figure}[ht]
\begin{center}
\includegraphics[scale=0.55]{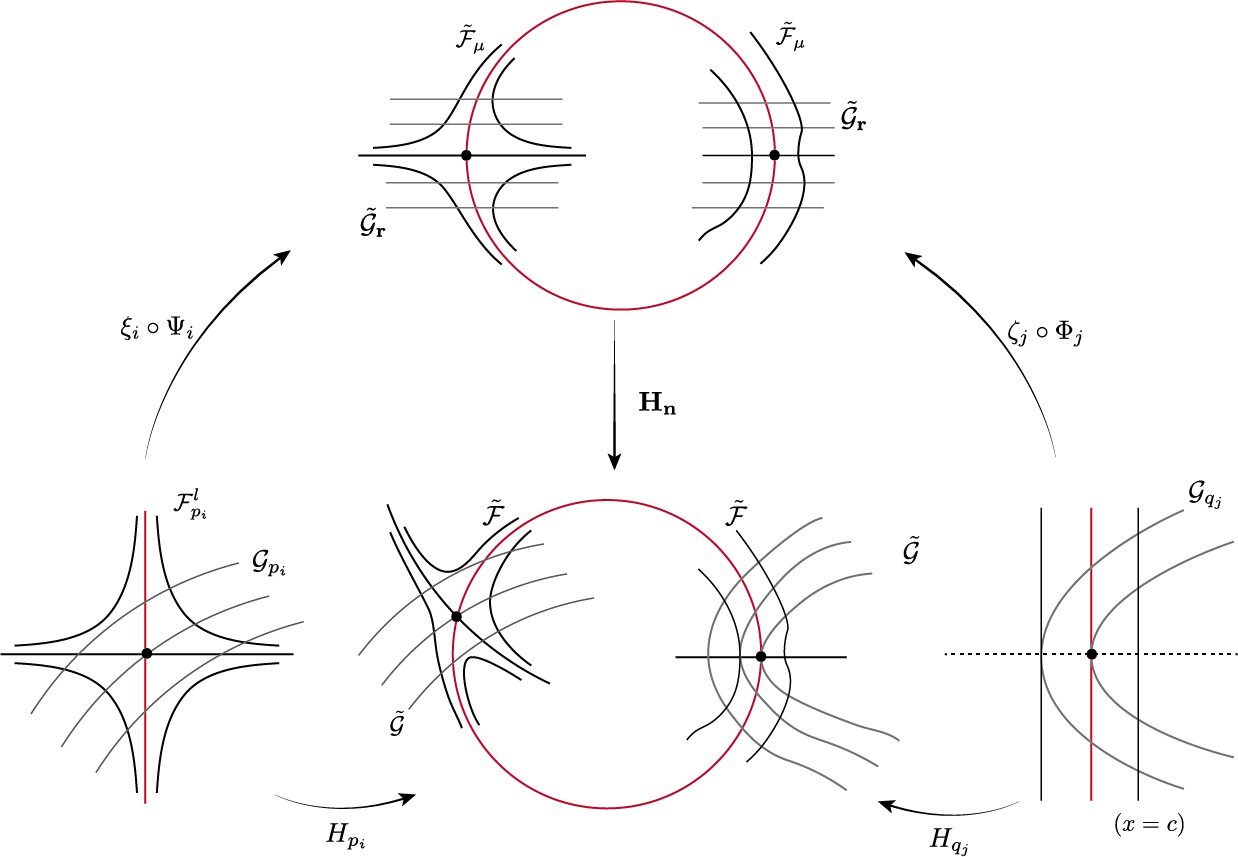}
\caption{\small{The factorization equations $H_{p_{i}} = \mathbf{H}_{\scriptscriptstyle \mathbf{n}} \circ \xi_{i} \circ \Psi_{i}$ and $H_{q_{j}} = \mathbf{H}_{\scriptscriptstyle \mathbf{n}} \circ \zeta_{j} \circ \Phi_{j}\,,$ give a decomposition of $(\tilde{\F}, \tilde{\G})$ by local analytic representatives.}}
\label{Figura_Construccion de PdF a partir de PL}
\end{center}
\end{figure}


In \eqref{eq:ecuaciones de factorizacion} the biholomorphism $\Psi_{i}$ is defined in a neighborhood in $\C^{2}$ of the annulus $D_{p_{i}} \smallsetminus \{(0, p_{i})\}$, and the biholomorphism $\xi_{i} \colon (\C^{2}, D_{p_{i}}) \rightarrow (\M, D_{p_{i}})$  satisfies (see Proposition 3.1 in \cite{[JOV]}) the equalities  
\begin{equation}
\label{Ec_Series de Psi y Xi}
\begin{split}
\Psi_{i} 
= \bigl(\hat{\psi}_{i}, \, u + s_{i}(x)\bigl)
\, , 
\hspace{0.4cm} 
\hat{\psi}_{i} 
& = \sum_{k \geqslant 1}  \hat{\psi}_{\scriptscriptstyle{i, k}} (u) \, x^{k}\,,
\hspace{0.3cm} 
\text{where}
\hspace{0.3cm}
\hat{\psi}_{i}(x,u)=x \bigl(1 + \tfrac{s_{i}(x)}{u - p_{i}}\bigl)^{\lambda_{i}} 
\, , \\
\xi_{i}  = (\varepsilon_{i}, u) \, ,
\hspace{0.5cm}
\varepsilon_{i} & = \sum_{k \geqslant 1}  \varepsilon_{\scriptscriptstyle{i, k}} (u) \, x^{k} \, ,
\hspace{0.3cm}
 \varepsilon_{\scriptscriptstyle{i, 1}}  (p_{i}) = 1 \,.
\end{split}
\end{equation}
The holomorphic map $s_{i} \colon (\C, 0) \rightarrow (\C, 0)$ depends on the foliation pair $(\F, \G)$ in $(0,p_{i})$ with respect to the foliation pair $(\F_{\scriptscriptstyle \mu}, \G_{\scriptscriptstyle \mathbf{r}})$. The power series of $s_{i}$ is written as
\begin{equation}\label{ec:si}
    s_{i}=\sum_{r \geqslant 1} s_{\scriptscriptstyle{i,r}}  \, x^{r}, \quad s_{\scriptscriptstyle{i,r}}  \in \C.
\end{equation}
The definition of the biholomorphism $\xi_{i}$ only depends on the non dicritical foliation $\F_{\scriptscriptstyle \mu}$, and not on the foliation pair $(\F, \G)$.

At the same time, the biholomorphism $\Phi_{j}$ defined in a neighborhood of the annulus $D_{q_{j}} \smallsetminus \{(0, q_{j})\}$ in $\C^{2}$, and the biholomorphism $\zeta_{j} \colon (\C^{2}, D_{q_{j}}) \rightarrow (\M, D_{q_{j}})$ satisfies (see Proposition 3.3 in \cite{[JOV]}) the equalities 
\begin{equation}
\label{Ec_Series de Phi y Zeta}
\begin{split}
& \Phi_{j} 
 = \bigl(x, \phi_{j}) \, ,
\hspace{0.2cm} 
\phi_{j}  = u  + \sum_{k \geqslant 1}  \phi_{\scriptscriptstyle{j, k}} (u) \, x^{k} \, ,
\hspace{0.1cm} 
\text{where} 
\hspace{0.1cm} 
g_{j} \circ \phi_{j} (x, u) = g_{j}(u) + z_{j} (x) \, ,
\\
& \zeta_{j} = (\varsigma_{j}, u) \, ,
\hspace{0.3cm}
\varsigma_{j}  = 
\sum_{k \geqslant 1}  \varsigma_{\scriptscriptstyle{j, k}} (u) \, x^{k} \, ,
\hspace{0.3cm}
 \varsigma_{\scriptscriptstyle{j, 1}} (q_{j}) = 1 \, ,
\hspace{0.2cm}
 \varsigma_{\scriptscriptstyle{j, k+1}}  (q_{j}) = 0 \, , 
\hspace{0.2cm}
\text{if}
\hspace{0.2cm}
k > 1 
\end{split}
\end{equation}
 The biholomorphism $z_{j} \colon (\C, 0) \rightarrow (\C, 0)$  is present in  the analytic model for the foliation pair
  $(\F, \G)$ at the point $(0,q_{j})$ with respect to the foliation pair $(\F_{\scriptscriptstyle \mu}, \G_{\scriptscriptstyle \mathbf{r}})$.  The power series of $z_{j}$ is written as
  \begin{equation}\label{ec:zj}
     z_{j}= \sum_{r \geqslant 1} z_{\scriptscriptstyle{j, r}}  \, x^{r},\quad z_{\scriptscriptstyle{j, r}}  \in \C
  \end{equation}
  Moreover, the definition of the biholomorphism $\zeta_{j}$ only depends on the non dicritical foliation $\F_{\scriptscriptstyle \mu}$, and not on the foliation pair $(\F, \G)$.


The following lemmas relate the expressions (\ref{Ec_Series de Psi y Xi}), (\ref{ec:si}), (\ref{Ec_Series de Phi y Zeta}), (\ref{ec:zj}) of the \emph{factorizing equations} in (\ref{eq:ecuaciones de factorizacion}) in terms of the power series defining the normalizing transformations of  the foliation pair $(\F, \G)$ with respect to the foliation pair $(\F_{\scriptscriptstyle \mu}, \G_{\scriptscriptstyle \mathbf{r}})$.

\begin{lemma}\label{lem: series TN por ecFac en pi}
Let $H_{p_{i}} = (A_{p_{i}}, B_{p_{i}})$  be the normalizing transformations
at the points $p_i$, $i=1,\dots, n+1$, defined by the corresponding factorization equations 
 $H_{p_{i}} = \mathbf{H}_{\scriptscriptstyle \mathbf{n}} \circ \xi_{i} \circ \Psi_{i}$. The power series defining the biholomorphisms  $ A_{p_{i}}$ and $B_{p_{i}}$  in (\ref{Ec_Series de transformaciones normalizantes}) satisfy the relations 
 
\begin{itemize}
    \item[a)] For $A_{p_{i}} (x, u)   = \sum_{k \geqslant 1}  A_{\scriptscriptstyle{p_{i}, {k}}} (u) \, x^{k} \, ,$

\begin{equation}\label{eq: Api,1}
     A_{\scriptscriptstyle{p_{i}, {1}}}  
=   \varepsilon_{\scriptscriptstyle{i, 1}}   \, ,
 \end{equation}
     
 \begin{equation}\label{eq: Api,2}
 A_{\scriptscriptstyle{p_{i}, k+1}} 
\, - \, ( \varepsilon_{\scriptscriptstyle{i, 1}} )^{k+1}
A_{\scriptscriptstyle \mathbf{n}, \scriptscriptstyle{k+1} }
 =  \varepsilon_{\scriptscriptstyle{i, k+1}}  
\, + \, \mathbf{A}_{\mathbf{\p}_{i},\scriptscriptstyle{k+1}}\,,
\end{equation}
where
\begin{equation}\label{eq: Api,3}
  \mathbf{A}_{\mathbf{\p}_{i},\scriptscriptstyle{k+1}}=  \mathbf{A}_{\mathbf{\p},\scriptscriptstyle{k+1}}[\lambda_{i}, s_{\scriptscriptstyle{i,r}} ;  \varepsilon_{\scriptscriptstyle{i, r}}, A_{\scriptscriptstyle {\mathbf{n}, r}}]^{k}_{r = 1} \, , \,\, k\geq 1\,,
\end{equation}
is holomorphic in the annulus $D_{p_{i}} \smallsetminus \{(0, p_{i})\}$ and it is defined by evaluation of the transformation  $\mathbf{A}_{\mathbf{\p},\scriptscriptstyle{k+1}}$ in the complex numbers $\lambda_{i}$, $s_{\scriptscriptstyle{i,r}} $ (information of the local model of $(\F, \G)$ at the point  $(0,p_{i})$), and in the functions, $ \varepsilon_{\scriptscriptstyle{i, r}}$, $A_{\scriptscriptstyle {\mathbf{n}, r}}$, $1 \leqslant r \leqslant k$ (coefficients of the series $\varepsilon_{i}$ and $A_{\scriptscriptstyle \mathbf{n}}$, respectively).
\item[b)] For $B_{p_{i}} (x, u)   = \sum_{k \geqslant 1}  B_{\scriptscriptstyle{p_{i}, {k}}} (u) \, x^{k} \, ,$
\begin{equation} \label{eq: Bpi,1}
     B_{\scriptscriptstyle{p_{i}, {1}}} 
\, - \,  \varepsilon_{\scriptscriptstyle{i, 1}}  B_{\scriptscriptstyle {\mathbf{n}, 1}}
= s_{\scriptscriptstyle{i, 1}}   \,, 
 \end{equation}     
 \begin{equation}\label{eq: Bpi,2}
B_{\scriptscriptstyle {p_{i}} , \scriptscriptstyle{k+1} } 
\, - \, ( \varepsilon_{\scriptscriptstyle{i, 1}} )^{k+1}  B_{\scriptscriptstyle \mathbf{n}, \scriptscriptstyle{k+1} } 
=   s_{\scriptscriptstyle{i, k+1}} 
\,  + \,  \varepsilon_{\scriptscriptstyle{i, k+1}}  B_{\scriptscriptstyle {\mathbf{n}, 1}} 
\, + \, \mathbf{B}_{\mathbf{\p}_{i},\scriptscriptstyle{k+1}} 
\end{equation}
where
\begin{equation}\label{eq: Bpi,3}
 \mathbf{B}_{\mathbf{\p}_{i},\scriptscriptstyle{k+1}}=\mathbf{B}_{\mathbf{\p},\scriptscriptstyle{k+1}}[\lambda_{i}, s_{\scriptscriptstyle{i,r}} ;  \varepsilon_{\scriptscriptstyle{i, r}}, B_{\scriptscriptstyle {\mathbf{n}, r}}]^{k}_{r = 1}\,, 
\end{equation}
is holomorphic in the annulus $D_{p_{i}} \smallsetminus \{(0, p_{i})\}$ and it is defined by evaluation of the transformation  $\mathbf{B}_{\mathbf{\p},\scriptscriptstyle{k+1}}$ in the complex numbers $\lambda_{i}$, $s_{\scriptscriptstyle{i,r}} $ (information of the local model of $(\F, \G)$ at the point  $(0,p_{i})$), and in the functions, $ \varepsilon_{\scriptscriptstyle{i, r}}$, $B_{\scriptscriptstyle {\mathbf{n}, r}}$, $1 \leqslant r \leqslant k$ (coefficients of the series $\varepsilon_{i}$, $B_{\scriptscriptstyle \mathbf{n}}$).
\end{itemize}

\end{lemma}

The proof of this and the next lemma is based on Lemma \ref{Lema_Series de composicion A en alpha y beta}  and Lemma \ref{Lema_Aplicacion hatPsik_i}.

\begin{lemma}\label{lem: series TN por ecFac en qj}
    Let $H_{q_{j}} = (A_{q_{j}}, B_{q_{j}})$, be the normalizing transformations
at the points $q_j$, $j=1,\dots, m$, defined by the corresponding factorization equations 
$H_{q_{j}} = \mathbf{H}_{\scriptscriptstyle \mathbf{n}} \circ \zeta_{j} \circ \Phi_{j}$.
 The power series defining the biholomorphisms  $ A_{q_{j}}$ and $B_{q_{j}}$  in (\ref{Ec_Series de transformaciones normalizantes}) satisfy the relations

\begin{itemize}
    \item[a)] For $A_{q_{j}} (x,u)   = \sum_{k \geqslant 1}  A_{\scriptscriptstyle{q_{j}, {k}}} (u) \, x^{k} \, ,$

\begin{equation}\label{eq: Aqj,1}
     A_{\scriptscriptstyle{q_{j}, {1}}}  
=   \varsigma_{\scriptscriptstyle{j, 1}}   \, ,
 \end{equation}
     
 \begin{equation}\label{eq: Aqj,2}
 A_{\scriptscriptstyle{q_{j}, k+1}} 
\, - \, ( \varsigma_{\scriptscriptstyle{j, 1}} )^{k+1}
A_{\scriptscriptstyle \mathbf{n}, \scriptscriptstyle{k+1} }
 =  \varsigma_{\scriptscriptstyle{j, k+1}}  
\, + \, \mathbf{A}_{\mathbf{\q}_{j},\scriptscriptstyle{k+1}}\,,
\end{equation}
where
\begin{equation}\label{eq: Aqj,3}
  \mathbf{A}_{\mathbf{\q}_{j},\scriptscriptstyle{k+1}}=  \mathbf{A}_{\mathbf{\q},\scriptscriptstyle{k+1}}[z_{\scriptscriptstyle{j,r}} ; g_{j}, \varsigma_{\scriptscriptstyle{i, r}}, A_{\scriptscriptstyle {\mathbf{n}, r}}]^{k}_{r = 1} \, , \,\, k\geq 1\,,
\end{equation}
is holomorphic in the annulus $D_{q_{j}} \smallsetminus \{(0, q_{j})\}$ and it is defined by evaluation of the transformation  $\mathbf{A}_{\mathbf{\q},\scriptscriptstyle{k+1}}$ in the complex number $z_{\scriptscriptstyle{j,r}}$, in the derivatives  $g_{j}^{(1)}$, $g_{j}^{(r+1)}$ (information of the local model of $(\F, \G)$ at the point  $(0,q_{j})$), and in the functions 
$ \varsigma_{\scriptscriptstyle{j, r}} $, $A_{\scriptscriptstyle {\mathbf{n}, r}}(u)$, $1 \leqslant r \leqslant k$ (coefficients of the series $\varsigma_{j}$ and  $A_{\scriptscriptstyle \mathbf{n}}$).
\item[b)]  For $B_{q_{j}} (x,u)   = \sum_{k \geqslant 1}  B_{\scriptscriptstyle{q_{j}, {k}}} (u) \, x^{k} \, ,$

\begin{equation}\label{eq: Bqj,1}
     B_{\scriptscriptstyle{q_{j}, {1}}} 
\, - \,  \varsigma_{\scriptscriptstyle{j, 1}}  B_{\scriptscriptstyle {\mathbf{n}, 1}}
= \tfrac{z_{\scriptscriptstyle{j, 1}}}{g^{(1)}_{j}}   \,, 
 \end{equation}
     
 \begin{equation}\label{eq: Bqj,2}
B_{\scriptscriptstyle {q_{j}} , \scriptscriptstyle{k+1} } 
\, - \, ( \varsigma_{\scriptscriptstyle{j, 1}} )^{k+1}  B_{\scriptscriptstyle \mathbf{n}, \scriptscriptstyle{k+1} } 
=   \tfrac{ z_{\scriptscriptstyle{j, k+1}}}{g_{j}^{(1)}} 
\,  + \,  \varsigma_{\scriptscriptstyle{j, k+1}}  B_{\scriptscriptstyle {\mathbf{n}, 1}} 
\, + \, \mathbf{B}_{\mathbf{\q}_{j},\scriptscriptstyle{k+1}} 
\end{equation}
where
\begin{equation}\label{eq: Bqj,3}
 \mathbf{B}_{\mathbf{\q}_{j},\scriptscriptstyle{k+1}}=\mathbf{B}_{\mathbf{\q},\scriptscriptstyle{k+1}}[  z_{\scriptscriptstyle{j,r}} ; g_{j},  \varsigma_{\scriptscriptstyle{j, r}}, B_{\scriptscriptstyle {\mathbf{n}, r}}]^{k}_{r = 1}\,, 
\end{equation}
is holomorphic in the annulus $D_{q_{j}} \smallsetminus \{(0, q_{j})\}$ and it is defined by evaluation of the transformation  $\mathbf{B}_{\mathbf{\q},\scriptscriptstyle{k+1}}$ in the complex number $z_{\scriptscriptstyle{j,r}}$ , the derivatives  $g_{j}^{(1)}$, $g_{j}^{(r+1)}$ (information of the local model of $(\F, \G)$ at the point  $(0,q_{j})$), and in the functions 
$ \varsigma_{\scriptscriptstyle{j, r}} $, $B_{\scriptscriptstyle {\mathbf{n}, r}}(u)$, $1 \leqslant r \leqslant k$ (coefficients of the series $\varsigma_{j}$ and  $B_{\scriptscriptstyle \mathbf{n}}$).
\end{itemize}
\end{lemma}

\begin{remark}\label{Obs_Ecuaciones de Factorizacion en coeficientes}
  
The relations (\ref{eq: Api,1})-(\ref{eq: Api,3}) and (\ref{eq: Bpi,1})-(\ref{eq: Bpi,3}) in Lemma \ref{lem: series TN por ecFac en pi} give the expressions of the factorization equations at the points $(0,p_i)$ in terms of the corresponding normalizing transformations. Analoguosly, the relations (\ref{eq: Aqj,1})-(\ref{eq: Aqj,3}) and (\ref{eq: Bqj,1})-(\ref{eq: Bqj,3}) in Lemma \ref{lem: series TN por ecFac en qj} give the expressions of the factorization equations at the points $(0,q_j)$ in terms of the corresponding normalizing transformations.

Note that the relations that satisfy the coefficients $ A_{\scriptscriptstyle{p_{i}, k+1}}$, $B_{\scriptscriptstyle{p_{i}, k+1}}$ of the power series defining the biholomorphisms  $ A_{p_{i}}$ and $B_{p_{i}}$ in Lemma \ref{lem: series TN por ecFac en pi} are given in terms of the information given by the local models at the points $(0,p_i)$ together with the information of the functions, $ \varepsilon_{\scriptscriptstyle{i, r}}$, $A_{\scriptscriptstyle {\mathbf{n}, r}}$, $1 \leqslant r \leqslant k$ (coefficients of the series $\varepsilon_{i}$ and $A_{\scriptscriptstyle \mathbf{n}}$, $B_{\scriptscriptstyle \mathbf{n}}$ respectively) that depend on the non dicritical foliation $\mathcal{F}_{\scriptscriptstyle \mu}$.

Analogusly, the relations that satisfy the coefficients $ A_{\scriptscriptstyle{q_{j}, k+1}}$, $B_{\scriptscriptstyle{q_{j}, k+1}}$ of the power series defining the biholomorphisms  $ A_{q_{j}}$ and $B_{q_{j}}$ in Lemma \ref{lem: series TN por ecFac en qj} are given in terms of the information given by the local models at the points $(0,q_j)$ together with the information of the functions,  $ \varepsilon_{\scriptscriptstyle{i, r}}$, $A_{\scriptscriptstyle \mathbf{n}}$, $B_{\scriptscriptstyle \mathbf{n}}$, respectively.

\end{remark}

\subsection{Suitable normal forms for normalizing transformations. Proof of Theorem \ref{Teorema_Limpieza de TransNorm_Intro}}
\label{Limpieza de TransNorm}

Our aim in this section is to get, by means of coordinate changes tangent to the identity, finite jets of the normalizing transformations expresed in terms of the local models of the original foliation pair. This result is expressed in Theorem \ref{Teorema_Limpieza de TransNorm_Intro}. Such finite jets give \emph{canonical solutions} to the equations expressed in (\ref{eq: Api,1})-(\ref{eq: Api,3}) and (\ref{eq: Bpi,1})-(\ref{eq: Bpi,3}) in Lemma \ref{lem: series TN por ecFac en pi}, and the equations (\ref{eq: Aqj,1})-(\ref{eq: Aqj,3}) and (\ref{eq: Bqj,1})-(\ref{eq: Bqj,3}) in Lemma \ref{lem: series TN por ecFac en qj}.


\vspace{0.2cm}

In what follows we fix the following data.

Let $(\F, \G)$ be a foliation pair in $\mathcal{N}_{\p} (\mathbf{h}) \times \mathcal{D}_{\q} (\mathcal{I})$. Suppose that $(\mathcal{F}_{p_{i}}^l, \G_{p_{i}})$ is the local model at the singular point $(0,p_{i})$, $i=1,\dots, n+1$, with respect to the foliation pair $(\F_{\scriptscriptstyle \mu}, \G_{\scriptscriptstyle \mathbf{r}})$, where $\mathcal{F}_{p_{i}}^l$ is the foliation generated by the vector field $\lambda_{i} \, x \tfrac{\partial}{\partial \, x} + (u-p_{i}) \tfrac{\partial}{\partial \, u}$, and $\G_{p_{i}}$ is the foliation defined by the level curves of  $u + s_{i} (x)$, where $s_{i} = \sum_{r \geqslant 1} s_{\scriptscriptstyle{i,r}}  \, x^{r}$, $s_{\scriptscriptstyle{i,r}}  \in \C$, is holomorphic. We assume as well that $((x = \text{cst}), \G_{q_{j}})$ is the local model at the point $(0,q_{j})$, , $j=1,\dots, m$, with respect to $(\F_{\scriptscriptstyle \mu}, \G_{\scriptscriptstyle \mathbf{r}})$, where $\G_{q_{j}}$ is the foliation given by the level curves of the holomorphic function $z_{j}(x) + g_{j}(u)$, $z_{j} = \sum_{r \geqslant 1} z_{\scriptscriptstyle{j, r}}  \, x^{r}$, $z_{\scriptscriptstyle{j, r}}  \in \C$, is a biholomorphism, and $g_{j} (u)$ is the holomorphic function $(q_{j} - u) (I_{j} (u) - q_{j})$ defined by the involution $I_{j}\in\mathcal{I} = (I_{1}, \ldots, I_{m})$.

\begin{proposition}\label{Prop:soluciones canonicas} 
  The \emph{canonical solutions} to the equations expressed in (\ref{eq: Api,1})-(\ref{eq: Api,3}) and (\ref{eq: Bpi,1})-(\ref{eq: Bpi,3}) in Lemma \ref{lem: series TN por ecFac en pi}, and the equations (\ref{eq: Aqj,1})-(\ref{eq: Aqj,3}) and (\ref{eq: Bqj,1})-(\ref{eq: Bqj,3}) in Lemma \ref{lem: series TN por ecFac en qj} are given by the following relations:

  \begin{itemize}
      \item[a)] For $k=1$
      \label{Pageref_Aplicaciones_a1_b1}
\begin{align*}
a_{\scriptscriptstyle{\mathbf{n}, 1}} & = 1 \, ,  
& a_{\scriptscriptstyle{p_{i} , 1}} & = \varepsilon_{\scriptscriptstyle{{i} , 1}} \, ,
& a_{\scriptscriptstyle{q_{j} , 1}} & = \varsigma_{\scriptscriptstyle{j, 1}} \, ,
\tag{$a^{1}$}
\\ b_{\scriptscriptstyle \mathbf{n}, 1} 
& = - \, \sum_{\jmb =1}^{m} \Bigl(\tfrac{z_{\scriptscriptstyle{{\jmb} , 1}}}{\varsigma_{\scriptscriptstyle{{\jmb} , 1}} \, g^{(1)}_{\jmb}}\Bigl)_{\CP, q_{\jmb}} \, , 
& b_{\scriptscriptstyle{p_{i} , 1}} 
& = s_{\scriptscriptstyle{i, 1}} 
\, + \,  \varepsilon_{\scriptscriptstyle{i, 1}}  \, b_{\scriptscriptstyle \mathbf{n}, 1}  \, ,
& b_{\scriptscriptstyle{q_{j} , 1}} 
& := \tfrac{z_{\scriptscriptstyle{j, 1}}}{g^{(1)}_{j}} 
\, + \,   \varsigma_{\scriptscriptstyle{j, 1}} \, b_{\scriptscriptstyle \mathbf{n}, 1} \, , 
\tag{$b^{1}$}
\end{align*}
\hspace{-1.2cm}where $(\hspace{0.2cm})_{\CP, q_{\jmb}}$ is the principal part of the function at the point $q_{\jmb} \in \C$. 
\item[b)] For $k\geqslant 1$ we define recursively 
\label{Pageref_Aplicaciones_ak+1}
\begin{equation*}
\begin{split}
a_{\scriptscriptstyle \mathbf{n}, \scriptscriptstyle{k+1} } 
 := -  \, & \sum_{\imb =1}^{n+1} \Bigl(\tfrac{\mathbf{A}_{\mathbf{\p},\scriptscriptstyle{k+1}}}{(\varepsilon^{1}_{\imb})^{k+1}} \Bigl)_{\CP, p_{\imb}}  
\, - \, \sum_{\jmb =1}^{m} \Bigl(\tfrac{\mathbf{A}_{\mathbf{\q},\scriptscriptstyle{k+1}}}{(\varsigma_{\scriptscriptstyle{{\jmb} , 1}})^{k+1}}\Bigl)_{\CP, q_{\jmb}}
\, , 
\\ a_{\scriptscriptstyle {p_{i}} , \scriptscriptstyle{k+1} }
& :=  \varepsilon_{\scriptscriptstyle{i, k+1}}  \, + \, \mathbf{A}_{\mathbf{\p}_{i},\scriptscriptstyle{k+1}}
\, + \,  ( \varepsilon_{\scriptscriptstyle{i, 1}} )^{k+1} \, a_{\scriptscriptstyle \mathbf{n}, \scriptscriptstyle{k+1} }  \, , 
\\ a_{\scriptscriptstyle {q_{j}} , \scriptscriptstyle{k+1} } 
& :=  \varsigma_{\scriptscriptstyle{j, k+1}}  \, + \, \mathbf{A}_{\mathbf{\q}_{j},\scriptscriptstyle{k+1}} 
\, + \,  ( \varsigma_{\scriptscriptstyle{j, 1}} )^{k+1} \, a_{\scriptscriptstyle \mathbf{n}, \scriptscriptstyle{k+1} } \, 
,
\end{split}
\tag{$a^{k+1}$}
\end{equation*}
\hspace{-1.2cm}$\mathbf{A}_{\mathbf{\p}_{i},\scriptscriptstyle{k+1}} =\mathbf{A}_{\mathbf{\p},\scriptscriptstyle{k+1}}[\lambda_{i}, s_{\scriptscriptstyle{i,r}} ; \varepsilon_{\scriptscriptstyle{j, r}}, a_{\scriptscriptstyle \mathbf{n},\, r}]^{k}_{r = 1}$, \hspace{0.5cm} $\mathbf{A}_{\mathbf{\q}_{j},\scriptscriptstyle{k+1}}=\mathbf{A}_{\mathbf{\q},\scriptscriptstyle{k+1}}[z_{\scriptscriptstyle{j, r}} ; g_{j},  \varsigma_{\scriptscriptstyle{j, r}} , a_{\scriptscriptstyle \mathbf{n},\, r}]^{k}_{r = 1};$
\begin{equation*}
\begin{split}
b_{\scriptscriptstyle \mathbf{n}, \scriptscriptstyle{k+1} } 
 :=  -  \, & \sum_{\imb =1}^{n+1} \Bigl(\tfrac{\mathbf{B}_{\mathbf{\p},\scriptscriptstyle{k+1}}}{(\varepsilon^{1}_{\imb})^{k+1}} \Bigl)_{\scriptscriptstyle{\CP, p_{\imb}}}  
-  \sum_{\jmb =1}^{m} \Bigl(\tfrac{z_{\scriptscriptstyle{{\jmb} , k+1}}}{(\varsigma_{\scriptscriptstyle{{\jmb} , 1}})^{k+1} g_{\jmb}^{(1)}}
 + \tfrac{\mathbf{B}_{\mathbf{\q},\scriptscriptstyle{k+1}}}{(\varsigma_{\scriptscriptstyle{{\jmb} , 1}})^{k+1}}\Bigl)_{\scriptscriptstyle{\CP, q_{\jmb}}}  
\\ b_{\scriptscriptstyle {p_{i}} , \scriptscriptstyle{k+1} } 
 & :=  s_{\scriptscriptstyle{i, k+1}}
\, + \,  \varepsilon_{\scriptscriptstyle{i, k+1}}  \, b_{\scriptscriptstyle{\mathbf{n}, 1}}
\, + \, \mathbf{B}_{\mathbf{\p}_{i},\scriptscriptstyle{k+1}} 
\, + \,  ( \varepsilon_{\scriptscriptstyle{i, 1}} )^{k+1} \, b_{\scriptscriptstyle \mathbf{n}, \scriptscriptstyle{k+1} }  \, , 
\\ b_{\scriptscriptstyle {q_{j}} , \scriptscriptstyle{k+1} }
 & := \tfrac{z_{\scriptscriptstyle{j, k+1}}}{g_{j}^{(1)}}
\, + \,  \varsigma_{\scriptscriptstyle{j, k+1}}  \,  b_{\scriptscriptstyle \mathbf{n}, 1}
\, + \, \mathbf{B}_{\mathbf{\q_{j}},\scriptscriptstyle{k+1}} 
\, + \, ( \varsigma_{\scriptscriptstyle{j, 1}} )^{k+1} \,  b_{\scriptscriptstyle \mathbf{n}, \scriptscriptstyle{k+1} } \, .
\end{split}
\tag{$b^{k+1}$}
\end{equation*}
\hspace{-1.2cm}$\mathbf{B}_{\mathbf{\p}_{i},\scriptscriptstyle{k+1}}=\mathbf{B}_{\mathbf{\p},\scriptscriptstyle{k+1}}[\lambda_{i}, s_{\scriptscriptstyle{i,r}} ; \varepsilon_{\scriptscriptstyle{j, r}}, a_{\scriptscriptstyle \mathbf{n},\, r}]^{k}_{r = 1}$, \hspace{0.5cm} $ \mathbf{B}_{\mathbf{\q}_{j},\scriptscriptstyle{k+1}}=\mathbf{B}_{\mathbf{\q},\scriptscriptstyle{k+1}}[z_{\scriptscriptstyle{j, r}} ; g_{j},  \varsigma_{\scriptscriptstyle{j, r}} , a_{\scriptscriptstyle \mathbf{n},\, r}]^{k}_{r = 1};$
  \end{itemize}
  where $(\hspace{0.2cm})_{\mathcal{P}, p_{\imb}}$ and $(\hspace{0.2cm})_{\mathcal{P}, q_{\jmb}}$ are the principal parts of the functions at the corresponding points $p_{\imb}$, $q_{\jmb} \in \C$.  
\end{proposition}

For the sake of readability, we state Theorem \ref{Teorema_Limpieza de TransNorm_Intro} in its extended version given in Theorem \ref{Teorema_Limpieza de TransNorm_Intro_1}. It asserts that the normalizing transformations of the foliation pair $(\mathcal{F},\mathcal{G})$ have a finite jet that provides canonical solutions to the equations (\ref{Ec_Series de transformaciones normalizantes_pi}) and (\ref{Ec_Series de transformaciones normalizantes_qj}). This is given up to coordinate changes tangent to the identity map:
\vspace{.3cm}

\begin{theorem}[Normal forms for normalizing transformations]\label{Teorema_Limpieza de TransNorm_Intro_1} 
Let $(\F, \G)$ be a pair of foliations in $\mathcal{N}_{\p}(\mathbf{h}) \times \mathcal{D}_{\q}(\mathcal{I})$ and assume that the $y$ axis is an invariant  branch of the dicritical foliation $\G$. Let  $\mathbf{H}_{\scriptscriptstyle \mathbf{n}} $, $H_{p_{i}}$, $H_{q_{j}}$ be the normalizing transformations with respect to the foliation pair $(\F_{\scriptscriptstyle \mathbf{\mu}}, \G_{\scriptscriptstyle \mathbf{r}})$.

For each $k \geqslant 1$, there exists $\mathcal{H}^{k} \colon (\C^{2}, 0) \rightarrow (\C^{2}, 0)$, a tangent to the identity biholomorphism, which preserves the $y$-axis, and such that the pair of foliations $\bigl(\mathcal{H}^{k}(\F), \mathcal{H}^{k} (\G)\bigl)$ that belong to $\mathcal{N}_{\p}(\mathbf{h}) \times \mathcal{D}_{\q}(\mathcal{I})$ has local normalizing transformations 
$\tilde{\mathcal{H}}^{k} \circ \mathbf{H}_{\scriptscriptstyle \mathbf{n}}$, $\tilde{\mathcal{H}}^{k} \circ H_{p_{i}}$, $\tilde{\mathcal{H}}^{k} \circ H_{q_{j}}$, correspondingly. The power series of such transformations in the coordinates $(x,u)$ are of the form
\label{Pageref_TN limpias_TNk} 
\begin{equation*}
\tilde{\mathcal{H}}^{k} \circ H_{[\centerdot]}=\Bigl(\sum_{r = 1}^{k} a_{[\centerdot], r}(u) x^{r}  
\, + \, \text{O} (x^{k+1}), 
\, u \, + \, \sum_{r = 1}^{k} b_{[\centerdot], r}(u) x^{r}
\, + \, \text{O} (x^{k+1})\Bigl) \, ,
\tag{$\mathbf{nt_{k}}$}
\end{equation*}
where the subindex $[\centerdot]$ represents $\scriptscriptstyle \mathbf{n}$, $p_{i}$, or $q_{j}$; the transformations $a_{\scriptscriptstyle{\mathbf{n}, r}}$, $b_{\scriptscriptstyle{\mathbf{n}, r}}$ are holomorphic in the sphere with holes $\mathcal{L} \setminus\{\hat{p}_{1},\ldots,\hat{p}_{n+1}, \hat{q}_{1},\ldots,\hat{q}_{m}\}$, and the local  transformations $a_{\scriptscriptstyle{p_{i} , r}}$, $b_{\scriptscriptstyle{p_{i} , r}}$ and $a_{\scriptscriptstyle{q_{j} , r}}$, $b_{\scriptscriptstyle{q_{j} , r}}$ are holomorphic in the neighborhoods $(\mathcal{L},\hat{p}_{i})$ and $(\mathcal{L},\hat{q}_{j})$, respectively (see Proposition \ref{Prop:soluciones canonicas}).
\end{theorem}

\begin{remark}
\label{Obs_Hipotesis del Teorema_Limpieza de TransNorm}
Any foliation pair $(\F, \G)$ in $\mathcal{N}_{\p}(\mathbf{h}) \times \mathcal{D}_{\q}(\mathcal{I})$ is strictly analytically equivalent to one satisfying the conditions of Theorem \ref{Teorema_Limpieza de TransNorm_Intro_1}. Thus, for any $k\geqslant 1$ there exists a foliation pair in the class of strict analytic equivalence of foliation pair $(\F, \G)$, such that its normalizing transformations have power series as in ($\mathbf{nt_{k}}$). 
\end{remark}

\begin{proof}[Proof of Theorem \ref{Teorema_Limpieza de TransNorm_Intro_1}]
We will prove the theorem by induction.

\emph{Base of induction}. For $k = 1$, we denote the power series of the normalizing transformation $H_{[\centerdot]} = (A_{[\centerdot]}, B_{[\centerdot]})$ of the foliation pair $(\F, \G)$ as in (\ref{Ec_Series de transformaciones normalizantes}).
From the definition of the normalizing transformation, it follows that $A_{\scriptscriptstyle \mathbf{n}\, , 1} =1 =: a_{\scriptscriptstyle{\mathbf{n}, 1}}$. Thus,
\begin{equation*}
  a_{\scriptscriptstyle{\mathbf{n}, 1}}:=  A_{\scriptscriptstyle \mathbf{n}\, , 1} =1\,.
\end{equation*}

From equalities (\ref{eq: Api,1}) in Lemma \ref{lem: series TN por ecFac en pi} and  (\ref{eq: Aqj,1}) in Lemma \ref{lem: series TN por ecFac en qj}, it follows that 
\begin{equation*}
   a_{\scriptscriptstyle{p_{i} , 1}}:= A_{\scriptscriptstyle{p_{i}, {1}}}=\varepsilon_{\scriptscriptstyle{i, 1}}
\end{equation*}
and 
\begin{equation*}
    a_{\scriptscriptstyle{q_{j} , 1}}:= A_{\scriptscriptstyle{q_{j}, {1}}}=\varsigma_{\scriptscriptstyle{j, 1}}\,.
\end{equation*}

Moreover, from relations (\ref{eq: Bpi,1}), (\ref{eq: Bqj,1}) it follows that 
\begin{equation*}
\tfrac{ B_{\scriptscriptstyle{p_{i}, {1}}}}{ \varepsilon_{\scriptscriptstyle{i, 1}} } 
\, - \, B_{\scriptscriptstyle {\mathbf{n}, 1}}
= \tfrac{s_{\scriptscriptstyle{i, 1}} }{ \varepsilon_{\scriptscriptstyle{i, 1}} }  \, ,
\hspace{1cm}
\tfrac{ B_{\scriptscriptstyle{q_{j}, {1}}}}{ \varsigma_{\scriptscriptstyle{j, 1}} } 
\, - \,  B_{\scriptscriptstyle {\mathbf{n}, 1}}
= \tfrac{z_{\scriptscriptstyle{j, 1}}}{ \varsigma_{\scriptscriptstyle{j, 1}} \, g^{(1)}_{j}}
\, .
\tag{$B^{1}$}
\end{equation*}
Since, by definition, the coefficients given by the functions  $b_{\scriptscriptstyle \mathbf{n} , 1}$,  $b_{\scriptscriptstyle{p_{i} , 1}}$  and $b_{\scriptscriptstyle{q_{j} , 1}}$ must satisfy also the former equalities ($B^{1}$), we can take the difference $b_{\scriptscriptstyle \mathbf{n} , 1} \, - B_{\scriptscriptstyle {\mathbf{n}, 1}} $ at the punctured disks
$D_{p_{i}} \smallsetminus \{(0, p_{i})\}$ and $D_{q_{j}} \smallsetminus \{(0, q_{j})\}$, and get the following relations

\begin{equation}\label{ec: bn1 en pi}
b_{\scriptscriptstyle \mathbf{n} , 1} \, - \, B_{\scriptscriptstyle {\mathbf{n}, 1}} 
= \tfrac{b_{\scriptscriptstyle{p_{i} , 1}}}{ \varepsilon_{\scriptscriptstyle{i, 1}} } 
\, - \,  \tfrac{ B_{\scriptscriptstyle{p_{i}, {1}}}}{ \varepsilon_{\scriptscriptstyle{i, 1}} } \, ,
\end{equation}
and
\begin{equation}\label{ec: bn1 en qj}
b_{\scriptscriptstyle \mathbf{n} , 1} \, - \,  B_{\scriptscriptstyle {\mathbf{n}, 1}}
= \tfrac{b_{\scriptscriptstyle{q_{j} , 1}}}{ \varsigma_{\scriptscriptstyle{j, 1}} } 
\, - \,  \tfrac{ B_{\scriptscriptstyle{q_{j}, {1}}}}{ \varsigma_{\scriptscriptstyle{j, 1}} }  \, ,
\end{equation}
respectively. Since the expressions at the right hand side of 
(\ref{ec: bn1 en pi}) and (\ref{ec: bn1 en qj}) are holomorphic in their respective disks, the difference $b_{\scriptscriptstyle \mathbf{n} , 1} \, - \,  B_{\scriptscriptstyle {\mathbf{n}, 1}}$ extends holomorphically to the whole complex plane $\C$. Note that $b_{\scriptscriptstyle \mathbf{n} , 1}$ is holomorphic at $u=\infty$, while $B_{\scriptscriptstyle \mathbf{n} , 1}$ has a pole of order at most $2$. This happens because the biholomorphism $\mathbf{H}_{\scriptscriptstyle \mathbf{n}} $ leaves invariant the straight line $\{v = 0\}$, where $v = x/y$. 
Therefore, since we know that  $b_{\scriptscriptstyle \mathbf{n} , 1} -  B_{\scriptscriptstyle {\mathbf{n}, 1}}$ is holomorphic in $\C$, it must be a polynomial $\mathbf{s}_{2} \in \C[u]$ of degree less or equal to $2$. 

Let $\mathbf{G}_{1}: (\C^{2},0)\longrightarrow (\C^{2},0)$ be the biholomorphism 
\begin{equation*}
\mathbf{G}_{1} (x, y) := (x, y + x^{2} \mathbf{s}_{2}(y/x)) \, .
\end{equation*}
We stress that $\mathbf{G}_{1}$ satisfies the assumptions of the theorem for $k=1$. Namely, its blowing-up morphism   $\tilde{\mathbf{G}}_{1}(x, u) = (x, u + x \mathbf{s}_{2} (u))$ is tangent to the identity and therefore, the transformations

\begin{itemize}
    \item[a)] $\tilde{\mathbf{G}}_{1} \circ \mathbf{H}_{\scriptscriptstyle \mathbf{n}} = (\hat{A}_{\scriptscriptstyle {\mathbf{n}}}, \hat{B}_{\scriptscriptstyle {\mathbf{n}}}) = (\sum_{k \geqslant 1}  \hat{A}_{\scriptscriptstyle \mathbf{n} ,\scriptscriptstyle {k}} (u) \, x^{k} \, ,\, u + \sum_{r \geqslant 1} \hat{B}_{\scriptscriptstyle \mathbf{n}\, , r} (u) \, x^{r})$.
    \item[b)] $\tilde{\mathbf{G}}_{1} \circ H_{p_{i}} = (\hat{A}_{p_{i}}, \hat{B}_{p_{i}}) = (\sum_{k \geqslant 1}  \hat{A}_{\scriptscriptstyle {p_{i}} ,\scriptscriptstyle {k}} (u) \, x^{k} \, ,\, u + \sum_{r \geqslant 1} \hat{B}_{\scriptscriptstyle {p_{i}}\, , r} (u) \, x^{r})$. 
    \item[c)] $\tilde{\mathbf{G}}_{1} \circ H_{q_{j}} = (\hat{A}_{q_{j}}, \hat{B}_{q_{j}}) = (\sum_{k \geqslant 1}  \hat{A}_{\scriptscriptstyle {q_{j}} ,\scriptscriptstyle {k}} (u) \, x^{k} \, ,\, u + \sum_{r \geqslant 1} \hat{B}_{\scriptscriptstyle {q_{j}}\, , r} (u) \, x^{r})$. 
\end{itemize}
are normalizing transformations of the foliation pair $\bigl(\mathbf{G}_{1} (\F), \mathbf{G}_{1}(\G)\bigl)$ with respect to the foliation pair $(\F_{\scriptscriptstyle \mathbf{\mu}}, \G_{\scriptscriptstyle \mathbf{r}})$. 

Note that the $y$ axis is an invariant branch of the dicritical foliation $\mathbf{G}_{1}(\G)$, since the biholomorphism $\mathbf{G}_{1}$ is the identity in ${x = 0}$. Moreover,

\begin{itemize}
    \item[a)] For $(x,u)$ in a neighborhood of $\mathcal{L} \setminus\{\hat{p}_{1},\ldots,\hat{p}_{n+1}, \hat{q}_{1},\ldots,\hat{q}_{m}\}$,
    \begin{equation*}
    \begin{split}
    \tilde{\mathbf{G}}_{1} \circ \mathbf{H}_{\scriptscriptstyle \mathbf{n}} 
& = \bigl(A_{\scriptscriptstyle \mathbf{n}}, \,  B_{\scriptscriptstyle \mathbf{n}} 
\, + \, A_{\scriptscriptstyle \mathbf{n}} \mathbf{s}_{2} (B_{\scriptscriptstyle \mathbf{n}})\bigl) \\
& = \bigl(A_{\scriptscriptstyle \mathbf{n}}, \,  B_{\scriptscriptstyle \mathbf{n}} 
\, + \, \bigl(x + \text{O}(x^{2})\bigl) \mathbf{s}_{2} (u + \text{O}(x))\bigl) \\
& = \bigl(A_{\scriptscriptstyle \mathbf{n}}, \,  u + (B_{\scriptscriptstyle \mathbf{n}, {1}} + \mathbf{s}_{2} (u))x  
\, + \, \text{O}(x^{2})\bigl) \, .
\end{split}
\end{equation*}
Thus, $\hat{A}_{\scriptscriptstyle {\mathbf{n}}}= A_{\scriptscriptstyle \mathbf{n}}$, and $\hat{B}_{\scriptscriptstyle \mathbf{n} , 1}= B_{\scriptscriptstyle \mathbf{n}, {1}} + \mathbf{s}_{2} (u) = b_{\scriptscriptstyle \mathbf{n} , 1} $.

\item[b)] For $(x,u)$ in a neighborhood of $(\mathcal{L},\hat{p}_{i})$
    \begin{equation*}
    \begin{split}
    \tilde{\mathbf{G}}_{1} \circ \mathbf{H}_{\scriptscriptstyle p_{i}} 
& = \bigl(A_{\scriptscriptstyle {p_{i}}}, \,  B_{\scriptscriptstyle {p_{i}}} 
\, + \, A_{\scriptscriptstyle {p_{i}}} \mathbf{s}_{2} (B_{\scriptscriptstyle {p_{i}}})\bigl) \\
& = \bigl(A_{\scriptscriptstyle {p_{i}}}, \,  B_{\scriptscriptstyle {p_{i}}} 
\, + \, \bigl(\varepsilon_{\scriptscriptstyle{{i} , 1}} \, x + \text{O}(x^{2})\bigl) \mathbf{s}_{2} (u + \text{O}(x))\bigl) \\
& = \bigl(A_{\scriptscriptstyle {p_{i}}}, \,  u + (B_{\scriptscriptstyle {p_{i}}, {1}} \, + \, \varepsilon_{\scriptscriptstyle{{i} , 1}} \, \mathbf{s}_{2} (u))x  
\, + \, \text{O}(x^{2})\bigl) \, .
\end{split}
\end{equation*} 
Thus, $\hat{A}_{p_{i}}=A_{\scriptscriptstyle {p_{i}}}$\,\, and
$\hat{B}_{\scriptscriptstyle {p_{i}}\, , 1}= B_{\scriptscriptstyle {p_{i}}, {1}} \, + \, \varepsilon_{\scriptscriptstyle{{i} , 1}} \, \mathbf{s}_{2} (u) $. On the other hand, by the equality \eqref{ec: bn1 en pi}, 
\begin{equation*}
b_{\scriptscriptstyle{p_{i} , 1}} 
\, - \, 
B_{\scriptscriptstyle{p_{i}, {1}}} 
\, = \, 
\varepsilon_{\scriptscriptstyle{i, 1}} \left(b_{\scriptscriptstyle \mathbf{n} , 1} \, - \, B_{\scriptscriptstyle {\mathbf{n}, 1}} \right)
\, = \,   
\varepsilon_{\scriptscriptstyle{i, 1}} \, \mathbf{s}_{2} \, .
\end{equation*}
Therefore, $b_{\scriptscriptstyle{p_{i} , 1}} \, = \, B_{\scriptscriptstyle {p_{i}}, {1}} \, + \, \varepsilon_{\scriptscriptstyle{{i} , 1}} \, \mathbf{s}_{2} (u) \, =  \, \hat{B}_{\scriptscriptstyle {p_{i} , 1} } $.

 
   \item[c)] For $(x,u)$ in a neighborhood of $(\mathcal{L},\hat{q}_{j})$
      \begin{equation*}
    \begin{split}
    \tilde{\mathbf{G}}_{1} \circ \mathbf{H}_{\scriptscriptstyle q_{j}} 
& = \bigl(A_{\scriptscriptstyle q_{j}}, \,  B_{\scriptscriptstyle q_{j}} 
\, + \, A_{\scriptscriptstyle q_{j}} \mathbf{s}_{2} (B_{\scriptscriptstyle q_{j}})\bigl) \\
& = \bigl(A_{\scriptscriptstyle q_{j}}, \,  B_{\scriptscriptstyle q_{j}} 
\, + \, \bigl(\varsigma_{\scriptscriptstyle{{i} , 1}} \, x + \text{O}(x^{2})\bigl) \mathbf{s}_{2} (u + \text{O}(x))\bigl) \\
& = \bigl(A_{\scriptscriptstyle {q_{j}}}, \,  u + (B_{\scriptscriptstyle {q_{j}}, {1}} \, + \, \varsigma_{\scriptscriptstyle{{i} , 1}} \, \mathbf{s}_{2} (u))x  
\, + \, \text{O}(x^{2})\bigl) \, .
\end{split}
\end{equation*} 
Thus, $\hat{A}_{q_{j}}=A_{\scriptscriptstyle {q_{j}}}$\,\, and
$\hat{B}_{\scriptscriptstyle {q_{j}}\, , 1}= B_{\scriptscriptstyle {q_{j}}, {1}} \, + \, \varsigma_{\scriptscriptstyle{{i} , 1}} \, \mathbf{s}_{2} (u) $. Using the equality \eqref{ec: bn1 en qj}, and proceeding analogously to the case b), we get $\hat{B}_{\scriptscriptstyle {q_{j} , 1} } = b_{\scriptscriptstyle{q_{j} , 1}}$.
\end{itemize}
Hence, defining $\mathcal{H}^{1}:=\mathbf{G}_{1}$, the theorem for $k = 1$ is proved.

\emph{Step of induction}. We now prove the step of induction; we first show that it is possible to modify the coefficients $A_{[\centerdot]\,, \scriptscriptstyle{k+1}}$ without changing neither the normalized coefficients of $A_{[\centerdot]}$ and $B_{[\centerdot]}$ of order less or equal than $k$, nor the coefficients $B_{[\centerdot]\,, \scriptscriptstyle{k+1}}$ of  order $k+1$. 
Next, we show how to modify the coefficients of $B_{[\centerdot]\,, \scriptscriptstyle{k+1}}$, without changing the coefficients of $A_{[\centerdot]}$ and $B_{[\centerdot]}$ less of equal than $k$, and the coefficients  $A_{[\centerdot]\,, \scriptscriptstyle{k+1}}$ of order  $k+1$ of the first coordinate. We will use Lemma \ref{lem: series TN por ecFac en pi} and Lemma \ref{lem: series TN por ecFac en qj}.

We now assume that the theorem takes place for $k \geqslant 1$. That is to say, there exists a tangent to the identity biholomorphim  $\mathcal{H}^{k} \colon (\C^{2}, 0) \rightarrow (\C^{2}, 0)$ that preserves the $y$--axis, and such that the foliation pair  $\bigl(\mathcal{H}^{k}(\F), \mathcal{H}^{k} (\G)\bigl)$ has normalizing transformations 
\begin{equation}\label{ec: Hk(H.)}
    \tilde{\mathcal{H}}^{k} \circ H_{[\centerdot]} = (\mathcal{A}_{[\centerdot]}, \mathcal{B}_{[\centerdot]})\,,
\end{equation}
 with power series expansion given by
\begin{equation}\label{ec: expansion de Hk(H.)}
\Bigl(\sum_{r = 1}^{k} a_{\scriptscriptstyle [\centerdot] ,\, r} x^{r}  
 +  \mathcal{A}_{[\centerdot] ,\scriptscriptstyle {k+1}} x^{k+1}
 +  \text{O} (x^{k+2}), 
\, u
 + \sum_{r = 1}^{k} b_{\scriptscriptstyle [\centerdot] ,\, r} x^{r}
 +  \mathcal{B}_{[\centerdot]\,,\scriptscriptstyle {k+1}} x^{k+1}
 +  \text{O} (x^{k+2})\Bigl) \, ,
\end{equation}
where $[\centerdot]$ represents $\scriptscriptstyle \mathbf{n}$, $p_{i}$, or $q_{j}$. 

The coefficients $ \mathcal{A}_{p_{i} ,\scriptscriptstyle {k+1}}$ , $ \mathcal{A}_{q_{j} ,\scriptscriptstyle {k+1}}$  in the power series expansion (\ref{ec: expansion de Hk(H.)}) satisfy the relations $(\ref{eq: Api,2})$ and $(\ref{eq: Aqj,2})$ and, therefore, the equalities 
\begin{equation*}
\begin{split}
\tfrac{ \mathcal{A}_{\scriptscriptstyle {p_{i}} , {k+1}}}{( \varepsilon_{\scriptscriptstyle{i, 1}} )^{k+1}} 
\, - \,   \mathcal{A}_{\scriptscriptstyle \mathbf{n}, {k+1}} 
& = \tfrac{ \varepsilon_{\scriptscriptstyle{i, k+1}} }{( \varepsilon_{\scriptscriptstyle{i, 1}} )^{k+1}} 
\, + \, \tfrac{1}{( \varepsilon_{\scriptscriptstyle{i, 1}} )^{k+1}} \mathbf{A}_{\mathbf{\p_{i}},\scriptscriptstyle{k+1}} \, , 
 \\
\tfrac{ \mathcal{A}_{\scriptscriptstyle {q_{j}} , {k+1}}}{( \varsigma_{\scriptscriptstyle{j, 1}} )^{k+1}} 
\, - \,    \mathcal{A}_{\scriptscriptstyle \mathbf{n}, {k+1}} 
& = \tfrac{ \varsigma_{\scriptscriptstyle{j, k+1}} }{( \varsigma_{\scriptscriptstyle{j, 1}} )^{k+1}} 
\, + \, \tfrac{1}{( \varsigma_{\scriptscriptstyle{j, 1}} )^{k+1}} \mathbf{A}_{\mathbf{\q_{j}},\scriptscriptstyle{k+1}}\,, 
\end{split}
\tag{$A^{\scriptscriptstyle{k+1}}$}
\end{equation*}
where $\mathbf{A}_{\mathbf{\p_{i}},\scriptscriptstyle{k+1}}$ is as in (\ref{eq: Api,3}) of Lemma \ref{lem: series TN por ecFac en pi}, and $\mathbf{A}_{\mathbf{\q_{j}},\scriptscriptstyle{k+1}}$ is as in  (\ref{eq: Aqj,3}) of Lemma \ref{lem: series TN por ecFac en qj}, are satisfied.

Analogously, by (\ref{eq: Bpi,2}), (\ref{eq: Bqj,2}), the coefficients  $ \mathcal{B}_{p_{i} ,\scriptscriptstyle {k+1}}$ , $ \mathcal{B}_{q_{j} ,\scriptscriptstyle {k+1}}$ in the power series expansion (\ref{ec: expansion de Hk(H.)}) satisfy 
\begin{equation*}
\begin{split}
\tfrac{ \mathcal{B}_{\scriptscriptstyle {p_{i}} , {k+1}}}{( \varepsilon_{\scriptscriptstyle{i, 1}} )^{k+1}} 
\, - \,   \mathcal{B}_{\scriptscriptstyle \mathbf{n}, {k+1}} 
& = \tfrac{ s_{\scriptscriptstyle{i, k+1}}}{( \varepsilon_{\scriptscriptstyle{i, 1}} )^{k+1}} 
\, + \, \tfrac{ \varepsilon_{\scriptscriptstyle{i, k+1}} }{( \varepsilon_{\scriptscriptstyle{i, 1}} )^{k+1}} \, b_{\scriptscriptstyle{\mathbf{n}, 1}}
\, + \,\tfrac{\mathbf{B}_{\mathbf{\p_{i}},\scriptscriptstyle{k+1}}}  \, , 
 \\
\tfrac{ \mathcal{B}_{\scriptscriptstyle {q_{j}} , {k+1}}}{( \varsigma_{\scriptscriptstyle{j, 1}} )^{k+1}} 
\, - \,    \mathcal{B}_{\scriptscriptstyle \mathbf{n}, {k+1}} 
& = \tfrac{z_{\scriptscriptstyle{j, k+1}}}{( \varsigma_{\scriptscriptstyle{j, 1}} )^{k+1} g_{j}^{(1)}}
\, + \, \tfrac{ \varsigma_{\scriptscriptstyle{j, k+1}} }{( \varsigma_{\scriptscriptstyle{j, 1}} )^{k+1}} \, b_{\scriptscriptstyle \mathbf{n}, 1}
\, + \, \tfrac{\mathbf{B}_{\mathbf{\q_{j}},\scriptscriptstyle{k+1}}}{( \varsigma_{\scriptscriptstyle{j, 1}} )^{k+1}}  \, ,
\end{split}
\tag{$B^{\scriptscriptstyle{k+1}}$}
\end{equation*}
where $\mathbf{B}_{\mathbf{\p_{i}},\scriptscriptstyle{k+1}}$ is as in (\ref{eq: Bpi,3}) of Lemma \ref{lem: series TN por ecFac en pi}, and $\mathbf{B}_{\mathbf{\q_{j}},\scriptscriptstyle{k+1}}$ is as in  (\ref{eq: Bqj,3}) of Lemma \ref{lem: series TN por ecFac en qj}.
\begin{remark}
    The relations ($A^{\scriptscriptstyle{k+1}}$), ($B^{\scriptscriptstyle{k+1}}$) are defined and holomorphic in the corresponding domains  $D_{p_{i}} \smallsetminus \{(0, p_{i})\}$ and $D_{q_{j}} \smallsetminus \{(0, q_{j})\}$.
\end{remark}

By Proposition \ref{Prop:soluciones canonicas} the transformations $a_{\scriptscriptstyle \mathbf{n}, \scriptscriptstyle{k+1} }$, $a_{\scriptscriptstyle {p_{i}} , \scriptscriptstyle{k+1} }$, $a_{\scriptscriptstyle {q_{j}} , \scriptscriptstyle{k+1} }$ satisfy, by definition, the equalities ($A^{\scriptscriptstyle{k+1}}$), while the transformations $b_{\scriptscriptstyle \mathbf{n}, \scriptscriptstyle{k+1} }$, $b_{\scriptscriptstyle {p_{i}} , \scriptscriptstyle{k+1} }$, $b_{\scriptscriptstyle {q_{j}} , \scriptscriptstyle{k+1} }$ satisfy as well the equalities ($B^{\scriptscriptstyle{k+1}}$); therefore taking the differences we get,
\begin{equation}
\label{Ec_an_An}
a_{\scriptscriptstyle \mathbf{n}, \scriptscriptstyle{k+1} } \, - \mathcal{A}_{\scriptscriptstyle \mathbf{n}, {k+1}} 
 = \tfrac{a_{\scriptscriptstyle {p_{i}} , \scriptscriptstyle{k+1} }}{( \varepsilon_{\scriptscriptstyle{i, 1}} )^{k+1}} 
\, - \,  \tfrac{ \mathcal{A}_{\scriptscriptstyle {p_{i}} , {k+1}}}{( \varepsilon_{\scriptscriptstyle{i, 1}} )^{k+1}} \, ,
\hspace{0.5cm}
a_{\scriptscriptstyle \mathbf{n}, \scriptscriptstyle{k+1} }
\, - \,   \mathcal{A}_{\scriptscriptstyle \mathbf{n}, {k+1}}
 = \tfrac{a_{\scriptscriptstyle {q_{j}} , \scriptscriptstyle{k+1} }}{( \varsigma_{\scriptscriptstyle{j, 1}} )^{k+1}} 
\, - \,  \tfrac{ \mathcal{A}_{\scriptscriptstyle {q_{j}} , {k+1}}}{( \varsigma_{\scriptscriptstyle{j, 1}} )^{k+1}}
\end{equation}
and
\begin{equation}
\label{Ec_bn_Bn}
b_{\scriptscriptstyle \mathbf{n}, \scriptscriptstyle{k+1} } 
\, - \,  \mathcal{B}_{\scriptscriptstyle \mathbf{n}, {k+1}} 
= \tfrac{b_{\scriptscriptstyle {p_{i}} , \scriptscriptstyle{k+1} }}{( \varepsilon_{\scriptscriptstyle{i, 1}} )^{k+1}} 
\, - \,  \tfrac{ \mathcal{B}_{\scriptscriptstyle {p_{i}} , {k+1}}}{( \varepsilon_{\scriptscriptstyle{i, 1}} )^{k+1}} \,,
\hspace{0.5cm}
b_{\scriptscriptstyle \mathbf{n}, \scriptscriptstyle{k+1} }
\, - \,   \mathcal{B}_{\scriptscriptstyle \mathbf{n}, {k+1}}
= \tfrac{b_{\scriptscriptstyle {q_{j}} , \scriptscriptstyle{k+1} }}{( \varsigma_{\scriptscriptstyle{j, 1}} )^{k+1}} 
\, - \,  \tfrac{ \mathcal{B}_{\scriptscriptstyle {q_{j}} , {k+1}}}{( \varsigma_{\scriptscriptstyle{j, 1}} )^{k+1}}  \, .
\end{equation}
By reasoning in a similar way to what was done in the case k=1, it follows that the differences $a_{\scriptscriptstyle \mathbf{n}, \scriptscriptstyle{k+1} } -  \mathcal{A}_{\scriptscriptstyle \mathbf{n}, {k+1}}$ and $b_{\scriptscriptstyle \mathbf{n}, \scriptscriptstyle{k+1} } -  \mathcal{B}_{\scriptscriptstyle \mathbf{n}, {k+1}}$ may be extended analytically to the whole plane $\C$. 

Note that the transformations $a_{\scriptscriptstyle \mathbf{n}, \scriptscriptstyle{k+1} }$ and $b_{\scriptscriptstyle \mathbf{n}, \scriptscriptstyle{k+1} }$ are holomorphic in $u = \infty$, while $ \mathcal{A}_{\scriptscriptstyle \mathbf{n}, {k+1}}$ has, at $u=\infty$, at most a pole of order $k$ and $\mathcal{B}_{\scriptscriptstyle \mathbf{n}, {k+1}}$ has at most a pole of order $k+2$  (since $\tilde{\mathcal{H}}^{k} \circ \mathbf{H}_{\scriptscriptstyle \mathbf{n}} $ preserves the line $\{v = 0\}$, where $v = x/y$). Hence,
the difference 
$a_{\scriptscriptstyle \mathbf{n}, \scriptscriptstyle{k+1} } -   \mathcal{A}_{\scriptscriptstyle \mathbf{n}, {k+1}}$ must be a polynomial $\mathbf{r}_{k} \in \C[u]$ of degree less or equal to  $k$, while the difference $b_{\scriptscriptstyle \mathbf{n}, \scriptscriptstyle{k+1} } -  \mathcal{B}_{\scriptscriptstyle \mathbf{n}, {k+1}}$ must be a polynomial $\mathbf{s}_{k+2} \in \C[u]$ of degree less or equal to $k+2$. 

Together with $\mathbf{r}_{k}$ and $\mathbf{s}_{k+2}$ we define  in $(\C^{2},0)$ the tangent to the identity biholomorphisms $\mathbf{F}_{k+1}$ and $\mathbf{G}_{k+1}$,
\begin{equation*}
\begin{split}
\mathbf{F}_{k+1} (x, y) 
& = \bigl(x \, + \, x^{k+1} \mathbf{r}_{k} (y/x), 
\, \, y \, + \, y x^{k} \mathbf{r}_{k} (y/x) \bigl) \, , 
\\ \mathbf{G}_{k+1} (x, y) 
& = \bigl(x , \, y \, + \, x^{k+2} \mathbf{s}_{k+2} (y/x) \bigl) \, .
\end{split}
\end{equation*}

Straight--forward calculations show that in the coordinates $(x,u)$ of the blow--up the biholomorphism $\mathbf{F}_{k+1}$ takes the form  
\begin{equation*}
    \tilde{\mathbf{F}}_{k+1}(x,u)=(x + x^{k+1} \mathbf{r}_{k} (u), u)\,,
\end{equation*}
and the biholomorphism $\mathbf{G}_{k+1}$ takes the form

\begin{equation*}
    \tilde{\mathbf{G}}_{k+1}(x, u)= (x, u + x^{k+1}\mathbf{s}_{k+2} (u))\,.
\end{equation*}
Note that the second coordinate of $\tilde{\mathbf{F}}_{k+1}$ is the identity, while the first coordinate of $ \tilde{\mathbf{G}}_{k+1}$ is the identity as well.
Therefore, the expression of the composition $\mathbf{G}_{k+1} \circ \mathbf{F}_{k+1}$ in the coordinates $(x, u)$ is given by
\begin{equation*}
\begin{split}
 \tilde{\mathbf{G}}_{k+1} \circ \tilde{\mathbf{F}}_{k+1} (x, u) 
= &\, \tilde{\mathbf{G}}_{k+1} (x + x^{k+1} \mathbf{r}_{k} (u), u) 
\\  = &\, \bigl(x + x^{k+1} \mathbf{r}_{k} (u), u + (x + x^{k+1} \mathbf{r}_{k} (u))^{k+1}\mathbf{s}_{k+2} (u) \bigl) \, .
\end{split}
\end{equation*}
We define 
\begin{equation}\label{ec: Hk+1 definicion}
    \mathcal{H}^{k+1}:=\mathbf{G}_{k+1} \circ \mathbf{F}_{k+1} \circ \mathcal{H}^{k}\,.
\end{equation}
We stress that the biholomorphism $\mathcal{H}^{k+1}$ satisfies the conditions of the theorem for $k+1$. Namely, since $\mathbf{F}_{k+1}$, $\mathbf{G}_{k+1}$ are tangent to the identity and leave invariant the $y$--axis, then the biholomorphism $\mathcal{H}^{k+1}$ has these properties as well.
Hence, the biholomorphisms
\begin{equation*} 
    \tilde{\mathcal{H}}^{k+1} \circ H_{[\centerdot]} = (\hat{\mathcal{A}}_{[\centerdot]}, \hat{\mathcal{B}}_{[\centerdot]})\,,
\end{equation*}
with power series expansion 
\begin{equation}\label{ec: Ntransf Hk+1(F,G)}
(\hat{\mathcal{A}}_{[\centerdot]}, \hat{\mathcal{B}}_{[\centerdot]})(x,u) = \biggl(\sum_{r \geqslant 1} \hat{\mathcal{A}}_{[\centerdot]\, , \, r}\, x^{r}, 
u + \sum_{r \geqslant 1} \hat{\mathcal{B}}_{[\centerdot]\, , \, r}\, x^{r}\biggl) \, ,
\end{equation}
where $[\centerdot]$ equals to $\scriptscriptstyle \mathbf{n}$, $p_{i}$, $q_{j}$, are the normalizing transformations of the foliation pair $\bigl(\mathcal{H}^{k+1} (\F), \mathcal{H}^{k+1}(\G)\bigl)$, which belongs to $\mathcal{N}_{\p}(\mathbf{h})\times\mathcal{D}_{\q}(\mathcal{I})$.

By the induction assumption \eqref{ec: Hk(H.)}, $\tilde{\mathcal{H}}^{k} \circ H_{[\centerdot]} = (\mathcal{A}_{[\centerdot]}, \mathcal{B}_{[\centerdot]})$ and definition \eqref{ec: Hk+1 definicion}, we obtain
\begin{equation*}
(\hat{\mathcal{A}}_{[\centerdot]}, \hat{\mathcal{B}}_{[\centerdot]}) 
\ = \
\tilde{\mathcal{H}}^{k+1} \circ H_{[\centerdot]}
\ = \ 
\tilde{\mathbf{G}}_{k+1} \circ \tilde{\mathbf{F}}_{k+1} \circ (\mathcal{A}_{[\centerdot]}, \mathcal{B}_{[\centerdot]}) \, ,
\end{equation*} 
Therefore, these expansions are equal to
\begin{equation*}
\begin{split}
(\hat{\mathcal{A}}_{[\centerdot]}, \hat{\mathcal{B}}_{[\centerdot]}) = & \Bigl(\mathcal{A}_{[\centerdot]} 
+ (\mathcal{A}_{[\centerdot]})^{k+1} \mathbf{r}_{k} (\mathcal{B}_{[\centerdot]}), 
\, \mathcal{B}_{[\centerdot]} 
+ (\mathcal{A}_{[\centerdot]} + \bigl(\mathcal{A}_{[\centerdot]})^{k+1} \mathbf{r}_{k} (\mathcal{B}_{[\centerdot]})\bigl)^{k+1}\mathbf{s}_{k+2} (\mathcal{B}_{[\centerdot]}) \Bigl) 
\\ = &  \Bigl(\mathcal{A}_{[\centerdot]} 
+ (a_{\scriptscriptstyle{[\centerdot], 1}} x + \text{O}(x^{2}))^{k+1} \mathbf{r}_{k} (u + \text{O}(x)), 
\, \mathcal{B}_{[\centerdot]} 
+ (a_{\scriptscriptstyle{[\centerdot], 1}} x + \text{O}(x^{2}))^{k+1}\mathbf{s}_{k+2} (u + \text{O}(x)) \Bigl) 
\\ = & \Bigl(\mathcal{A}_{[\centerdot]} 
+ (a_{\scriptscriptstyle{[\centerdot], 1}})^{k+1} x^{k+1} \mathbf{r}_{k} (u) + \text{O}(x^{k+2}), 
\, \mathcal{B}_{[\centerdot]} 
+ (a_{\scriptscriptstyle{[\centerdot], 1}})^{k+1}x^{k+1}\mathbf{s}_{k+2} (u) + \text{O}(x^{k+2}) \Bigl)  \, ,
\end{split}
\end{equation*}
where $a_{\scriptscriptstyle{\mathbf{n}, 1}} = 1$, $a_{\scriptscriptstyle{p_{i}, 1}} = \varepsilon_{\scriptscriptstyle{{i}, 1}}$, and $a_{\scriptscriptstyle{q_{j}, 1}} = \varsigma_{\scriptscriptstyle{{j}, 1}}$ (Proposition \ref{Prop:soluciones canonicas}).
Thus, considering the power series expansion \eqref{ec: Ntransf Hk+1(F,G)}, the following equalities for the coefficients take place 
\begin{equation}
\label{Ec_HatA_A}
\hat{\mathcal{A}}_{[\centerdot]\, , \, r} = \mathcal{A}_{[\centerdot]\, , \, r} \, , 
\hspace{0.5cm} 1 \leqslant r \leqslant k \, ; 
\hspace{1cm}
 \hat{\mathcal{A}}_{[\centerdot]\, , \, k+1}= \mathcal{A}_{[\centerdot] ,\scriptscriptstyle {k+1}} + (a_{\scriptscriptstyle{[\centerdot], k+1}})^{k+1}\mathbf{r}_{k}
\end{equation}
and
\begin{equation}
\label{Ec_HatB_B}
\hat{\mathcal{B}}_{[\centerdot]\, , \, r} = \mathcal{B}_{[\centerdot]\, , \, r} \, , 
\hspace{0.5cm} 1 \leqslant r \leqslant k \, ; 
\hspace{1cm}
\hat{\mathcal{B}}_{[\centerdot]\, , \, k+1}=\mathcal{B}_{[\centerdot],\scriptscriptstyle {k+1}} + (a_{\scriptscriptstyle{[\centerdot], k+1}})^{k+1}\mathbf{s}_{k+2}\,.
\end{equation}
By the induction assumption (see (\ref{ec: Hk(H.)}) and (\ref{ec: expansion de Hk(H.)})), $\hat{\mathcal{A}}_{[\centerdot]\, , \, r} = a_{\scriptscriptstyle [\centerdot] ,\, r}$, $\hat{\mathcal{B}}_{[\centerdot]\, , \, r} = b_{\scriptscriptstyle [\centerdot] ,\, r}$, for $[\centerdot]$ equal to $\scriptscriptstyle \mathbf{n}$, $p_{i}$, $q_{j}$, and $1 \leqslant r \leqslant k$. 

Now we will prove that $\hat{\mathcal{A}}_{[\centerdot]\, , \, k+1} = a_{\scriptscriptstyle [\centerdot] ,\, k+1}$ and $\hat{\mathcal{B}}_{[\centerdot]\, , \, k+1} = b_{\scriptscriptstyle [\centerdot] ,\, k+1}$, for $[\centerdot]$ equal to $\scriptscriptstyle \mathbf{n}$, $p_{i}$, $q_{j}$.

Namely, since 
$a_{\scriptscriptstyle \mathbf{n}, \scriptscriptstyle{k+1} } -   \mathcal{A}_{\scriptscriptstyle \mathbf{n}, {k+1}} = \mathbf{r}_{k}$, using \eqref{Ec_HatA_A} we have
\begin{equation*}
\hat{\mathcal{A}}_{\mathbf{n}\, , \, k+1}
= \mathcal{A}_{\mathbf{n} ,\scriptscriptstyle {k+1}} + \mathbf{r}_{k}
= a_{\scriptscriptstyle \mathbf{n}, \scriptscriptstyle{k+1}} \, ;
\end{equation*}  
and since $a_{\scriptscriptstyle{p_{i}, 1}} = \varepsilon_ {\scriptscriptstyle{{i}, 1}}$, $a_{\scriptscriptstyle{q_{j}, 1}} = \varsigma_ {\scriptscriptstyle{{j}, 1}}$, then 
\begin{equation*}
\hat{\mathcal{A}}_{\scriptscriptstyle{p_{i}\, , \, k+1}}= \mathcal{A}_{\scriptscriptstyle{p_{i}\, , \, k+1}} + (\varepsilon_ {\scriptscriptstyle{{i}, 1}})^{k+1}\mathbf{r}_{k} \, ,
\hspace{0.5cm}
\hat{\mathcal{A}}_{\scriptscriptstyle{q_{j}\, , \, k+1}}= \mathcal{A}_{\scriptscriptstyle{q_{j}\, , \, k+1}} + (\varsigma_ {\scriptscriptstyle{{j}, 1}})^{k+1}\mathbf{r}_{k} \, .
\end{equation*}
On the other hand, by the equalities \eqref{Ec_an_An}
\begin{equation*}
a_{\scriptscriptstyle {p_{i}} , \scriptscriptstyle{k+1} }
\, - \,  \mathcal{A}_{\scriptscriptstyle {p_{i}} , {k+1}}
\, = \, 
(\varepsilon_{\scriptscriptstyle{i, 1}} )^{k+1} \mathbf{r}_{k} \, ,
\hspace{1cm}
a_{\scriptscriptstyle {q_{j}} , \scriptscriptstyle{k+1} }
\, - \,  \mathcal{A}_{\scriptscriptstyle {q_{j}} , {k+1}}
\, = \, 
(\varsigma_{\scriptscriptstyle{j, 1}} )^{k+1} \mathbf{r}_{k}\,.
\end{equation*}
Therefore, $\hat{\mathcal{A}}_{\scriptscriptstyle{p_{i}\, , \, k+1}} = a_{\scriptscriptstyle {p_{i}} , \scriptscriptstyle{k+1} }$ and $\hat{\mathcal{A}}_{\scriptscriptstyle{q_{j}\, , \, k+1}} = a_{\scriptscriptstyle {q_{j}} , \scriptscriptstyle{k+1} }$.

We get $\hat{\mathcal{B}}_{[\centerdot]\, , \, k+1} = b_{\scriptscriptstyle [\centerdot] ,\, k+1}$, for $[\centerdot]$ equal to $\scriptscriptstyle \mathbf{n}$, $p_{i}$, $q_{j}$, proceeding analogously, considering $b_{\scriptscriptstyle \mathbf{n}, \scriptscriptstyle{k+1} } -   \mathcal{B}_{\scriptscriptstyle \mathbf{n}, {k+1}} = \mathbf{s}_{k+2}$, and the equalities \eqref{Ec_bn_Bn} and \eqref{Ec_HatB_B}.

Therefore, the normalizing transformations $\tilde{\mathcal{H}}^{k+1} \circ H_{[\centerdot]}$, for $[\centerdot]$ equal to $\scriptscriptstyle \mathbf{n}$, $p_{i}$, $q_{j}$, of the foliation pair $\bigl(\mathcal{H}^{k+1} (\F), \mathcal{H}^{k+1}(\G)\bigl)$, satisfy the form $(\mathbf{nt_{k + 1}})$. In this way Theorem \ref{Teorema_Limpieza de TransNorm_Intro} is proved.

 \end{proof}

\begin{remark}
\label{Obs_Aplicaciones_ak+1_bk+1_como funcionales}
By Proposition \ref{Prop:soluciones canonicas} we know that the transformations $a_{\scriptscriptstyle \mathbf{n}, \scriptscriptstyle{k+1} }$ defined in ($a^{k+1}$) depend on the complex numbers $\lambda_{\imb}$, $s_{\scriptscriptstyle{{\imb} , r}}$, $z_{\scriptscriptstyle{{\jmb} , r}}$, $\imb = 1,\ldots,n+1$, on the function $g_{\jmb}$ (local information of the pair foliation $(\tilde{\F}, \tilde{\G})$ at each tangency point $(0,q_{\jmb})$, $\jmb = 1,\ldots,m$), and on the coefficients $\varepsilon_{\scriptscriptstyle{{\imb} , r}}$, $\varsigma_{\scriptscriptstyle{{\jmb} , r}}$ of $\varepsilon_{\imb}$, $\varsigma_{\jmb}$ (local information of the non dicritical foliation $\tilde{\F}_{\scriptscriptstyle \mu}$), $1 \leqslant r \leqslant k$, depending on all the singular points $(0, p_{\imb})$, $\imb = 1,\ldots,n+1$, and on all the tangency points   $(0, q_{\jmb})$, $\jmb = 1,\ldots,m$. 

 Analogously, by  Proposition \ref{Prop:soluciones canonicas} we know that the transformations $b_{\scriptscriptstyle \mathbf{n}, \scriptscriptstyle{k+1} }$ defined in ($b^{k+1}$) depend on the complex numbers $\lambda_{\imb}$, $s_{\scriptscriptstyle{{\imb} , r}}$, $z_{\scriptscriptstyle{{\jmb} , r}}$, $z_{\scriptscriptstyle{{\jmb} , k+1}}$, $\imb = 1,\ldots,n+1$, on the function $g_{\jmb}$ (local information of the pair foliation $(\tilde{\F}, \tilde{\G})$ at each tangency point $(0,q_{\jmb})$, $\jmb = 1,\ldots,m$), and on the coefficients $\varepsilon_{\scriptscriptstyle{{\imb} , r}}$, $\varsigma_{\scriptscriptstyle{{\jmb} , r}}$ of $\varepsilon_{\imb}$, $\varsigma_{\jmb}$ (local information of the non dicritical foliation $\tilde{\F}_{\scriptscriptstyle \mu}$), $1 \leqslant r \leqslant k$, depending on all the singular points $(0, p_{\imb})$, $\imb = 1,\ldots,n+1$, and on all the tangency points   $(0, q_{\jmb})$, $\jmb = 1,\ldots,m$. Therefore, the transformation 

\begin{equation*}
b_{\scriptscriptstyle \mathbf{n}, \scriptscriptstyle{k+1} } 
+ \sum_{\jmb =1}^{m} \Bigl(\tfrac{z_{\scriptscriptstyle{{\jmb} , k+1}}}{(\varsigma_{\scriptscriptstyle{{\jmb} , 1}})^{k+1} g_{\jmb}^{(1)}} \Bigl)_{\CP, q_{\jmb}}
\end{equation*}
depends on the aforementioned parameters, just as the following transformations do  
\begin{equation*}
\begin{split}
b_{\scriptscriptstyle {p_{i}} , \scriptscriptstyle{k+1} } 
\, - \,  s_{\scriptscriptstyle{i, k+1}}
\, - \,  \varepsilon_{\scriptscriptstyle{i, k+1}}  \, b_{\scriptscriptstyle{\mathbf{n}, 1}} 
\, + \, ( \varepsilon_{\scriptscriptstyle{i, 1}} )^{k+1} \sum_{\jmb =1}^{m} \Bigl(\tfrac{z_{\scriptscriptstyle{{\jmb} , k+1}}}{(\varsigma_{\scriptscriptstyle{{\jmb} , 1}})^{k+1} g_{\jmb}^{(1)}} \Bigl)_{\CP, q_{\jmb}} \, ,
\\ b_{\scriptscriptstyle {q_{j}} , \scriptscriptstyle{k+1} }
\, - \, \tfrac{z_{\scriptscriptstyle{j, k+1}}}{g_{j}^{(1)}}
\, - \,  \varsigma_{\scriptscriptstyle{j, k+1}}  \,  b_{\scriptscriptstyle \mathbf{n}, 1}
\, + \, ( \varsigma_{\scriptscriptstyle{j, 1}} )^{k+1} \sum_{\jmb = 1}^{m} \Bigl(\tfrac{z_{\scriptscriptstyle{{\jmb} , k+1}}}{(\varsigma_{\scriptscriptstyle{{\jmb} , 1}})^{k+1} g_{\jmb}^{(1)}} \Bigl)_{\CP, q_{\jmb}} \, .
\end{split}
\end{equation*}
\end{remark}

\section{\textbf{\emph{Parametrizations of curves of tangencies of foliation pairs}}}
\label{Parametrizaciones de polares}
\setcounter{equation}{0}
The aim of this section is to describe finite jets of the parametrizations of the curves of tangencies related to foliation pairs in $(\C^2,0)$, up to coordinate changes tangent to the identity. Such a description is done by using the local analytic representatives (local models) of the foliation pair with respect to $(\F_{\scriptscriptstyle \mathbf{\mu}}, \G_{\scriptscriptstyle \mathbf{r}})$. This was briefly stated in Theorem \ref{Teorema_Parametrizacion_Polar_pi_qj_Intro} and is now precisely stated and proved in Theorem \ref{Teorema_Parametrizacion_Polar_pi_qj}. It is relevant to underline again that if we choose sufficiently high order jets, we will be able to determine the analytic type of such curves of tangencies.

In order to prove Theorem \ref{Teorema_Parametrizacion_Polar_pi_qj} we will use Theorem \ref{Teorema_Limpieza de TransNorm_Intro_1}. To this aim let $(\F, \G)$ be a foliation pair in $\mathcal{N}_{\p} (\mathbf{h}) \times \mathcal{D}_{\q} (\mathcal{I})$, and $(\tilde{\F}, \tilde{\G})$ its blow-up foliation. We denote by $\CP_{p_{i}} (\tilde{\F}, \tilde{\G})$  the \textit{curves of tangencies} of  $(\tilde{\F}, \tilde{\G})$ passing, respectively, through the singular points $(0, p_{i})$, $i=1,\dots,n+1$, and by $\CP_{q_{j}} (\tilde{\F}, \tilde{\G})$ those passing through the tangency points $(0, q_{j})$, $j=1,\dots,m$. The curves  $\CP_{p_{i}} (\tilde{\F}, \tilde{\G})$ and  $\CP_{q_{j}} (\tilde{\F}, \tilde{\G})$ are smooth branches having transversal intersection with the sphere $\mathcal{L}$,  and hence, they can be parametrized by the $x$ variable. By Remark \ref{Obs_Hipotesis del Teorema_Limpieza de TransNorm} we know that any foliation pair $(\F, \G)$ in $\mathcal{N}_{\p}(\mathbf{h}) \times \mathcal{D}_{\q}(\mathcal{I})$ is strictly analytically equivalent to one satisfying the conditions of Theorem \ref{Teorema_Limpieza de TransNorm_Intro_1}. Thus, for any $k\geqslant 1$ there exists a foliation pair in the class of strict analytic equivalence of foliation pair $(\F, \G)$, such that its normalizing transformations have power series as in ($\mathbf{nt_{k}}$). In what follows we will assume that such representant of the foliation pair $(\F, \G)$ is used.

 \begin{theorem}
\label{Teorema_Parametrizacion_Polar_pi_qj}
Let $(\F, \G) \in \mathcal{N}_{\p} (\mathbf{h}) \times \mathcal{D}_{\q} (\mathcal{I})$ be a foliation pair whose normalizing transformations with respect to the foliation pair $(\F_{\scriptscriptstyle \mathbf{\mu}}, \G_{\scriptscriptstyle \mathbf{r}})$, have power series expansion with $k_{0}$-jet ($\mathbf{nt_{k_{0}}}$) as in Theorem \ref{Teorema_Limpieza de TransNorm_Intro_1}. 

Let $\pi_{p_{i}}$ and $\pi_{q_{j}}$ be the parametrizations by    $x$ of the curves of tangencies $\CP_{p_{i}} (\tilde{\F}, \tilde{\G})$ and $\CP_{q_{j}} (\tilde{\F}, \tilde{\G})$, respectively. Then,
\begin{itemize}
    \item [a)]The coefficients of the power series expansion $\pi_{p_{i}} = p_{i} + \textstyle \sum_{r \geqslant 1}  \mathbf{c}_{\scriptscriptstyle{p_{i}, r}}\,x^{r}$\,  of the parametrizations by $x$ of the curve of tangencies $\CP_{p_{i}} (\tilde{\F}, \tilde{\G})$ satisfy,  for $1 \leqslant k \leqslant k_{0}$, 
\begin{equation*}
\mathbf{c}_{\scriptscriptstyle{p_{i}, k}}\, 
 = (1 - k\lambda_{i}) s_{\scriptscriptstyle{{i} , k}} 
\, - \,  \sum_{\jmb =1}^{m} \tfrac{z_{\scriptscriptstyle{{\jmb} , k}}}{2(p_{i}-q_{\jmb})} 
\,  + \,  \boldsymbol{\pi}_{\scriptscriptstyle{p_{i}\, , k}}  (p_{i})
 \, , 
 \end{equation*}
where $\boldsymbol{\pi}_{\scriptscriptstyle{p_{i}\, , k}}  (p_{i})$ represents the value at $p_{i}$ under the holomorphic transformation 

\begin{equation*}
   \boldsymbol{\pi}_{\scriptscriptstyle{p_{i}\, , k}}= \boldsymbol{\pi}_{\scriptscriptstyle{p_{i}\, , k}}  [\varepsilon_{\scriptscriptstyle{{i} , k}}; \, (\lambda_{\imb}, s_{\scriptscriptstyle{{\imb} , r}}, 
z_{\scriptscriptstyle{{\jmb} , r}}, g_{\jmb}; 
\varepsilon_{\scriptscriptstyle{{\imb} , r}}, \varsigma_{\scriptscriptstyle{{\jmb} , r}})_{\imb, \jmb}]_{r=1}^{k-1}\,.
\end{equation*}
The value $\boldsymbol{\pi}_{\scriptscriptstyle{p_{i}\, , k}} $ is defined in the disk $D_{p_{i}}$ and it is given by the evaluation of a functional transformation  $\boldsymbol{\pi}_{\scriptscriptstyle{p_{i}\, , k}} $ in  $\varepsilon_{\scriptscriptstyle{{i} , k}}$ (local information of  $\tilde{\F}_{\scriptscriptstyle \mu}$ at $(0,p_{i})$), as well as in the complex numbers $\lambda_{\imb}$, $s_{\scriptscriptstyle{{\imb} , r}}$, $z_{\scriptscriptstyle{{\jmb} , r}}$, the function  $g_{\jmb}$ (local information of the foliation pair $(\tilde{\F}, \tilde{\G})$), and the coefficients  $\varepsilon_{\scriptscriptstyle{{\imb} , r}}$, $\varsigma_{\scriptscriptstyle{{\jmb} , r}}$ of $\varepsilon_{\imb}$, $\varsigma_{\jmb}$ (local information of $\tilde{\F}_{\scriptscriptstyle \mu}$), for $1 \leqslant r \leqslant k-1$, and with respect to every singular and tangency point $(0, p_{\imb})$ and $(0, q_{\jmb})$(in particular, $\boldsymbol{\pi}_{\scriptscriptstyle{p_{i}, 1}} = 0$). 

\item[b)] The coefficients of the power series expansion $\pi_{q_{j}} = q_{j} + \textstyle \sum_{r \geqslant 1}  \mathbf{c}_{\scriptscriptstyle{q_{j}, r}}\, x^{r}$ of the parametrizations by $x$ of the curve of tangencies $\CP_{q_{j}} (\tilde{\F}, \tilde{\G})$ satisfy, for $1 \leqslant k \leqslant k_{0}$,  

\begin{equation*}
\hspace{1.0cm}\mathbf{c}_{\scriptscriptstyle{q_{j}, k}}\, 
 = - \tfrac{z_{\scriptscriptstyle{j, k}}}{2} \Bigl(k(\varsigma_{\scriptscriptstyle{j, 1}})^{(1)}(q_{j}) 
\, + \, \tfrac{g_{j}^{(3)}(q_{j})}{4}\Bigl)
\, - \, \sum_{\jmb \neq j}^{m} \tfrac{z_{\scriptscriptstyle{{\jmb} , k}}}{2(q_{j} - q_{\jmb})} 
\, + \, \boldsymbol{\pi}_{\scriptscriptstyle{q_{j}\, , k}}  (q_{j}) \,,
\end{equation*}
where $\boldsymbol{\pi}_{\scriptscriptstyle{q_{j}\, , k}} (q_{j})$ represents the value at $q_{j}$ under the holomorphic transformation 
\begin{equation*}
    \boldsymbol{\pi}_{\scriptscriptstyle{q_{j}\, , k}}=\boldsymbol{\pi}_{\scriptscriptstyle{q_{j}\, , k}}  [\varsigma_{\scriptscriptstyle{ j \, , k}}; \, (\lambda_{\imb}, s_{\scriptscriptstyle{{\imb} , r}}, 
z_{\scriptscriptstyle{{\jmb} , r}}, g_{\jmb}; 
\varepsilon_{\scriptscriptstyle{{\imb} , r}}, \varsigma_{\scriptscriptstyle{{\jmb} , r}})_{\imb, \jmb}]_{r=1}^{k-1}\,.
\end{equation*}
The value $\boldsymbol{\pi}_{\scriptscriptstyle{q_{j}\, , k}}$ is defined in the disk $D_{q_{j}}$ and it is given by the evaluation of a functional transformation $\boldsymbol{\pi}_{\scriptscriptstyle{q_{j}\, , k}} $ in $\varsigma_{\scriptscriptstyle{ j \, , k}}$ (local information of $\tilde{\F}_{\scriptscriptstyle \mu}$ at $(0,q_{j})$), the complex numbers $\lambda_{\imb}$, $s_{\scriptscriptstyle{{\imb} , r}}$, $z_{\scriptscriptstyle{{\jmb} , r}}$, the function  $g_{\jmb}$, and the coefficients $\varepsilon_{\scriptscriptstyle{{\imb} , r}}$, $\varsigma_{\scriptscriptstyle{{\jmb} , r}}$, for $1 \leqslant r \leqslant k-1$, and with respect to every singular and tangency point  $(0, p_{\imb})$ and $(0, q_{\jmb})$ (in particular, $\boldsymbol{\pi}_{\scriptscriptstyle{q_{j}, 1}} = 0$).
 \end{itemize}
 \end{theorem}

In order to prove Theorem \ref{Teorema_Parametrizacion_Polar_pi_qj} we need to delve deeper into the relationship between the curves of tangencies and the normalizing transformations that establish the link between the local analytical representatives and the original foliation pair. For this purpose we recall some notation and state two lemmas that will be used in the proof of the theorem.

\subsection{Preparation lemmas for the proof of Theorem \ref{Teorema_Parametrizacion_Polar_pi_qj}}

Let $(\F, \G)$ be a foliation pair in $\mathcal{N}_{\p} (\mathbf{h}) \times \mathcal{D}_{\q} (\mathcal{I})$, and let  $\mathbf{H}_{\scriptscriptstyle \mathbf{n}} $, $H_{p_{i}}$ and  $H_{q_{j}}$ be its normalizing transformations with respect to the foliation pair $(\F_{\scriptscriptstyle \mathbf{\mu}}, \G_{\scriptscriptstyle \mathbf{r}})$. By Theorem \ref{Teorema_Limpieza de TransNorm_Intro_1} we know that the $k_0$--jet of the respective transformations is given by 
$\tilde{\mathcal{H}}^{k_{0}} \circ \mathbf{H}_{\scriptscriptstyle \mathbf{n}}$, $\tilde{\mathcal{H}}^{k_{0}} \circ H_{p_{i}}$, $\tilde{\mathcal{H}}^{k_{0}} \circ H_{q_{j}}$. The power series of such transformations in the coordinates $(x,u)$ are of the form (\ref{Pageref_TN limpias_TNk}). Namely, 
\begin{equation*}
\tilde{\mathcal{H}}^{k_{0}} \circ H_{[\centerdot]}=\Bigl(\sum_{r = 1}^{k_{0}} a_{[\centerdot], r}(u) x^{r}  
\, + \, \text{O} (x^{k+1}), 
\, u \, + \, \sum_{r = 1}^{k_{0}} b_{[\centerdot], r}(u) x^{r}
\, + \, \text{O} (x^{k+1})\Bigl) \, ,
\tag{$\mathbf{nt_{k}}$}
\end{equation*}
where the subindex $[\centerdot]$ represents $\scriptscriptstyle \mathbf{n}$, $p_{i}$, or $q_{j}$; the transformations $a_{\scriptscriptstyle{\mathbf{n}, r}}$, $b_{\scriptscriptstyle{\mathbf{n}, r}}$ are holomorphic in $\mathcal{L} \setminus\{\hat{p}_{1},\ldots,\hat{p}_{n+1}, \hat{q}_{1},\ldots,\hat{q}_{m}\}$, and the local  transformations $a_{\scriptscriptstyle{p_{i} , r}}$, $b_{\scriptscriptstyle{q_{j} , r}}$ are holomorphic in the neighborhoods $(\mathcal{L},\hat{p}_{i})$ and $(\mathcal{L},\hat{q}_{j})$, respectively.

\begin{remark}
    We recall that, unless otherwise stated, we assume that the foliation pair  $(\F, \G)$ is the representant in the analytic class whose normalizing transformations $\mathbf{H}_{\scriptscriptstyle \mathbf{n}} $, $H_{p_{i}}$ and  $H_{q_{j}}$ have the power series ($\mathbf{nt_{k}}$).
\end{remark}\label{obs: nt y series ntk}
 We know that the local models $(\mathcal{F}_{p_{i}}^l, \G_{p_{i}})$ and $((x = \text{cst}), \G_{q_{j}})$ are transformed by the normalizing transformations $H_{p_{i}}$ and $H_{q_{j}}$ to the foliation pair $(\tilde{\F}, \tilde{\G})$. Straight-forward calculations show that the curves of tangencies of the local models are given by 
\begin{equation*}
\begin{split}
\CP_{p_{i}} (\mathcal{F}_{p_{i}}^l, \G_{p_{i}}) 
& = \{ u = p_{i} - \lambda_{i} x s_{i}^{(1)} (x)\} \, ,
\hspace{0.7cm} 
\CP_{q_{j}} ((x = \text{cst}), \G_{q_{j}}) 
= \{u = q_{j}\} \, .
\end{split}
\end{equation*}
These curves of tangencies are transformed by $H_{p_{i}}$ and $H_{q_{j}}$ to the corresponding curves of tangencies 

\begin{equation*}
    \CP_{p_{i}} (\tilde{\F}, \tilde{\G})= H_{p_{i}} (x, p_{i} + \tilde{s}_{i} (x))\,,
\end{equation*}
 where $\tilde{s}_{i} (x):= - \lambda_{i} x s_{i}^{(1)} (x)$ at the singular point, and 

\begin{equation*}
    \CP_{q_{j}} (\tilde{\F}, \tilde{\G}) = H_{q_{j}} (x, q_{j})\,.
\end{equation*}

The power series expansion for the polar curve $\CP_{p_{i}} (\tilde{\F}, \tilde{\G})$ at the singular point $(0,p_i)$ is given by

\begin{equation*}
\begin{split}
 (&\alpha_{p_{i}}(x), p_{i}+\beta_{p_{i}}(x)):=\, \\
\biggl(&\sum_{r = 1}^{k_{0}} a_{\scriptscriptstyle{p_{i} , r}} (p_{i} + \tilde{s}_{i} (x)) x^{r}, 
\, p_{i} + \tilde{s}_{i} (x) \, + \, \sum_{r = 1}^{k_{0}} b_{\scriptscriptstyle{p_{i} , r}}(p_{i} + \tilde{s}_{i} (x)) x^{r}\biggl) \, + \, \text{O} (x^{k_{0}+1}) \, 
\hspace{0.9cm}
\end{split} 
\tag{$\CP_{p_{i}}$}
\end{equation*}
and  power series expansion for the curve of tangencies $\CP_{q_{j}} (\tilde{\F}, \tilde{\G}) = H_{q_{j}} (x, q_{j})$  at the tangency point $(0,q_j)$ is given by  \label{Pageref_Pqj}
\begin{equation*}
(\alpha_{q_{j}}(x), q_{j}+ \beta_{q_{j}}(x)):=\biggl(\sum_{r = 1}^{k_{0}} a_{\scriptscriptstyle{{q_{j}} , r}}(q_{j}) x^{r}, 
\, q_{j} \, + \, \sum_{r = 1}^{k_{0}} b_{\scriptscriptstyle{q_{j} , r}}(q_{j}) x^{r}\biggl) \, + \, \text{O} (x^{k_{0}+1}) 
.\quad
\tag{$\CP_{q_{j}}$}
\end{equation*}

\begin{lemma}[Implicit parametrization of curves of tangencies at singular points]
\label{Lema_ParamImpli_Polar_pi}
Let $(\F, \G) \in \mathcal{N}_{\p} (\mathbf{h}) \times \mathcal{D}_{\q} (\mathcal{I})$ be a foliation pair whose normalizing transformations $\mathbf{H}_{\scriptscriptstyle \mathbf{n}} $, $H_{p_{i}}$, $H_{q_{j}}$ with respect to the foliation pair $(\F_{\scriptscriptstyle \mathbf{\mu}}, \G_{\scriptscriptstyle \mathbf{r}})$ have the $k_{0}$--jet of its power series expansion ($\mathbf{nt_{k_{0}}}$) given in Theorem \ref{Teorema_Limpieza de TransNorm_Intro_1}. Let $(\alpha_{p_{i}} (x), p_{i} + \beta_{p_{i}} (x))$ be the implicit parametrization of the curve of tancencies  $\CP_{p_{i}} (\tilde{\F}, \tilde{\G})$ as in ($\CP_{p_{i}}$). Then the corresponding power series expansions 
 $\alpha_{p_{i}} = \textstyle \sum_{r \geqslant 1} \alpha_{\scriptscriptstyle{p_{i} , r}} x^{r}$ and $\beta_{p_{i}} = \textstyle \sum_{r \geqslant 1} \beta_{\scriptscriptstyle{p_{i} , r}} x^{r}$ satisfy the following relations: 

\begin{itemize}
    \item [a)] $k=1$ the coefficients with respect to the monomial $x$ are
\begin{equation*}
\alpha_{\scriptscriptstyle{p_{i} , 1}} = 1 \, ,
\hspace{1cm}
\beta_{\scriptscriptstyle{p_{i} , 1}}
= (1 - \lambda_{i}) s_{\scriptscriptstyle{i,1}}  
-  \sum_{\jmb =1}^{m} \tfrac{z_{\scriptscriptstyle{{\jmb} , 1}}}{2(p_{i}-q_{\jmb})} \, .
\end{equation*}
\item [b)] For $2 \leqslant k \leqslant k_{0}$, the coefficients with respect to the monomial $x^k$ satisfy
\begin{equation*}
\begin{split}
& \alpha_{\scriptscriptstyle{p_{i} , k}}
= \mathbf{a}_{\scriptscriptstyle{p_{i} , k}}(p_{i})
 \, ,
\\
& \beta_{\scriptscriptstyle{p_{i} , k}} 
= (1 - k\lambda_{i}) s_{\scriptscriptstyle{{i} , k}} 
\, - \, \sum_{\jmb =1}^{m} \tfrac{z_{\scriptscriptstyle{{\jmb} , k}}}{2(p_{i}-q_{\jmb})} 
\,  + \, \mathbf{b}_{\scriptscriptstyle{p_{i} , k}}(p_{i})
\end{split}
\end{equation*}
\end{itemize}
where the holomorphic transformations 
\begin{equation*}
    \mathbf{a}_{\scriptscriptstyle{p_{i} , k}}=\mathbf{a}_{\scriptscriptstyle{p_{i} , k}}[\varepsilon_{\scriptscriptstyle{{i} , k}}; \, (\lambda_{\imb}, s_{\scriptscriptstyle{{\imb} , r}}, 
z_{\scriptscriptstyle{{\jmb} , r}}, g_{\jmb}; 
\varepsilon_{\scriptscriptstyle{{\imb} , r}}, \varsigma_{\scriptscriptstyle{{\jmb} , r}})_{\imb, \jmb}]_{r=1}^{k-1}\,
\end{equation*}
and
\begin{equation*}
   \mathbf{b}_{\scriptscriptstyle{p_{i} , k}} =\mathbf{b}_{\scriptscriptstyle{p_{i} , k}}[\varepsilon_{\scriptscriptstyle{{i} , k}}; \, (\lambda_{\imb}, s_{\scriptscriptstyle{{\imb} , r}}, 
z_{\scriptscriptstyle{{\jmb} , r}}, g_{\jmb}; 
\varepsilon_{\scriptscriptstyle{{\imb} , r}}, \varsigma_{\scriptscriptstyle{{\jmb} , r}})_{\imb, \jmb}]_{r=1}^{k-1}\, 
\end{equation*}
 are defined in the disk $D_{p_{i}}$, and are obtained by the evaluation of the functional transformations $\mathbf{a}_{\scriptscriptstyle{p_{i} , k}}$ and $\mathbf{b}_{\scriptscriptstyle{p_{i} , k}}$ in $\varepsilon_{\scriptscriptstyle{{i} , k}}$ (local information of $\tilde{\F}_{\scriptscriptstyle \mu}$ in $(0,p_{i})$), the complex numbers $\lambda_{\imb}$, $s_{\scriptscriptstyle{{\imb} , r}}$, $z_{\scriptscriptstyle{{\jmb} , r}}$, the function $g_{\jmb}$ (local information of the foliation pair $(\tilde{\F}, \tilde{\G})$), and the coefficients $\varepsilon_{\scriptscriptstyle{{\imb} , r}}$, $\varsigma_{\scriptscriptstyle{{\jmb} , r}}$ de $\varepsilon_{\imb}$, $\varsigma_{\jmb}$ (local information of $\tilde{\F}_{\scriptscriptstyle \mu}$), for $1 \leqslant r \leqslant k-1$, and with respect to every singularity  $(0, p_{\imb})\,, \, \imb = 1,\dots, n+1$ and every tangency point  $(0, q_{\jmb})\,, \,\jmb = 1,\dots, m$.
\end{lemma}


Analogously Lemma \ref{Lema_ParamImpli_Polar_qj} provides an implicit parametrization of curves of tangencies at the tangency points. Namely, the following lemma gives the relations that the coefficients of the $k_{0}$--jet of the power series expansion of the implicit parametrization of the curve of tangencies $ \CP_{q_{j}} (\tilde{\F}, \tilde{\G}) = H_{q_{j}} (x, q_{j})$ must satisfy

\begin{lemma}[Implicit parametrization of curves of tangencies at tangency points]
\label{Lema_ParamImpli_Polar_qj}
Let $(\F, \G) \in \mathcal{N}_{\p} (\mathbf{h}) \times \mathcal{D}_{\q} (\mathcal{I})$ be a foliation pair whose normalizing transformations $\mathbf{H}_{\scriptscriptstyle \mathbf{n}} $, $H_{p_{i}}$, $H_{q_{j}}$ with respect to the foliation pair $(\F_{\scriptscriptstyle \mathbf{\mu}}, \G_{\scriptscriptstyle \mathbf{r}})$ have the $k_{0}$--jet of its power series expansion ($\mathbf{nt_{k_{0}}}$) given in Theorem \ref{Teorema_Limpieza de TransNorm_Intro_1}. Let $(\alpha_{q_{j}} (x), q_{j} + \beta_{q_{j}} (x))$ be the implicit parametrization of the curve of tancencies  $\CP_{q_{j}} (\tilde{\F}, \tilde{\G})$ as in ($\CP_{q_{j}}$). Then the corresponding power series expansions 
 $\alpha_{q_{j}} = \textstyle \sum_{r \geqslant 1} \alpha_{\scriptscriptstyle{q_{j} , r}} x^{r}$ and $\beta_{q_{j}} = \textstyle \sum_{r \geqslant 1} \beta_{\scriptscriptstyle{q_{j} , r}} x^{r}$ satisfy the following relations:

 \begin{itemize}
     \item [a)] For $k=1$ the coefficients with respect to the monomial $x$ are
 
\begin{equation*}
\alpha_{\scriptscriptstyle{q_{j} , 1}} 
= 1 \, ,
\hspace{0.8cm}
\beta_{\scriptscriptstyle{q_{j} , 1}}
= - \tfrac{z_{\scriptscriptstyle{j,1}} }{2} 
\Bigl((\varsigma_{\scriptscriptstyle{j, 1}})^{(1)}(q_{j}) 
+ \tfrac{g_{j}^{(3)}(q_{j})}{4}\Bigl)
-  \sum_{\jmb \neq j}^{m} \tfrac{z_{\scriptscriptstyle{{\jmb} , 1}}}{2(q_{j}-q_{\jmb})} \, ;
\end{equation*}
\item[b)] For $2 \leqslant k \leqslant k_{0}$, the coefficients with respect to the monomial  $x^{k}$ satisfy
\begin{equation*}
\begin{split}
& \alpha_{\scriptscriptstyle{q_{j} , k}}
= \mathbf{a}_{\scriptscriptstyle{q_{j} , k}} (q_{j}) \, ,
\\ & \beta_{\scriptscriptstyle{q_{j} , k}} 
= - \tfrac{z_{\scriptscriptstyle{j, k}}}{2} \Bigl(k(\varsigma_{\scriptscriptstyle{j, 1}})^{(1)}(q_{j}) 
\, + \, \tfrac{g_{j}^{(3)}(q_{j})}{4}\Bigl)
\, - \, \sum_{\jmb \neq j}^{m} \tfrac{z_{\scriptscriptstyle{{\jmb} , k}}}{2(q_{j} - q_{\jmb})} 
\, + \, \mathbf{b}_{\scriptscriptstyle{q_{j} , k}} (q_{j})
 \, ,
\end{split}
\end{equation*}
\end{itemize}
where the holomorphic transformations 
\begin{equation*}
    \mathbf{a}_{\scriptscriptstyle{q_{j} , k}}=\mathbf{a}_{\scriptscriptstyle{q_{j} , k}}[\varsigma_{\scriptscriptstyle{ j \, , k}}; \, (\lambda_{\imb}, s_{\scriptscriptstyle{{\imb} , r}}, 
z_{\scriptscriptstyle{{\jmb} , r}}, g_{\jmb}; 
\varepsilon_{\scriptscriptstyle{{\imb} , r}}, \varsigma_{\scriptscriptstyle{{\jmb} , r}})_{\imb, \jmb}]_{r=1}^{k-1}\,
\end{equation*}
and
\begin{equation*}
   \mathbf{b}_{\scriptscriptstyle{q_{j} , k}} =\mathbf{b}_{\scriptscriptstyle{q_{j} , k}}[\varsigma_{\scriptscriptstyle{ j \, , k}}; \, (\lambda_{\imb}, s_{\scriptscriptstyle{{\imb} , r}}, 
z_{\scriptscriptstyle{{\jmb} , r}}, g_{\jmb}; 
\varepsilon_{\scriptscriptstyle{{\imb} , r}}, \varsigma_{\scriptscriptstyle{{\jmb} , r}})_{\imb, \jmb}]_{r=1}^{k-1}\, 
\end{equation*}
 are defined in the disk $D_{q_{j}}$, and are obtained by the evaluation of the functional transformations $\mathbf{a}_{\scriptscriptstyle{q_{j} , k}}$ and $\mathbf{b}_{\scriptscriptstyle{q_{j} , k}}$ in $\varsigma_{\scriptscriptstyle{ j \, , k}}$ (local information of $\tilde{\F}_{\scriptscriptstyle \mu}$ in $(0,q_{j})$), the complex numbers $\lambda_{\imb}$, $s_{\scriptscriptstyle{{\imb} , r}}$, $z_{\scriptscriptstyle{{\jmb} , r}}$, the function $g_{\jmb}$ (local information of the foliation pair $(\tilde{\F}, \tilde{\G})$), and the coefficients $\varepsilon_{\scriptscriptstyle{{\imb} , r}}$, $\varsigma_{\scriptscriptstyle{{\jmb} , r}}$ of $\varepsilon_{\imb}$, $\varsigma_{\jmb}$ (local information of $\tilde{\F}_{\scriptscriptstyle \mu}$), for $1 \leqslant r \leqslant k-1$, and with respect to every singularity  $(0, p_{\imb})\,, \, \imb = 1,\dots, n+1$, and every tangency point  $(0, q_{\jmb})\,, \,\jmb = 1,\dots, m$.
\end{lemma}

Since for the proof of Theorem \ref{Teorema_Parametrizacion_Polar_pi_qj} we deal with the power series expansions of the parametrizations by $x$ of the curve of tangencies $\CP_{p_{i}} (\tilde{\F}, \tilde{\G})$, we need to look to the derivatives of the compositions of several transformations. This is done by using the well known \emph{Fa\`a di Bruno formula}. For the sake of redability we state  here the corresponding theorem.

\begin{theorem}[Fa\`a di Bruno formula]
\label{Teorema_Formula de Faa di Bruno}
Let $f$ and $h$ be two holomorphic transformations defined in their corresponding domains in $\C$, and such that its composition $h \circ f$ is well defined and is complex valuated. Then the $k$-th derivative  $(h \circ f)^{(k)}$ of the composition $h \circ f$ has the expression 
\begin{equation*}
(h \circ f)^{(k)}=\sum \tfrac{k!}{r_{1}! \cdots r_{k}!} \, h^{(r)} \vert_{f} \, \left(\tfrac{f^{(1)}}{1!}\right)^{r_{1}} \, \cdots \, \left(\tfrac{f^{(k)}}{k!}\right)^{r_{k}} \, ,
\end{equation*}
where $r_{1}, \ldots, r_{k}$ are non-negative integer numbers satisfying $r_{1} + 2 r_{2} + \cdots + k r_{k} = k$ and where $r= r_{1} + r_{2} + \cdots + r_{k}$. 
\end{theorem}

\begin{remark}
\label{Obs_Sobre la suma en el Teo_FaaDiBruno}
As in Theorem \ref{Teorema_Formula de Faa di Bruno}, let $r_{1}, \ldots, r_{k}$ be non-negative integer numbers such that $r_{1} + 2 r_{2} + \cdots + k r_{k} = k$ and with $r=r_{1} + r_{2} + \cdots + r_{k}$. The following equivalence take place
\begin{equation*}
r_{k} \neq 0 
\hspace{0.3cm}
\text{if and only if}
\hspace{0.3cm}
r_{k} = 1 
\hspace{0.2cm} 
\text{and} 
\hspace{0.2cm} 
r_{1}= \cdots = r_{k-1} = 0\,,
\hspace{0.1cm}
\text{and this implies}
\hspace{0.1cm}
r = 1 \, .
\end{equation*}
Since $r \leqslant k$, the case $r = k$ takes place if and only if $r_{1} = k$ and $r_{2} = \cdots = r_{k} = 0$.
\end{remark}

We now define three polynomials that will be needed in the proof of Theorem \ref{Teorema_Parametrizacion_Polar_pi_qj}. Namely, for 
  $k \geqslant 1$, we define the polynomials $\mathbf{P}^{k}$, $\tilde{\mathbf{P}}^{k}$, $\hat{\mathbf{P}}^{k}$ in the variables  $\mathbf{w}_{1},\dots,\mathbf{w}_{s}$ and $\mathbf{z}_{1}\dots, \mathbf{z}_{s}$ and with coefficients in $\N$.  
\begin{equation}
\label{Ec_Definicion de Pk}
\begin{split}
\mathbf{P}^{k}  [\mathbf{w}_{s}; & \mathbf{z}_{s}]^{k}_{s =1} 
:= \sum \tfrac{k!}{r_{1}! \cdots r_{k}!} \, \mathbf{w}_{r} \, \left(\tfrac{\mathbf{z}_{1}}{1!}\right)^{r_{1}} \, \cdots \, \left(\tfrac{\mathbf{z}_{k}}{k!}\right)^{r_{k}}   \, ,
\\  \tilde{\mathbf{P}}^{k} & = \mathbf{P}^{k} - \mathbf{w}_{1} \mathbf{z}_{k} \, , 
\hspace{0.8cm}
\hat{\mathbf{P}}^{k} = \mathbf{P}^{k} - \mathbf{w}_{k} (\mathbf{z}_{1})^{k} \, ,
\end{split}
\end{equation}
where $r_{1}, \ldots, r_{k}$ satisfy $r_{1} + 2 r_{2} + \cdots + k r_{k} = k$, and $r=r_{1} + r_{2} + \cdots + r_{k}$ as in Remark \ref{Obs_Sobre la suma en el Teo_FaaDiBruno}. Note that  $\tilde{\mathbf{P}}^{1}=\hat{\mathbf{P}}^{1}=0$. This remark implies that the polynomial $\tilde{\mathbf{P}}^{k}$ depends on the variables $\mathbf{w}_{s+1}$, $\mathbf{z}_{s}$, with $1 \leqslant s \leqslant k-1$, i.e. $\tilde{\mathbf{P}}^{k} [\mathbf{w}_{s+1}; \mathbf{z}_{s}]^{k-1}_{s =1} $. It implies as well that the polynomial  $\hat{\mathbf{P}}^{k}$ depends on the  variables $\mathbf{w}_{s}$, $\mathbf{z}_{1}$, $\mathbf{z}_{s+1}$, with $1 \leqslant s \leqslant k-1$, i.e. $\hat{\mathbf{P}}^{k} [\mathbf{w}_{s};\mathbf{z}_{1}, \mathbf{z}_{s+1}]^{k-1}_{s =1}$.  

Fa\`a di Bruno formula for $k \geqslant 1$ is expressed in terms of polynomials   $\mathbf{P}^{k}$, $\tilde{\mathbf{P}}^{k}$, $\hat{\mathbf{P}}^{k}$, as follows
\begin{equation}
\label{Ec_Formula de FaaDiBruno con polinomios}
\begin{split}
(h \circ f)^{(k)} 
& = \mathbf{P}^{k}\bigl[h^{(s)} \vert_{f} \,  ; \, f^{(s)}\bigl]^{k}_{s=1}
= h^{(1)} \vert_{f} f^{(k)} 
\, + \, \tilde{\mathbf{P}}^{k}\bigl[h^{(s+1)} \vert_{f} \,  ; \, f^{(s)}\bigl]^{k-1}_{s=1} 
\\ & = h^{(k)}\vert_{f} (f^{(1)})^{k}
\, + \, \hat{\mathbf{P}}^{k} \bigl[h^{(s)} \vert_{f}; f^{(1)}, f^{(s+1)}\bigl]^{k-1}_{s =1} \, .
\end{split}
\end{equation}


\subsection{Proof of Theorem \ref{Teorema_Parametrizacion_Polar_pi_qj}}
\label{Demostracion_Teorema_Parametrizacion_Polar_pi_qj}
Let $\CP_{p_{i}} (\tilde{\F}, \tilde{\G})$ and $\CP_{q_{j}} (\tilde{\F}, \tilde{\G})$ be the curves of tangencies of the foliation pair  $(\tilde{\F}, \tilde{\G})$, and let $(\alpha_{p_{i}} (x), p_{i} + \beta_{p_{i}} (x))$ and $(\alpha_{q_{j}} (x), q_{j} + \beta_{q_{j}} (x))$ be their corresponding implicit parametrizations whose power series expansions are given in ($\CP_{p_{i}}$) and ($\CP_{q_{j}}$)).

Therefore, the parametrization by $x$ of the curve of tangencies  $\CP_{p_{i}} (\tilde{\F}, \tilde{\G})$ is given by the expression
\begin{equation*}
\pi_{p_{i}} 
\ = \ p_{i} + \beta_{p_{i}} \circ \alpha_{p_{i}}^{-1} (x) 
\ = \ p_{i} + \textstyle \sum_{r \geqslant 1}  \mathbf{c}_{\scriptscriptstyle{p_{i}, r}}\, x^{r}\,.
\end{equation*}
Analogously,  the parametrization by $x$ of the curve of tangencies
$\CP_{q_{j}} (\tilde{\F}, \tilde{\G})$ is given by the expression
\begin{equation*}
\pi_{q_{j}} 
\ = \ q_{j} + \beta_{q_{j}} \circ \alpha_{q_{j}}^{-1} (x) 
\ = \ q_{j} + \textstyle \sum_{r \geqslant 1}  \mathbf{c}_{\scriptscriptstyle{q_{j}, r}}\, x^{r}\,.
\end{equation*}

It is now clear that if we want to give an explicit expression of the coefficients of the power series expansion of  $\pi_{p_{i}}$ in terms of the power series expansions of $\alpha_{p_{i}}$ and $\beta_{p_{i}}$, and of the coefficients of $\pi_{q_{j}}$, in terms of the power series expansions of  $\alpha_{q_{j}}$ and $\beta_{q_{j}}$, we will need the expression introduced in Theorem \ref{Teorema_Formula de Faa di Bruno} of Fa\`a di Bruno formula.

First note that since the transformations $\alpha_{p_{i}}$ and $\alpha_{q_{j}}$ have identity linear part, $\alpha_{\scriptscriptstyle{p_{i} , 1}}=\alpha_{\scriptscriptstyle{q_{j} , 1}}=1$ then their inverses 
$\alpha_{p_{i}}^{-1} = \textstyle \sum_{r \geqslant 1} \tilde{\alpha}_{\scriptscriptstyle{{p_{i}} , r}} x^{r}$ and $\alpha_{q_{j}}^{-1} = \textstyle \sum_{r \geqslant 1} \tilde{\alpha}_{\scriptscriptstyle{{q_{j}} , r}} x^{r}$,  have identity linear part as well, $\tilde{\alpha}_{p_{i}}^{1}=\tilde{\alpha}_{\scriptscriptstyle{{q_{j}} , 1}}=1$. Moreover, for  $r \geqslant 2$, the corresponding coefficients are

\begin{itemize}
    \item [a)]$\tilde{\alpha}_{\scriptscriptstyle{{p_{i}} , r}}=\tfrac{-1}{r!}\,\hat{\mathbf{P}}^{r} [s! \tilde{\alpha}_{\scriptscriptstyle{{p_{i}} , s}}; 1, (s+1)!  \alpha_{\scriptscriptstyle{{p_{i}} , s+1}}]^{r-1}_{s =1}$\,,
    where $\hat{\mathbf{P}}^{r}$ is the polynomial (\ref{Ec_Definicion de Pk}) and it is evaluated in $\alpha_{\scriptscriptstyle{p_{i} , 1}}, \ldots, \alpha_{\scriptscriptstyle{p_{i} , r}}$.\\
    \item [b)] $\tilde{\alpha}_{\scriptscriptstyle{{q_{j}} , r}}=\tfrac{-1}{r!}\,\hat{\mathbf{P}}^{r} [s! \tilde{\alpha}_{\scriptscriptstyle{{q_{j}} , s}}; 1, (s+1)!  \alpha_{\scriptscriptstyle{{q_{j}} , s+1}}]^{r-1}_{s =1}$, where $\hat{\mathbf{P}}^{r}$ is again the polynomial (\ref{Ec_Definicion de Pk}) and it is evaluated in $\alpha_{\scriptscriptstyle{q_{j} , 1}}, \ldots, \alpha_{\scriptscriptstyle{q_{j} , r}}$.
\end{itemize}

By (\ref{Ec_Formula de FaaDiBruno con polinomios}) it can be shown that the coefficients of $x^{k}$ in the compositions $\beta_{p_{i}} \circ \alpha_{p_{i}}^{-1}$ and $\beta_{q_{j}} \circ \alpha_{q_{j}}^{-1}$ are, respectively, 

\begin{equation}\label{Pageref_Pik_i_j}
\mathbf{c}_{\scriptscriptstyle{p_{i}, k}}\,  
= \beta_{\scriptscriptstyle{p_{i} , k}} 
\, + \, \tfrac{1}{k!}\, \hat{\mathbf{P}}^{k} \bigl[r! \, \beta_{\scriptscriptstyle{p_{i} , r}} ; 1,
(r+1)! \, \tilde{\alpha}_{\scriptscriptstyle{{p_{i}} , r+1}}\bigl]^{k-1}_{r =1} \, , 
\end{equation}

\begin{equation}   \label{Pageref_Pik_i_j_bis} 
\mathbf{c}_{\scriptscriptstyle{q_{j}, k}}\,  
 = \beta_{\scriptscriptstyle{q_{j} , k}} 
\, + \, \tfrac{1}{k!} \,\hat{\mathbf{P}}^{k} \bigl[r! \, \beta_{\scriptscriptstyle{q_{j} , r}} ; 1,
(r+1)! \, \tilde{\alpha}_{\scriptscriptstyle{{q_{j}} , r+1}}\bigl]^{k-1}_{r =1} \, .
\end{equation}

Since the polynomial $\hat{\mathbf{P}}^{1}$ is identically zero, the coefficients $\mathbf{c}_{\scriptscriptstyle{p_{i}, 1}}$ and $\mathbf{c}_{\scriptscriptstyle{q_{j}, 1}}$ are given by
\begin{equation*}
\mathbf{c}_{\scriptscriptstyle{p_{i}, 1}} =\beta_{\scriptscriptstyle{p_{i} , 1}} \quad \text{and} \quad \mathbf{c}_{\scriptscriptstyle{q_{j}, 1}} = \beta_{\scriptscriptstyle{q_{j} , 1}}\,.
\end{equation*}

For $2 \leqslant k  \leqslant k_{0}$, Lemma \ref{Lema_ParamImpli_Polar_pi} implies that the coefficients $\beta_{\scriptscriptstyle{p_{i} , r}}$ and $\alpha_{\scriptscriptstyle{p_{i} , r+1}}$, for $1 \leqslant r \leqslant k-1$, are obtained by the evaluation at $p_{i}$ of holomorphic transformations that are, in turn, obtained from functional transformations that are evaluated in the holomorphic transformation $\varepsilon_{\scriptscriptstyle{{i} , k}}$ (local information depending on $\tilde{\F}_{\scriptscriptstyle \mu}$ at $(0,p_{i})$), the complex numbers $\lambda_{\imb}$, $s_{\scriptscriptstyle{{\imb} , l}}$, $z_{\scriptscriptstyle{{\jmb} , l}}$, the application $g_{\jmb}$ (local information depending on the foliation pair  $(\tilde{\F}, \tilde{\G})$), and on the coefficients $\varepsilon_{\scriptscriptstyle{{\imb} , l}}$, $\varsigma_{\scriptscriptstyle{{\jmb} , l}}$ of $\varepsilon_{\imb}$, $\varsigma_{\jmb}$ (local information depending on $\tilde{\F}_{\scriptscriptstyle \mu}$), for $1 \leqslant l \leqslant k-1$, with respect to all the singular points $(0, p_{\imb})$, $\imb=1,\dots, n+1$ and tangency points $(0, q_{\jmb})$, $\jmb=1, \dots, m$. 

Analogously, for $2 \leqslant k  \leqslant k_{0}$, Lemma \ref{Lema_ParamImpli_Polar_qj} implies that the coefficients $\beta_{\scriptscriptstyle{q_{j} , r}}$, $\alpha_{\scriptscriptstyle{q_{j} , r+1}}$, for $1 \leqslant r \leqslant k-1$, are obtained by the evaluation at the corresponding point $q_{j}$  of holomorphic transformations that are, in turn, obtained from functional transformations that are evaluated in the holomorphic transformation $\varsigma_{\scriptscriptstyle{ j \, , k}}$ (local information depending on $\tilde{\F}_{\scriptscriptstyle \mu}$ at the point $(0,q_{j})$), the complex numbers $\lambda_{\imb}$, $s_{\scriptscriptstyle{{\imb} , l}}$, $z_{\scriptscriptstyle{{\jmb} , l}}$, the function $g_{\jmb}$, and the coefficients $\varepsilon_{\scriptscriptstyle{{\imb} , l}}$, $\varsigma_{\scriptscriptstyle{{\jmb} , l}}$, for $1 \leqslant l \leqslant k-1$, with respect to all the singular points $(0, p_{\imb})$ and tangency points $(0, q_{\jmb})$. 
The proof of Theorem \ref{Teorema_Parametrizacion_Polar_pi_qj} is finished.


\subsection{Proof of Lemma \ref{Lema_ParamImpli_Polar_qj} and Lemma \ref{Lema_ParamImpli_Polar_pi}}
\label{Demostracion_Lema_ParamImpli_Polar_qj}
In this subsection we prove Lemma \ref{Lema_ParamImpli_Polar_qj} and Lemma \ref{Lema_ParamImpli_Polar_pi}. For this purpose we begin with a remark, already used in several statements, that involves the importance of looking to the principal and regular parts of some meromorphic expressions.


\begin{remark}
\label{Obs_PartePrincipal y ParteRegular en cocientes}
Let $f, h \colon D \rightarrow \C$ be holomorphic functions defined in the open disk $D \subseteq \C$ with center at $u_{0}$. Assume that 
\begin{equation*}
    h(u_{0})=0\,\quad\text{and}\quad h^{(1)}(u_{0})\neq 0.
\end{equation*}
Then the meromorphic map has a simple pole at $u_{0}$, and the principal and regular part of the quotient $\frac{f}{h}$ at $u_0$ is given by
\begin{equation} \label{ec: polos simples}
\Bigl(\tfrac{f}{h}\Bigl)_{\scriptscriptstyle{\CP, u_{0}}} 
= \tfrac{f(u_{0})}{h^{(1)}(u_{0})} 
\, \tfrac{1}{u - u_{0}} \, ,
\hspace{1cm}
\Bigl(\tfrac{f}{h}\Bigl)_{\scriptscriptstyle{\mathcal{R}, u_{0}}} (u_{0})
= \tfrac{f^{(1)}(u_{0})}{h^{(1)}(u_{0})} 
\, - \, \tfrac{f(u_{0})}{2\bigl(h^{(1)}(u_{0})\bigl)^{2}} \, h^{(2)}(u_{0}) \, .
\end{equation}

This expressions are useful when one considers the quotient 
\begin{equation*}
    \frac{c}{((\varsigma_{\scriptscriptstyle{{\jmb} , 1}})^{r}\, g^{(1)}_{\jmb})}\,,
\end{equation*}
where $c$ is a complex number, the transformation $\varsigma_{\scriptscriptstyle{{\jmb} , 1}}$ as defined in  (\ref{Ec_Series de Phi y Zeta}) (in particular, $\varsigma_{\scriptscriptstyle{{\jmb} , 1}}(q_{\jmb}) = 1$), and the function $g_{\jmb}(u)=(q_{\jmb} - u)(I_{\jmb}(u) - q_{\jmb})$, where $I_{\jmb}$ is the non trivial involution centered at $q_{\jmb}$. Then 
$g_{\jmb}^{(1)}(q_{\jmb}) = 0$ and $g_{\jmb}^{(2)}(q_{\jmb}) = 2$, since  $I_{\jmb}^{(1)}(q_{\jmb}) = -1$. 
Therefore, by (\ref{ec: polos simples}), the following equalities take place  
\begin{equation*}
\begin{split}
\Bigl(\tfrac{c}{(\varsigma_{\scriptscriptstyle{{\jmb} , 1}})^{r} \, g^{(1)}_{\jmb}}\Bigl)_{\scriptscriptstyle{\CP, q_{\jmb}}} 
& = \tfrac{c}{(\varsigma_{\scriptscriptstyle{{\jmb} , 1}} (q_{\jmb}))^{r} \, g^{(2)}_{\jmb}(q_{\jmb})}  \, \tfrac{1}{u - q_{\jmb}} 
= \tfrac{c}{2(u - q_{\jmb})} \, , 
\\ \Bigl(\tfrac{c}{(\varsigma_{\scriptscriptstyle{\jmb \, , 1}})^{r} \, g^{(1)}_{\jmb}}\Bigl)_{\scriptscriptstyle{\mathcal{R}, q_{\jmb}}} (q_{\jmb}) 
& = - \tfrac{r \, c}{(\varsigma_{\scriptscriptstyle{\jmb \, , 1}}(q_{\jmb}))^{r+1}}
 \, \tfrac{(\varsigma_{\scriptscriptstyle{\jmb \, , 1}})^{(1)}(q_{\jmb})}{g_{\jmb}^{(2)}(q_{\jmb})} 
\, - \, \tfrac{c}{(\varsigma_{\scriptscriptstyle{\jmb \, , 1}}(q_{\jmb}))^{r}} 
\, \tfrac{g_{\jmb}^{(3)}(q_{\jmb})}{2 (g_{\jmb}^{(2)}(q_{\jmb}))^{2}}
\\ & = -\tfrac{c}{2}\bigl(r \, (\varsigma_{\scriptscriptstyle{\jmb \, , 1}})^{(1)}(q_{\jmb}) + \tfrac{g_{\jmb}^{(3)}(q_{\jmb})}{4}\bigl) \, .
\end{split}
\end{equation*}
\end{remark}

The following lemma is a direct consequence of Remark \ref{Obs_Aplicaciones_ak+1_bk+1_como funcionales} and Remark \ref{Obs_PartePrincipal y ParteRegular en cocientes}.

\begin{lemma}
\label{Lema_Descripcion de bAst}
Let $b_{\scriptscriptstyle{\mathbf{n}, 1}}$ and $b_{\scriptscriptstyle{\mathbf{n}, k}}$, $k \geqslant 2$, be the maps defined in ($b^{1}$) and ($b^{k+1}$), respectively, in Proposition \ref{Prop:soluciones canonicas}. Then 
\begin{equation*}
\begin{split}
b_{\scriptscriptstyle{\mathbf{n}, 1}}  
 = - & \sum_{\jmb = 1}^{m} \tfrac{z_{\scriptscriptstyle{{\jmb} , 1}}}{2(u - q_{\jmb})} \, ,
\hspace{0.7cm}
b_{\scriptscriptstyle{\mathbf{n}, k}} - \mathbf{b}_{\scriptscriptstyle \mathbf{n}, k} 
= - \sum_{\jmb = 1}^{m} \tfrac{z_{\scriptscriptstyle{{\jmb} , k}}}{2(u - q_{\jmb})} \, ,
\end{split}
\end{equation*}
where $\mathbf{b}_{\scriptscriptstyle \mathbf{n}, k} = \mathbf{b}_{\scriptscriptstyle \mathbf{n}, k}[(\lambda_{\imb}, s_{\scriptscriptstyle{{\imb} , r}}, 
z_{\scriptscriptstyle{{\jmb} , r}}, g_{\jmb}; 
\varepsilon_{\scriptscriptstyle{{\imb} , r}}, \varsigma_{\scriptscriptstyle{{\jmb} , r}})_{\imb, \jmb}]_{r=1}^{k-1}$ is a holomorphic function in the punctured sphere   $\mathcal{L} \smallsetminus (\p, \q)$ (i.e., it extends analytically to $u = \infty$), which is, in turn, the image of a functional application  $\mathbf{b}_{\scriptscriptstyle \mathbf{n}, k}$ evaluated in the complex numbers   $\lambda_{\imb}$, $s_{\scriptscriptstyle{{\imb} , r}}$, $z_{\scriptscriptstyle{{\jmb} , r}}$, in the function  $g_{\jmb}$ (local information of the foliation pair  $(\tilde{\F}, \tilde{\G})$), and the coefficients $\varepsilon_{\scriptscriptstyle{{\imb} , r}}$, $\varsigma_{\scriptscriptstyle{{\jmb} , r}}$ of $\varepsilon_{\imb}$, $\varsigma_{\jmb}$ (local information of $\tilde{\F}_{\scriptscriptstyle \mu}$), for $1 \leqslant r \leqslant k-1$, with respect to every singular point  $(0, p_{\imb})$ and every tangency point $(0, q_{\jmb})$. 
\end{lemma}

\begin{proof}[Proof of Lemma \ref{Lema_ParamImpli_Polar_qj}]
By definition  (see ($\CP_{p_{i}}$) and ($\CP_{q_{j}}$))
\begin{equation*} 
\alpha_{q_{j}} (x)
= \sum_{r = 1}^{k_{0}} a_{\scriptscriptstyle{{q_{j}} , r}}(q_{j}) x^{r} 
+ \text{O} (x^{k_{0} + 1}) \, ,
\hspace{0.8cm} 
\beta_{q_{j}} (x)
= \sum_{r = 1}^{k_{0}} b_{\scriptscriptstyle{q_{j} , r}}(q_{j}) x^{r} 
+ \text{O} (x^{k_{0} + 1}) \, ;
\end{equation*}
therefore, $\alpha_{\scriptscriptstyle{q_{j} , r}} = a_{\scriptscriptstyle{{q_{j}} , r}} (q_{j})$, $\beta_{\scriptscriptstyle{q_{j} , r}} = b_{\scriptscriptstyle{q_{j} , r}}(q_{j})$, for $1 \leqslant r \leqslant k_{0}$. 

We begin by analyzing the coefficients of the transformation  $\alpha_{q_{j}}$.
From ($a^{1}$) in Proposition \ref{Prop:soluciones canonicas} and from (\ref{Ec_Series de Phi y Zeta}) we know that $a_{\scriptscriptstyle{{q_{j}} , 1}} (q_{j})= \varsigma_{\scriptscriptstyle{j, 1}}(q_{j})=1$.
By Remark \ref{Obs_Aplicaciones_ak+1_bk+1_como funcionales}, for  $2 \leqslant k \leqslant k_{0}$,
\begin{equation*}
\begin{split}
\alpha_{\scriptscriptstyle{q_{j} , k}}
= a_{\scriptscriptstyle{q_{j} , k}} (q_{j})
& = \mathbf{a}_{\scriptscriptstyle{q_{j} , k}}[\varsigma_{\scriptscriptstyle{ j \, , k}}; \, (\lambda_{\imb}, s_{\scriptscriptstyle{{\imb} , r}}, 
z_{\scriptscriptstyle{{\jmb} , r}}, g_{\jmb}; 
\varepsilon_{\scriptscriptstyle{{\imb} , r}}, \varsigma_{\scriptscriptstyle{{\jmb} , r}})_{\imb, \jmb}]_{r=1}^{k-1} (q_{j}) \, ,
\end{split}
\end{equation*}
where $\mathbf{a}_{\scriptscriptstyle{q_{j} , k}}$ is a functional transformation satisfying the properties in the statement of Lemma \ref{Lema_ParamImpli_Polar_qj}.

For the coefficients of the transformation $\beta_{q_{j}}$, we recall that from ($b^{1}$) in Proposition \ref{Prop:soluciones canonicas} we know that 
\begin{equation*}
    b_{\scriptscriptstyle{q_{j} , 1}} 
 := \tfrac{z_{\scriptscriptstyle{j, 1}}}{g^{(1)}_{j}} 
\, + \,   \varsigma_{\scriptscriptstyle{j, 1}} \, b_{\scriptscriptstyle \mathbf{n}, 1} \, .
\end{equation*}
Thus,

\begin{equation*}
       b_{\scriptscriptstyle{q_{j} , 1}} 
 = \varsigma_{\scriptscriptstyle{j, 1}} \Bigl(\tfrac{z_{\scriptscriptstyle{j, 1}}}{\varsigma_{\scriptscriptstyle{j, 1}} g^{(1)}_{j}}  
\, + \,  b_{\scriptscriptstyle \mathbf{n}, 1}\Bigl) \, .
\end{equation*} 
By Lemma \ref{Lema_Descripcion de bAst}, it follows that the coefficient $\beta_{\scriptscriptstyle{q_{j} , 1}}$ is given by

\begin{equation*}
\beta_{\scriptscriptstyle{q_{j} , 1}} =        b_{\scriptscriptstyle{q_{j} , 1}}  (q_{j}) 
 = \Bigl(\varsigma_{\scriptscriptstyle{j, 1}} \Bigl(\tfrac{z_{\scriptscriptstyle{j, 1}}}{\varsigma_{\scriptscriptstyle{j, 1}} g^{(1)}_{j}}  
\,   - \sum_{\jmb = 1}^{m} \tfrac{z_{\scriptscriptstyle{{\jmb} , k}}}{2(u - q_{\jmb})} \Bigl) \Bigl)\,  (q_{j}) \, ,
\end{equation*}
 Hence, since $\varsigma_{\scriptscriptstyle{j, 1}} (q_{j}) = 1$ (by (\ref{Ec_Series de Phi y Zeta})), and by Remark  \ref{Obs_PartePrincipal y ParteRegular en cocientes},  we have

\begin{equation*}
\beta_{\scriptscriptstyle{q_{j} , 1}} =        b_{\scriptscriptstyle{q_{j} , 1}}  (q_{j}) 
 =   -\tfrac{z_{\scriptscriptstyle{j,1}} }{2}\Bigl((\varsigma_{\scriptscriptstyle{j, 1}})^{(1)}(q_{j}) + \tfrac{g_{j}^{(3)}(q_{j})}{4}\Bigl) 
-  \sum_{\jmb \neq j}^{m} \tfrac{z_{\scriptscriptstyle{{\jmb} , 1}}}{2(q_{j} - q_{\jmb})} \, .
\end{equation*}

For $2 \leqslant k \leqslant k_{0}$,

\begin{equation}\label{coeficientes betaqjk}
\begin{split}
b_{\scriptscriptstyle{q_{j} , k}}
- \mathbf{b}_{\scriptscriptstyle{q_{j} , k}} 
& = \tfrac{z_{\scriptscriptstyle{j, k}}}{g_{j}^{(1)}}
 - ( \varsigma_{\scriptscriptstyle{j, 1}} )^{k} \sum_{\jmb = 1}^{m} \Bigl(\tfrac{z_{\scriptscriptstyle{{\jmb} , k}}}{(\varsigma_{\scriptscriptstyle{{\jmb} , 1}})^{k} g_{\jmb}^{(1)}} \Bigl)_{\CP, q_{\jmb}}
\\ & = ( \varsigma_{\scriptscriptstyle{j, 1}} )^{k} \biggl(\Bigl(\tfrac{z_{\scriptscriptstyle{j, k}}}{( \varsigma_{\scriptscriptstyle{j, 1}} )^{k} \, g_{j}^{(1)}}\Bigl)_{\mathcal{R}, q_{\jmb}}
 -  \sum_{\jmb \neq j}^{m} \Bigl(\tfrac{z_{\scriptscriptstyle{{\jmb} , k}}}{(\varsigma_{\scriptscriptstyle{{\jmb} , 1}})^{k} g_{\jmb}^{(1)}} \Bigl)_{\CP, q_{\jmb}}\biggl)
  \, ,
\end{split}
\end{equation}
where $\mathbf{b}_{\scriptscriptstyle{q_{j} , k}}$ is a functional transformation satisfying the properties in the statement of Lemma \ref{Lema_ParamImpli_Polar_qj}. Since   $ \varsigma_{\scriptscriptstyle{j, 1}} (q_{j}) = 1$ and $\beta_{\scriptscriptstyle{q_{j} , k}}  = b_{\scriptscriptstyle{q_{j} , k}} (q_{j})$, the evaluation of (\ref{coeficientes betaqjk}) at $q_{j}$ imply, by Remark \ref{Obs_PartePrincipal y ParteRegular en cocientes}, the equalities
\begin{equation*}
\begin{split}
\beta_{\scriptscriptstyle{q_{j} , k}}
- \mathbf{b}_{\scriptscriptstyle{q_{j} , k}} (q_{j})
= & \Bigl(\tfrac{z_{\scriptscriptstyle{j, k}}}{( \varsigma_{\scriptscriptstyle{j, 1}} )^{k} \, g_{j}^{(1)}} \Bigl)_{\mathcal{R}, q_{\jmb}}  (q_{j})
 - \sum_{\jmb \neq j}^{m} \Bigl(\tfrac{z_{\scriptscriptstyle{{\jmb} , k}}}{(\varsigma_{\scriptscriptstyle{{\jmb} , 1}})^{k} g_{\jmb}^{(1)}}\Bigl)_{\CP, q_{\jmb}} (q_{j})
\\ = - & \tfrac{z_{\scriptscriptstyle{j, k}}}{2} \Bigl(k (\varsigma_{\scriptscriptstyle{j, 1}})^{(1)}(q_{j})  + \tfrac{g_{j}^{(3)}(q_{j})}{4}\Bigl) 
- \sum_{\jmb \neq j}^{m} \tfrac{z_{\scriptscriptstyle{{\jmb} , k}}}{2(q_{j} - q_{\jmb})} \, .
\end{split}
\end{equation*}
Lemma \ref{Lema_ParamImpli_Polar_qj} is proved.

\end{proof}

\label{Demostracion_Lema_ParamImpli_Polar_pi}
\begin{proof}[Proof of Lemma \ref{Lema_ParamImpli_Polar_pi}]
We asume that $(\F, \G) \in \mathcal{N}_{\p} (\mathbf{h}) \times \mathcal{D}_{\q} (\mathcal{I})$ is a foliation pair whose normalizing transformations $\mathbf{H}_{\scriptscriptstyle \mathbf{n}} $, $H_{p_{i}}$, $H_{q_{j}}$ with respect to the foliation pair $(\F_{\scriptscriptstyle \mathbf{\mu}}, \G_{\scriptscriptstyle \mathbf{r}})$ have the $k_{0}$--jet of its power series expansion ($\mathbf{nt_{k_{0}}}$) as given in Theorem \ref{Teorema_Limpieza de TransNorm_Intro_1}. 

Let $\CP_{p_{i}} (\tilde{\F}, \tilde{\G})$  be the corresponding curve of tangencies in the charts $(x,u)$
\begin{equation*}
    \CP_{p_{i}} (\tilde{\F}, \tilde{\G})= H_{p_{i}} (x, p_{i} + \tilde{s}_{i} (x))\,,
\end{equation*}
 where $\tilde{s}_{i} (x):= - \lambda_{i} x s_{i}^{(1)} (x)$ at the singular point, $(0,p_{i})$.

We asume as well that $(\alpha_{p_{i}} (x), p_{i} + \beta_{p_{i}} (x))$ represents the implicit parametrization of the curve of tancencies  $\CP_{p_{i}} (\tilde{\F}, \tilde{\G})$ as in ($\CP_{p_{i}}$). Such parametrization is expressed in terms of the series
\begin{equation}\label{series implicitas de polar pi}
\begin{split}
\sum_{r = 1}^{k_{0}} a_{\scriptscriptstyle{p_{i} , r}} (p_{i} + \tilde{s}_{i} (x)) x^{r} 
\, + \,\text{O} (x^{k_{0}+1}) \, ,
\hspace{0.3cm}
\sum_{r = 1}^{k_{0}} b_{\scriptscriptstyle{p_{i} , r}}(p_{i} + \tilde{s}_{i} (x)) x^{r} 
\, + \, \text{O} (x^{k_{0}+1}) \, , 
\end{split} 
\end{equation}
where $\tilde{s}_{i}(x) = -\lambda_{i} \textstyle\sum_{r \geqslant 1} r s_{\scriptscriptstyle{{i} , r}} x^{r}$. 

For $1 \leqslant k \leqslant k_{0}$, 
the coefficients corresponding to the monomial $x^k$ in the power series expansion (\ref{series implicitas de polar pi}) are given, respectively by

\begin{equation}\label{ec: Pageref_Pk en ar y br}
a_{\scriptscriptstyle{p_{i} , k}} (p_{i}) 
\, + \, \mathsf{P}^{k}[a_{\scriptscriptstyle{p_{i} , r}}; s_{\scriptscriptstyle{{i} , r}}]_{r=1}^{k-1} (p_{i}) \, ,
\hspace{0.8cm}
b_{\scriptscriptstyle{p_{i} , k}} (p_{i}) 
\, + \, \mathsf{P}^{k}[b_{\scriptscriptstyle{p_{i} , r}}; s_{\scriptscriptstyle{{i} , r}}]_{r=1}^{k-1} (p_{i}) \, ,
\end{equation}
where $\mathsf{P}^{k}[a_{\scriptscriptstyle{p_{i} , r}}; s_{\scriptscriptstyle{{i} , r}}]_{r=1}^{k-1}$ is a holomorphic transformation in $D_{p_{i}}$ obtained by the evaluation of a functional transformation $\mathsf{P}^{k}$ in the coefficients $a_{\scriptscriptstyle{p_{i} , r}}$ and the complex numbers $s_{\scriptscriptstyle{{i} , r}}$, and 
  $\mathsf{P}^{k}[b_{\scriptscriptstyle{p_{i} , r}}; s_{\scriptscriptstyle{{i} , r}}]_{r=1}^{k-1}$
is the same functional transformation $\mathsf{P}^{k}$ evaluated in the coefficients $b_{\scriptscriptstyle{p_{i} , r}}$ and the complex numbers $s_{\scriptscriptstyle{{i} , r}}$, for $1 \leqslant r \leqslant k-1$. The functional transformation $\mathsf{P}^{k}$ is constructed by Fa\`a di Bruno formula (\ref{Ec_Formula de FaaDiBruno con polinomios}), in terms of the polinomials $\mathbf{P}^{1}, \ldots, \mathbf{P}^{k-1}$. In particular, $\mathsf{P}^{1}$ vanishes identically.

\begin{itemize}
    \item [a)] Coefficients of $\alpha_{p_{i}}=\textstyle \sum_{r = 1}^{k_{0}} a_{\scriptscriptstyle{p_{i} , r}} (p_{i} 
+ \tilde{s}_{i} (x)) x^{r} \, 
+ \,\text{O} (x^{k_{0}+1})$.

By the expressions ($a^{1}$) in Proposition \ref{Prop:soluciones canonicas} we know that $a_{\scriptscriptstyle{p_{i} , 1}}=\varepsilon_{\scriptscriptstyle{{i} , 1}}$, and by \eqref{Ec_Series de Psi y Xi}, its evaluation at  $p_{i}$ satisfies
\begin{equation}\label{ec: apien pi es uno}
    a_{\scriptscriptstyle{p_{i} , 1}}(p_{i})=\varepsilon_{\scriptscriptstyle{{i} , 1}}(p_{i})=1\,.
\end{equation}
Moreover, by Remark \ref{Obs_Aplicaciones_ak+1_bk+1_como funcionales}, for $r \geqslant 2$, the holomorphic transformation $a_{\scriptscriptstyle{p_{i} , r}}$ is the image of a  functional transformation evaluated in $\varepsilon_{\scriptscriptstyle{{i} , r}}$, the complex numbers $\lambda_{\imb}$, $s_{\scriptscriptstyle{{\imb} , l}}$, $z_{\scriptscriptstyle{{\jmb} , l}}$, the transformation  $g_{\jmb}$, and the coefficients $\varepsilon_{\scriptscriptstyle{{\imb} , l}}$, $\varsigma_{\scriptscriptstyle{{\jmb} , l}}$, with $1 \leqslant l \leqslant r-1$, and with respect to all the singularity points $(0, p_{\imb})$, $\imb=1,\dots,n+1$, and the tangency points $(0, q_{\jmb})$, $\jmb=1,\dots,m$.

Since $\alpha_{p_{i}}=\textstyle \sum_{r = 1}^{k_{0}} a_{\scriptscriptstyle{p_{i} , r}} (p_{i} 
+ \tilde{s}_{i} (x)) x^{r} \, 
+ \,\text{O} (x^{k_{0}+1})$, using (\ref{ec: apien pi es uno}) and the expressions (\ref{ec: Pageref_Pk en ar y br}), it follows that
\begin{equation*}
\begin{split}
\alpha_{\scriptscriptstyle{p_{i} , k}}
& = \mathbf{a}_{\scriptscriptstyle{p_{i} , k}}[\varepsilon_{\scriptscriptstyle{{i} , k}}; \, (\lambda_{\imb}, s_{\scriptscriptstyle{{\imb} , r}}, 
z_{\scriptscriptstyle{{\jmb} , r}}, g_{\jmb}; 
\varepsilon_{\scriptscriptstyle{{\imb} , r}}, \varsigma_{\scriptscriptstyle{{\jmb} , r}})_{\imb, \jmb}]_{r=1}^{k-1} \, (p_{i})\,,
\end{split}
\end{equation*}
where $\mathbf{a}_{\scriptscriptstyle{p_{i} , k}}$ is a functional transformation that satisfies the properties stated  in Lemma \ref{Lema_ParamImpli_Polar_pi}.


\item[b)] Coefficients of $\beta_{p_i} = \tilde{s}_{i} (x) + \sum_{r = 1}^{k_{0}} b_{\scriptscriptstyle{p_{i} , r}}(p_{i} + \tilde{s}_{i} (x)) x^{r} 
\, + \, \text{O} (x^{k_{0}+1})$.

By the expressions ($b^{1}$) in Proposition \ref{Prop:soluciones canonicas}
we know that $b_{\scriptscriptstyle{p_{i} , 1}}=s_{\scriptscriptstyle{i,1}}  + \varepsilon_{\scriptscriptstyle{{i} , 1}}b_{\scriptscriptstyle{\mathbf{n}, 1}}  $, and by Lemma \ref{Lema_Descripcion de bAst}, 
\begin{equation}\label{ec: bpi1 en pi}
    b_{\scriptscriptstyle{p_{i} , 1}}(p_{i})=\Bigl( s_{\scriptscriptstyle{i,1}}  
- \varepsilon_{\scriptscriptstyle{{i} , 1}}  \sum_{\jmb = 1}^{m} \frac{z_{\scriptscriptstyle{{\jmb} , 1}}}{2(u - q_{\jmb})}\Bigl) (p_{i})\,.
\end{equation}

Moreover, if $r \geqslant 2$, by Remark \ref{Obs_Aplicaciones_ak+1_bk+1_como funcionales} and Remark \ref{Obs_PartePrincipal y ParteRegular en cocientes}, the transformation
\begin{equation*}
\begin{split}
b_{\scriptscriptstyle{p_{i} , r}} 
- s_{\scriptscriptstyle{i,r}} 
+ ( \varepsilon_{\scriptscriptstyle{i, 1}} )^{r} \sum_{\jmb =1}^{m}  \tfrac{z_{\scriptscriptstyle{{\jmb} , r}}}{2(u-q_{\jmb})}
\, =  \, b_{\scriptscriptstyle{p_{i} , r}} 
- s_{\scriptscriptstyle{i,r}} 
+ ( \varepsilon_{\scriptscriptstyle{i, 1}} )^{r} \sum_{\jmb =1}^{m} \Bigl(  \tfrac{z_{\scriptscriptstyle{{\jmb} , r}}}{(\varsigma_{\scriptscriptstyle{{\jmb} , 1}})^{r} g_{\jmb}^{(1)}} \Bigl)_{\CP, q_{\jmb}} 
\end{split}
\end{equation*}
is obtained once more by the evaluation of a functional application in $\varepsilon_{\scriptscriptstyle{{i} , r}}$, the complex numbers $\lambda_{\imb}$, $s_{\scriptscriptstyle{{\imb} , l}}$, $z_{\scriptscriptstyle{{\jmb} , l}}$, the transformation $g_{\jmb}$, and the coefficients $\varepsilon_{\scriptscriptstyle{{\imb} , l}}$, $\varsigma_{\scriptscriptstyle{{\jmb} , l}}$ of $\varepsilon_{\imb}$ and $\varsigma_{\jmb}$, for  $1 \leqslant l \leqslant r-1$, and with respect to all the singularities $(0, p_{\imb})$, $\imb= 1\dots,n+1$ and all the tangencies  $(0, q_{\jmb})$, $\jmb=1,\dots,m$.

By ($\CP_{p_{i}}$) we know that 
\begin{equation*}
    \beta_{p_{i}} - p_{i} 
- \tilde{s}_{i} (x)=\textstyle \sum_{r = 1}^{k_{0}} b_{\scriptscriptstyle{p_{i} , r}} (p_{i} 
+ \tilde{s}_{i} (x)) x^{r} \, 
+ \,\text{O} (x^{k_{0}+1})\,.
\end{equation*}
Hence, since $\tilde{s}_{i} (x) = - \lambda_{i} \sum_{r \geqslant 1} r s_{\scriptscriptstyle{i, r}} x^{r} $, by (\ref{ec: bpi1 en pi}) 
\begin{equation*}
\beta_{\scriptscriptstyle{p_{i} , 1}}=(1 - \lambda_{i}) s_{\scriptscriptstyle{i,1}}  
- \textstyle \sum_{\jmb =1}^{m} \frac{z_{\scriptscriptstyle{{\jmb} , 1}}}{2(p_{i}-q_{\jmb})}\,,
\end{equation*}
and for $2 \leqslant k \leqslant k_{0}$, by using the previous properties and the expressions (\ref{ec: Pageref_Pk en ar y br}), we get
\begin{equation*}
\begin{split}
\beta_{\scriptscriptstyle{p_{i} , k}} 
& = (1 - k\lambda_{i}) s_{\scriptscriptstyle{{i} , k}} 
- \sum_{\jmb =1}^{m} \tfrac{z_{\scriptscriptstyle{{\jmb} , k}}}{2(p_{i}-q_{\jmb})} 
 + \mathbf{b}_{\scriptscriptstyle{p_{i} , k}} (p_{i})\,,
\end{split}
\end{equation*}
where $\mathbf{b}_{\scriptscriptstyle{p_{i} , k}}:=\mathbf{b}_{\scriptscriptstyle{p_{i} , k}}[\varepsilon_{\scriptscriptstyle{{i} , k}}; \, (\lambda_{\imb}, s_{\scriptscriptstyle{{\imb} , r}}, 
z_{\scriptscriptstyle{{\jmb} , r}}, g_{\jmb}; 
\varepsilon_{\scriptscriptstyle{{\imb} , r}}, \varsigma_{\scriptscriptstyle{{\jmb} , r}})_{\imb, \jmb}]_{r=1}^{k-1} \,$ is the functional transformation satisfying the stated properties in Lemma    \ref{Lema_ParamImpli_Polar_pi}.
\end{itemize}
\end{proof}

\subsection{ Proof of Lemma \ref{lem: series TN por ecFac en pi} and Lemma \ref{lem: series TN por ecFac en qj}}\label{Dem dos lemas}

We recall that it was stressed that Lemma \ref{lem: series TN por ecFac en pi} on the factorizing equations at the singular points $p_{i}$ and Lemma \ref{lem: series TN por ecFac en qj} on the factorizing equations at the tangency points $q_{j}$ rely on Lemma \ref{Lema_Series de composicion A en alpha y beta}, Lemma \ref{Lema_Aplicacion hatPsik_i} and Lemma \ref{Lema_Aplicacion Phik_j}. In this subsection we state and proof such lemmas. We begin by stating  Lemma \ref{Lema_Series de composicion A en alpha y beta}.

\subsubsection*{Power series of the composition $A (\alpha, \beta)$.}
\label{Series de composicion A en alpha beta}
Let $A$, $\alpha$, $\beta$ be holomorphic functions defined in domains in  $\C^{2}$ intersecting the $u$--axis. Assume that the restriction to the $u$--axis of the functions  $\alpha$, $\beta$ equals to the identity.
Assume as well, that the composition  $A (\alpha, \beta)$ is well defined, and let 
\begin{equation*}
\begin{split}
A & = A_{ 0 } (u) + A_{ 1 }(u) x + A_{2}(u) x^{2} + \cdots  + A_{ k }(u) x^{k} + \cdots \, , \\
\alpha & =  \hspace{1.38cm} \alpha_{1}(u) x + \alpha_{2}(u) x^{2} + \cdots + \alpha_{k}(u) x^{k} + \cdots \, , \\
\beta & = \hspace{0.8cm} u + \beta_{1}(u) x + \beta_{2}(u) x^{2} + \cdots + \beta_{k}(u) x^{k} + \cdots \, ,
\end{split}
\end{equation*}
be the series expansion with respect to $u$ of $A$, $\alpha$, $\beta$  in the corresponding domains.

\begin{lemma}
\label{Lema_Series de composicion A en alpha y beta} 
The composition $A (\alpha, \beta)$ equals to $A_{ 0 }$ at the $u$--axis; the coefficient $x$ in the power series of $A (\alpha, \beta)$ equals $\beta_{1} \, \tfrac{\text{d} \, A_{ 0 }}{\text{d} \, u}  \, + \, \alpha_{1} \, A_{ 1 }$, and for $k \geqslant 2$ the coefficient  of $x^{k}$ equals to
\begin{equation*}
\beta_{k} \, \tfrac{\text{d} \, A_{ 0 }}{\text{d} \, u}  
\, + \,  \alpha_{k} \, A_{ 1 }  
\, + \, (\alpha_{1})^{k} \, A_{ k }  \, + \, \Omega^{k} [A_{ 0 }, A_{ r }; \alpha_{r}, \beta_{r}]^{k-1}_{r = 1} \, , 
\end{equation*}
where $\Omega^{k} [A_{ 0 }, A_{ r }; \alpha_{r}, \beta_{r}]^{k-1}_{r = 1}$ is a polynomial with rational non-negative coefficients depending on the derivatives    $\tfrac{\text{d}^{i} \, A_{ 0 }}{\text{d} \, u^{i}}$ de $A_{ 0 }$ of order $2 \leqslant i \leqslant k$, and for $1 \leqslant r \leqslant k-1$, it depends on the coefficients $\alpha_{r}$, $\beta_{r}$ of the power series $\alpha$ and $\beta$, and on the derivatives   $\tfrac{\text{d}^{s} \, A_{ r }}{\text{d} \, u^{s}}$ of $A_{ r }$ of order $0 \leqslant s \leqslant k - r$. The definition of the polynomial $\Omega^{k}$ is independent of the functions $A$, $\alpha$, $\beta$.
\end{lemma}
For sake of completeness we give the proof of Lemma \ref{Lema_Series de composicion A en alpha y beta} at the end of this section (see \ref{Aplicacion Phik_j_Demostracion}).


\subsubsection*{ On the normalizing transformations $\mathbf{H}_{\scriptscriptstyle \mathbf{n}} $, $H_{p_{i}}$}
Recall the normalizing transformations $\mathbf{H}_{\scriptscriptstyle \mathbf{n}} $, $H_{p_{i}}$,
which satisfy the \textit{factorizing equations} (see figure \ref{Figura_Construccion de PdF a partir de PL}):
\begin{equation*} 
H_{p_{i}} = \mathbf{H}_{\scriptscriptstyle \mathbf{n}} \circ \xi_{i} \circ \Psi_{i} \, ,
\hspace{1cm}
\end{equation*}
where the biholomorphism $\Psi_{i}$, defined in a neighborhood in $\C^{2}$ of the annulus $D_{p_{i}} \smallsetminus \{(0, p_{i})\}$, and the biholomorphism $\xi_{i} \colon (\C^{2}, D_{p_{i}}) \rightarrow (\M, D_{p_{i}})$  satisfy 
\begin{equation*} 
\begin{split}
\Psi_{i} 
= \bigl(\hat{\psi}_{i}, \, u + s_{i}(x)\bigl)
\, , 
\hspace{0.4cm} 
\hat{\psi}_{i} 
& = \sum_{k \geqslant 1}  \hat{\psi}_{\scriptscriptstyle{i, k}} (u) \, x^{k}\,,
\hspace{0.3cm} 
\text{where}
\hspace{0.3cm}
\hat{\psi}_{i}(x,u)=x \bigl(1 + \tfrac{s_{i}(x)}{u - p_{i}}\bigl)^{\lambda_{i}} 
\, , \\
\xi_{i}  = (\varepsilon_{i}, u) \, ,
\hspace{0.5cm}
\varepsilon_{i} & = \sum_{k \geqslant 1}  \varepsilon_{\scriptscriptstyle{i, k}} (u) \, x^{k} \, ,
\hspace{0.3cm}
 \varepsilon_{\scriptscriptstyle{i, 1}}  (p_{i}) = 1 \,.
\end{split}
\end{equation*}
The holomorphic map $s_{i} \colon (\C, 0) \rightarrow (\C, 0)$ depends on the foliation pair $(\F, \G)$ in $(0,p_{i})$ with respect to the foliation pair $(\F_{\scriptscriptstyle \mathbf{\mu}}, \G_{\scriptscriptstyle \mathbf{r}})$. The power series of $s_{i}$ is written as
\begin{equation} 
    s_{i}=\sum_{r \geqslant 1} s_{\scriptscriptstyle{i,r}}  \, x^{r}, \quad s_{\scriptscriptstyle{i,r}}  \in \C.
\end{equation}


\begin{lemma} 
\label{Lema_Aplicacion hatPsik_i} 
The transformation $\hat{\psi}_{\scriptscriptstyle{ {i} , 1}}$ equals to $1$, and if  $k \geqslant 1$, 
\begin{equation*}
\begin{split}
\hat{\psi}_{\scriptscriptstyle{ {i} , k+1}} (u)
& = \lambda_{i} \, \tfrac{\,  s_{\scriptscriptstyle{i, k}}}{u-p_i}
\, + \, \tfrac{1}{k!} \, \tilde{\mathbf{P}}^{k}\bigl[\lambda_{i} (\lambda_{i} - 1) \cdots (\lambda_{i} - r)  \,  ; \, \tfrac{r! \, s_{\scriptscriptstyle{i,r}} }{u-p_i}\bigl]^{k-1}_{r=1} \\
& = \lambda_{i} \, \tfrac{\,  s_{\scriptscriptstyle{i, k}}}{u-p_i}
\, + \, \sum \tfrac{1}{r_{1}! \cdots r_{k-1}!} \, \tfrac{\lambda_{i} (\lambda_{i} - 1) \cdots (\lambda_{i} - r + 1)}{(u - p_{i})^{r}} \, (s_{\scriptscriptstyle{i, 1}} )^{r_{1}} \cdots (s_{\scriptscriptstyle{i, k-1}})^{r_{k-1}}\, ,
\end{split}
\end{equation*}
where $\tilde{\mathbf{P}}^{k}$ is the polynomial defined in \eqref{Ec_Definicion de Pk}.
\end{lemma}

\begin{proof}
    
 For sake of simplicity, throughout the proof we will skip the subscript $i$ in the notation.

 Let $\psi := 1 + \tfrac{s(x)}{u - p}$. If $s = \sum_{k \geqslant 1} s_{k} \, x^{k}$, then the power series expanssion $\psi = \sum_{k \geqslant 0} \psi_{k} (u) \, x^{k}$ in a neighborhood of  the $u$--axis equals to
\begin{equation*}
1 
+ \tfrac{s_{1}}{u-p} \, x 
+ \tfrac{s_{2}}{u-p} \, x^{2} + \cdots +
\tfrac{s_{k}}{u-p} \, x^{k} + \cdots \, ,
\end{equation*}
that is to say, for $k \geqslant 1$,
 $\tfrac{1}{k!} \psi_{x^{k}} (0, u) = \psi_{k} (u) = \tfrac{s_{k}}{u-p}$. 
 
 Let $h (z) = z^{\lambda}$; then, $h^{(r)} (1) = \lambda (\lambda - 1) \cdots (\lambda - r+1)$. By Fa\`a di Bruno's formula (see  (\ref{Ec_Formula de FaaDiBruno con polinomios})) with respect to the composition $h \circ \psi = \psi^{\lambda}$ we get\begin{equation*}
\begin{split}
\tfrac{\partial^{k} \, \psi^{\lambda}}{\partial \, x^{k}} \Bigl\vert_{(0, u)}
& = h^{(1)} (1) \psi_{x^{k}} (0, u)
\, + \, \tilde{\mathbf{P}}^{k}\Bigl[h^{(i+1)} (1)  \,  ; \, \psi_{x^{i}} (0, u) \bigl]^{k-1}_{i=1} \\
& = \lambda \, \tfrac{k! \, s_{k}}{u-p}
\, + \, \tilde{\mathbf{P}}^{k}\Bigl[\lambda (\lambda - 1) \cdots (\lambda - i)  \,  ; \, \tfrac{i! \, s_{i}}{u-p}\bigl]^{k-1}_{i=1} \, .
\end{split}
\end{equation*}
Since $\hat{\psi} = x \psi^{\lambda}$, its power series expansion  $\hat{\psi} = \sum_{k \geqslant 1} \hat{\psi}_{k} (u) \, x^{k}$ in a neighborhood of the  $u$--axis satisfies $\hat{\psi}_{1} (u) = 1$ and  $\hat{\psi}_{k+1} (u)= \tfrac{1}{k!} \, \tfrac{\partial^{k} \, \psi^{\lambda}}{\partial \, x^{k}} \bigl\vert_{(0, u)}$ for $k \geqslant 1$, i.e., 
\begin{equation*}
\begin{split}
\hat{\psi}^{k+1} (u)
& = \lambda \, \tfrac{\, s_{k}}{u-p}
\, + \, \tfrac{1}{k!} \, \tilde{\mathbf{P}}^{k}\Bigl[\lambda (\lambda - 1) \cdots (\lambda - i)  \,  ; \, \tfrac{i! \, s_{i}}{u-p}\bigl]^{k-1}_{i=1} \\
& = \lambda \, \tfrac{\, s_{k}}{u-p}
\, + \, \sum \tfrac{1}{r_{1}! \cdots r_{k-1}!} \, \tfrac{\lambda (\lambda - 1) \cdots (\lambda - r + 1)}{(u - p)^{r}} \, (s_{1})^{r_{1}} \cdots (s_{k-1})^{r_{k-1}}
\, ,
\end{split}
\end{equation*}
where the sum is in the non-negative integers   $r_{1}, \ldots, r_{k-1}$ satisfying $r_{1} + 2 r_{2} + \cdots + (k-1) r_{k-1} = k$, and $r=r_{1} + r_{2} + \cdots + r_{k-1}$.
The proof of Lemma \ref{Lema_Aplicacion hatPsik_i} is finished.
\end{proof}


\subsubsection*{ On the normalizing transformations $\mathbf{H}_{\scriptscriptstyle \mathbf{n}} $, $H_{q_{j}}$}
We now recall 
 that the normalizing transformations $\mathbf{H}_{\scriptscriptstyle \mathbf{n}} $, $H_{q_{j}}$ satisfying the \textit{factorizing equations} (see figure \ref{Figura_Construccion de PdF a partir de PL}):
\begin{equation}\label{eq:ecuaciones de factorizacion_bis}
\hspace{1cm}
H_{q_{j}} = \mathbf{H}_{\scriptscriptstyle \mathbf{n}} \circ \zeta_{j} \circ \Phi_{j} \, ,
\end{equation}
the biholomorphism $\Phi_{j}$ defined in a neighborhood of the annulus $D_{q_{j}} \smallsetminus \{(0, q_{j})\}$ in $\C^{2}$, and the biholomorphism $\zeta_{j} \colon (\C^{2}, D_{q_{j}}) \rightarrow (\M, D_{q_{j}})$ satisfy   
\begin{equation} \label{Ec_Series de Phi y Zeta_Bis}
\begin{split}
& \Phi_{j} 
 = \bigl(x, \phi_{j}) \, ,
\hspace{0.2cm} 
\phi_{j}  = u  + \sum_{k \geqslant 1}  \phi_{\scriptscriptstyle{j, k}} (u) \, x^{k} \, ,
\hspace{0.1cm} 
\text{where} 
\hspace{0.1cm} 
g_{j} \circ \phi_{j} (x, u) = g_{j}(u) + z_{j} (x) \, ,
\\
& \zeta_{j} = (\varsigma_{j}, u) \, ,
\hspace{0.3cm}
\varsigma_{j}  = 
\sum_{k \geqslant 1}  \varsigma_{\scriptscriptstyle{j, k}} (u) \, x^{k} \, ,
\hspace{0.3cm}
 \varsigma_{\scriptscriptstyle{j, 1}} (q_{j}) = 1 \, ,
\hspace{0.2cm}
 \varsigma_{\scriptscriptstyle{j, k+1}}  (q_{j}) = 0 \, , 
\hspace{0.2cm}
\text{if}
\hspace{0.2cm}
k > 1  \, ,
\end{split}
\end{equation}
(see \eqref{Ec_Series de Phi y Zeta}). The biholomorphism $z_{j} \colon (\C, 0) \rightarrow (\C, 0)$  is present in  the analytic model for the foliation pair $(\F, \G)$ at the point $(0,q_{j})$ with respect to the foliation  pair $(\F_{\scriptscriptstyle \mathbf{\mu}}, \G_{\scriptscriptstyle \mathbf{r}})$.  The power series of $z_{j}$ is written as
  \begin{equation*} 
     z_{j}= \sum_{r \geqslant 1} z_{\scriptscriptstyle{j, r}}  \, x^{r},\quad z_{\scriptscriptstyle{j, r}}  \in \C
  \end{equation*}
Moreover, the definition of the biholomorphism $\zeta_{j}$ only depends on the non dicritical foliation $\F_{\scriptscriptstyle \mu}$, and not on the foliation pair $(\F, \G)$.

\begin{lemma}
\label{Lema_Aplicacion Phik_j}
Let $\Phi_{j}= (x,\phi_{j})$ as in (\ref{eq:ecuaciones de factorizacion_bis}) and (\ref{Ec_Series de Phi y Zeta_Bis}). For every $k \geqslant 1$, 
\begin{equation*}
\tfrac{\partial^{k} \, \phi_{j}}{\partial \, x^{k}} \Bigl\vert_{(0, u)}
= \tfrac{z_{j}^{(k)} (0)}{g_{j}^{(1)} (u)} \, - \, \tfrac{1}{g_{j}^{(1)} (u)} \, \tilde{\mathbf{P}}^{k} \Bigl[g_{j}^{(s+1)}(u) \, ; \, \tfrac{\partial^{s} \, \phi_{j}}{\partial \, x^{s}} \Bigl\vert_{(0, u)}
 \Bigl]_{s=1}^{k-1} \, ,
\end{equation*}
where $\tilde{\mathbf{P}}^{k}$ is the polynomial defined in \eqref{Ec_Definicion de Pk}. Moreover, $\tilde{\mathbf{P}}^{k} \bigl[g_{j}^{(s+1)}(u) \, ; \, \tfrac{\partial^{s} \, \phi_{j}}{\partial \, x^{s}} \bigl\vert_{(0, u)} \bigl]_{s=1}^{k-1}$ is a rational transformation with denominator $(g_{j}^{(1)}(u))^{2k-2}$ and numerator equal to a polynomial with coefficients in  $\Z$ and variables $z_{j}^{(s)}(0)$, $g_{j}^{(1)} (u)$, $g_{j}^{(s+1)}(u)$, with $1 \leqslant s \leqslant k-1$. The definition of the polynomial is independent of the transformations $g_{j}$ and $z_{j}$.
\end{lemma}

\begin{proof}

The proof is by induction over $k \geqslant 1$. For sake of simplicity we omit the subscripts $j$ and $j, k$. Thus, we consider 
\begin{equation*}
\begin{split}
\Phi
= \bigl(x, \phi) \, ,
\hspace{0.3cm} 
\phi = u + \sum_{k \geqslant 1} \phi_{x^{k}}(0, u) \, x^{k} \, ,
\hspace{0.3cm} 
\text{where} 
\hspace{0.3cm} 
g \circ \phi (x, u) = g(u) + z (x)\,,
\end{split}
\end{equation*}
where $\phi_{x^{k}}(0, u) $ denotes the partial derivative of $\phi$ evaluated at $(0, u)$, $\phi_{x^{k}}(0, u)=\tfrac{\partial^{k} \, \phi_{j}}{\partial \, x^{k}} \Bigl\vert_{(0, u)}$.

Since $g \circ \phi (x, u)= g(u) + z (x)$, then the $k$-th  partial derivative  $(g \circ \phi)_{x^{k}}$ with respect to $x$ equals to the derivative $z^{(k)}$, for every $k \geqslant 1$. In particular, for $k = 1$ we get $(g \circ \phi)_{x} (0, u) = g^{(1)}(u) \, \phi_{x} (0, u)= z^{(1)}(0)$. Moreover, since  $\tilde{\mathbf{P}}^{1} = 0$ (see (\ref{Ec_Definicion de Pk})),  the case $k = 1$ is proved.

We now assume the Lemma for $1 \leqslant k \leqslant k_{0}$; hence, $\phi_{x^{k}} (0, u)$ is a rational map with  $\bigl(g^{(1)}(u)\bigl)^{2k-1}$ as denominator, and having as numerator a polynomial with integer coefficients and variables $g^{(s)}(u)$, $z^{(s)}(0)$, $1 \leqslant s \leqslant k$. Note that the definition of the polynomial is independent of the transformations $g$ and $z$.

We prove the case $k_{0} + 1$.  By definition, the polynomial $\tilde{\mathbf{P}}^{k_{0} + 1}$ (see (\ref{Ec_Definicion de Pk})) it follows that  $\tilde{\mathbf{P}}^{k_{0} + 1} \bigl[g^{(s+1)}(u) \, ; \,  \phi_{x^{s}} (0, u) \bigl]_{s=1}^{k_{0}}$ equals to
\begin{equation*}
\sum \tfrac{(k_{0} + 1)!}{r_{1}! \cdots r_{k_{0}}!} \, \tfrac{1}{(1!)^{r_{1}} (2!)^{r_{2}} \cdots (k_{0}!)^{r_{k_{0}}}} g^{(r)}(u) \, \left(\phi_{x}(0, u)\right)^{r_{1}} \, \cdots \, \left(\phi_{x^{k_{0}}}(0, u)\right)^{r_{k_{0}}} \, ,
\end{equation*}
where the sum is considered over the non-negative integers $r_{1}, \ldots, r_{k_{0}}$ satisfying $r_{1} + 2 r_{2} + \cdots + k_{0} r_{k_{0}} = k_{0} + 1$, and $r$ represents the sum  $r_{1} + r_{2} + \cdots + r_{k_{0}}$ (see Remark \ref{Obs_Sobre la suma en el Teo_FaaDiBruno}). By the induction step, the rational transformation  
\begin{equation}\label{producto parciales phi}
\left(\phi_{x}(0, u)\right)^{r_{1}} \, \cdots \, \left(\phi_{x^{k_{0}}}(0, u)\right)^{r_{k_{0}}} 
\hspace{0.4cm}
\end{equation}
has as denominator
\begin{equation*}
\hspace{0.4cm}
\prod^{k_{0}}_{k = 1} \bigl(g^{(1)}(u)\bigl)^{(2k-1) r_{k}}\,,
\end{equation*}
where the exponents $r_{k}$ satisfy the relation  $\sum^{k_{0}}_{k = 1} (2k - 1) r_{k} = 2(k_{0} + 1) - r$. The numerator of (\ref{producto parciales phi}) is given by a polynomial with integer coefficients and variables  $g^{(s)}(u)$, $z^{(s)}(0)$, with $1 \leqslant s \leqslant k_{0}$ (again, the definition of such polynomial is independent of $g$ y $z$). 

From Remark  \ref{Obs_Sobre la suma en el Teo_FaaDiBruno}, the sum $\sum^{k_{0}}_{k = 1} (2k - 1) r_{k} = 2(k_{0} + 1) - r$ is less or equal to $2(k_{0} + 1) - 2$, since $r \geqslant 2$. Hence,   $\tilde{\mathbf{P}}^{k_{0} + 1} \bigl[g^{(s+1)}(u) \, ; \,  \phi_{x^{s}} (0, u) \bigl]_{s=1}^{k_{0}}$ is a rational transformation with denominator $\bigl(g^{(1)}(u)\bigl)^{2(k_{0} + 1) -2}$, and numerator given by a polynomial with integer coefficients and variables $g^{(1)}(u)$, $g^{(s+1)}(u)$, $z^{(s)}(0)$, where $1 \leqslant s \leqslant k_{0}$ (the definition of such polynomial is independent of $g$ y $z$). Finally, we recall that  $z^{(k_{0} + 1)} (0)= (g \circ \phi)_{x^{k_{0} + 1}} (0, u)$. By 
Fa\`a di Bruno formula (\ref{Ec_Formula de FaaDiBruno con polinomios}), it equals to 
\begin{equation*}
g^{(1)} (u) \, \phi_{x^{k_{0} + 1}} (0, u) 
+ \tilde{\mathbf{P}}^{k_{0} + 1} \bigl[g^{(s+1)}(u) \, ; \,  \phi_{x^{s}} (0, u) \bigl]_{s=1}^{k_{0}} \, .
\end{equation*}

This proves the case $k_{0} + 1$.  Lemma \ref{Lema_Aplicacion Phik_j} is proved.
\end{proof}

\subsection{Proof of Lemma \ref{Lema_Series de composicion A en alpha y beta} on the coefficients of the power series expansion of $A(\alpha, \beta)$}
\label{Aplicacion Phik_j_Demostracion}

In order to prove  Lemma \ref{Lema_Series de composicion A en alpha y beta}, we assume the notation introduced at the beginning of subsection \ref{Dem dos lemas}.

Let 
\begin{equation*}
    A (\alpha, \beta) = \sum_{k \geqslant 0} A_{ k } (\beta (x, u)) \, (\alpha(x, u))^{k}\,,
\end{equation*}
where
\begin{equation*}
    (\alpha(x, u))^{r}=\sum_{s \geqslant r} \alpha_{r , s} x^{s}\,.
\end{equation*}
Thus, for $r \geqslant 1$, the term $A_{ r } (\beta (x, u)) \, (\alpha(x, u))^{r} $ equals to 
\begin{equation*}
\begin{split}
    A_{ r } (\beta (x, u)) \, (\alpha(x, u))^{r}&= 
\biggl(\sum_{s \geqslant 0} \tfrac{1}{s!} \tfrac{\partial^{s} \, (A_{ r } \circ \beta)}{\partial \, x^{s}} \bigl\vert_{(0, u)} \, x^{s} \biggl)
\biggl(\sum_{s \geqslant r} \alpha_{r , s } x^{s}\biggl) \\
= \alpha_{\scriptscriptstyle{ {r} , r}} \, A_{ r } \, x^{r} 
& + \sum_{i \geqslant 1} x^{r+i} \biggl(\alpha_{\scriptscriptstyle{ {r} , r+1}} \, A_{ r } + \sum_{s = 1}^{i} \tfrac{1}{s!} \tfrac{\partial^{s} \, (A_{ r } \circ \beta)}{\partial \, x^{s}} \bigl\vert_{(0, u)} \, \alpha_{\scriptscriptstyle{ {r} , r+i-s}}\biggl) \, .
\end{split}
\end{equation*} 
Therefore, the composition $A (\alpha, \beta)$ equals to $A_{ 0 }$ at the points in the $u$--axis, and the coefficient in $x$ of the power series $A (\alpha, \beta)$ is 
\begin{equation*}
\tfrac{\partial \, (A_{ 0 } \circ \beta)}{\partial \, x} \bigl\vert_{(0, u)} 
\, + \, \alpha_{\scriptscriptstyle{ {1} , 1}} \, A_{ 1 } 
= \beta_{1} \, \tfrac{\text{d} \, A_{ 0 }}{\text{d} \, u} 
\, + \, \alpha_{1} \, A_{ 1 } \, .
\end{equation*}
This proves the fist part of the lemma. 

Let $k \geqslant 2$. Note that the coefficient in $x^{k}$ of the power series 
$A (\alpha, \beta)$ is the sum of the coefficients of $x^{k}$ with respect to the terms $A_{ r } (\beta (x, u)) \, (\alpha(x, u))^{r}$, for $0 \leqslant r \leqslant k$. That is to say, 
\begin{equation}
\label{Ec_Coeficiente de xk en la composicion A en alpha y beta}
\tfrac{1}{k!} \tfrac{\partial^{k} \, (A_{ 0 } \circ \beta)}{\partial \, x^{k}} \bigl\vert_{(0, u)}
\, + \, \sum_{r = 1}^{k-1} \biggl(\alpha_{\scriptscriptstyle{ {r} , k}} \, A_{ r } 
\, + \, \sum_{s = 1}^{k-r} \tfrac{1}{s!} \tfrac{\partial^{s} \, (A_{ r } \circ \beta)}{\partial \, x^{s}} \bigl\vert_{(0, u)} \, \alpha_{\scriptscriptstyle{ {r} , k-s}}\biggl) 
\, + \, \alpha_{\scriptscriptstyle{ {k} , k}} \, A_{ k } \, .
\end{equation}

Using the expression (\ref{Ec_Coeficiente de xk en la composicion A en alpha y beta}) of the coefficient of $x^{k}$,  in the composition $A (\alpha, \beta)$, we will prove the lemma for $k \geqslant 2$. To this purpose we express the partial derivatives  $(A_{ r } \circ \beta)_{x^{s}}$ in terms of the derivatives $A_{ r }$ and the coefficients of $\beta$, and express the coefficients $\alpha_{r , s }(u)$ in terms of the coefficients of $\alpha$.

By using Fa\`a di Bruno's formula (see Theorem \ref{Teorema_Formula de Faa di Bruno}) we obtain that, for $s \geqslant 1$
\begin{equation*}
\tfrac{\partial^{s} \, (A_{ r } \circ \beta)}{\partial \, x^{s}} \bigl\vert_{(0, u)}
= 
\sum \tfrac{s!}{r_{1}! r_{2}! \cdots r_{s}!} \, \tfrac{\text{d}^{l} \, A_{ r }}{\text{d} \, u^{l}}
\, \left(\beta_{1}\right)^{r_{1}} \, \left(\beta_{2}\right)^{r_{2}} \, \cdots \, \left(\beta_{s}\right)^{r_{s}} \, ,
\tag{$\partial^{s}$}
\end{equation*}
where the sum is taken over the non-negative integers $r_{1}, r_{2}, \ldots, r_{s}$ satisfying $r_{1} + 2 r_{2} + \cdots + s r_{s} = s$, and where  $l$ represents its sum $r_{1} + r_{2} + \cdots + r_{s}$.
 
Therefore, $\bigl(A_{ r } \circ \beta\bigl)_{x^{s}} (0, u)$ is a polynomial with natural coefficients and variables    $(A_{ r })^{(1)}, \ldots, (A_{ r })^{(s)}$, $\beta_{1}$, $\beta_{2}, \cdots, \beta_{s}$. In particular, by 
(\ref{Ec_Formula de FaaDiBruno con polinomios}) we get 
\begin{equation}\label{ec: igualdad k derivada}
\tfrac{1}{k!} \, \tfrac{\partial^{k} \, (A_{ 0 } \circ \beta)}{\partial \, x^{k}} \bigl\vert_{(0, u)}
- \, \beta_{k} \, \tfrac{\text{d} \, A_{ 0 }}{\text{d} \, u} 
= \tfrac{1}{k!} \, \tilde{\mathbf{P}}^{k}\Bigl[\tfrac{\text{d}^{s+1} \, A_{ 0 }}{\text{d} \, u^{s+1}} \,  ; \, s! \, \beta_{s} \Bigl]^{k-1}_{s=1} \, ,
\end{equation}  
The equality (\ref{ec: igualdad k derivada}) is expressed by a polynomial with coefficients in the non-negative rational numbers and variables $(A_{ 0 })^{(s+1)}$, $\beta_{s}$, where $1 \leqslant s \leqslant k-1$. 

Since the power series of $(\alpha(x, u))^{r}$ is  $\sum_{s \geqslant r} \alpha_{r , s } x^{s}$ then
\begin{equation*} 
\alpha_{r , s } (u) = \sum \tfrac{r!}{r_{1}! r_{2}!\cdots r_{s}!}
\,
\bigl(\alpha_{1}(u)\bigl)^{r_{1}} \cdots 
\bigl(\alpha_{s}(u)\bigl)^{r_{s}}
\end{equation*}
where the sum is considered in the non-negative integer numbers
 $r_{1}, \ldots, r_{s}$ such that  $r_{1} + \cdots + r_{s}=r$ and such that  $r_{1} + 2 r_{2} + \cdots + s r_{s} = s$. Hence, $\alpha_{r , s }$ is is a polynomial with coefficients in the natural numbers and variables   $\alpha_{1}, \ldots, \alpha_{s}$; moreover, if $r > 1$ it only depends on $\alpha_{1}, \ldots, \alpha_{s-1}$.

Since $\alpha_{1, k} = \alpha_{k}$, $\alpha_{k,k} = (\alpha_{1})^{k}$, the equality (\ref{Ec_Coeficiente de xk en la composicion A en alpha y beta}) expresses that the coefficient in $x^{k}$ in the power series expansion $A (\alpha, \beta)$ equals to 
\begin{equation*}
     \beta_{k} \, \tfrac{\text{d} \, A_{ 0 }}{\text{d} \, u}
\, + \, \alpha_{k} A_{ 1 } 
\, + \, (\alpha_{1})^{k} A_{ k }+ \\
\end{equation*}
\begin{equation*}
\biggl(\tfrac{1}{k!} \tfrac{\partial^{k} \, (A_{ 0 } \circ \beta)}{\partial \, x^{k}} \bigl\vert_{(0, u)} 
\, - \, \beta_{k} \, \tfrac{\text{d} \, A_{ 0 }}{\text{d} \, u}\biggl) \, + \, \sum_{r = 1}^{k-1} \biggl(\alpha_{r, k } \, A_{ r } 
\, + \, \sum_{s = 1}^{k-r} \tfrac{1}{s!} \tfrac{\partial^{s} \, (A_{ r } \circ \beta)}{\partial \, x^{s}} \bigl\vert_{(0, u)} \, \alpha_{r, k-s } \biggl) \, - \, \alpha_{k} \, A_{ 1 } \, .
\end{equation*}

By $(\partial^{s})$, the sum $\tfrac{1}{s!}(A_{ 1 } \circ \beta)_{x^{s}} \, \alpha_{1, k-s}$, for $1 \leqslant s \leqslant k-1$, is a polynomial with natural coefficients and variables $(A_{ 1 })^{(i)}$, $\beta_{i}$ y $\alpha_{i}$, for $1 \leqslant i \leqslant k-1$. Asuming that $1 < r \leqslant k-1$, and using the relations $(\partial^{s})$, it follows that the sum of $\tfrac{1}{s!}(A_{ r } \circ \beta)_{x^{s}} \, \alpha_{r,k-s}$, for $0 \leqslant s \leqslant k-r$, is a polynomial with coefficients in the natural numbers and variables $\beta_{i}$, $\alpha_{i}$, $1 \leqslant i \leqslant k-1$, and the derivatives  $(A_{ r })^{(s)}$, $0 \leqslant s \leqslant k-r$. These properties, together with the relations $(\partial^{k})$ finish the proof of Lemma \ref{Lema_Series de composicion A en alpha y beta}.

\section{\textbf{\emph{Realization of curves as curves of tangencies of foliation pairs}}}
\label{Realizacion de curvas}

\setcounter{equation}{0}

In this section we prove Theorem \ref{teo:realizacion curvas como curvas de tangencias} on realization of a germ of analytic curve in $(\C^{2}, 0)$ having $m+n+1$ pairwise transversal smooth branches as a curve of tangencies of a foliation pair in $\mathcal{N}_{\p}(\mathbf{h}) \times \mathcal{D}_{\mathsf{q}}(\mathcal{I})$. Note in figure \ref{Fig: Figura curvas polares} how the different elements that we have been using will be present.

First, we consider $C$ a germ of analytic curve in $(\C^{2}, 0)$ having $m+n+1$ pairwise transversal smooth branches, whose tangent lines are either $y = p_{i}x$ or $y = q_{j} x$. The parametrizations of the branches of $C$ with respect to $x$ may be expressed as follows:
\begin{equation}
\label{Ec_Curvas_Parametrizaciones}
\begin{split}
C_{p_{i}} & = \{y = \sigma_{p_{i}} (x) = p_{i} x + a_{i} x^{2} + \cdots\} \, , 
\hspace{1.3cm}
1 \leqslant i \leqslant n+1 \, ,
\\
C_{q_{j}} & = \{y = \sigma_{q_{j}} (x) = q_{j} x + a_{j + n + 1} x^{2} + \cdots\} \, ,
\hspace{0.5cm}
1 \leqslant j \leqslant m \, .
\end{split}
\end{equation} 
The collection $(a_{1}, a_{2}, \ldots, a_{m+n+1})$ of coefficients corresponding to the quadratic monomials of the corresponding $m+n+1$ branches of the curve $C$, will be called \textit{the quadratic coefficients of $C$}.

\begin{figure}[ht]
\begin{center}
\includegraphics[scale=0.4]{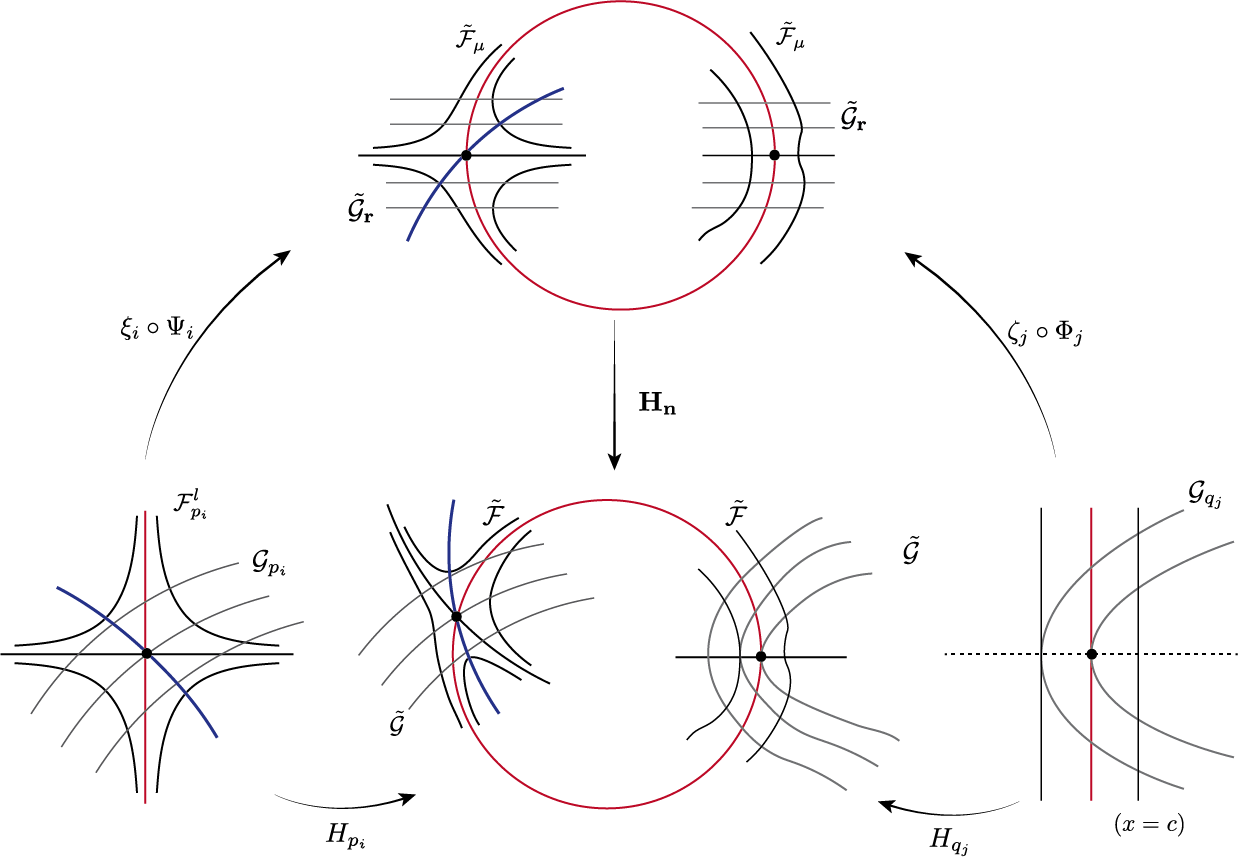}
\caption{Realization of curves as curves of tangencies of foliation pairs.}
\label{Fig: Figura curvas polares}
\end{center}
\end{figure} 

\begin{remark} \label{Obs_DeterminacionFinita}

In 1979, Granger obtained the analytic classification of germs of analytic curves in $(\C^{2}, 0)$ having pairwise transversal smooth branches. In particular, Granger proved the finite determinacy of the analytic type of these curves (see \cite[Proposition 1]{[Gra]}). 

As a consequence of this result, we have the following property for germs of curves having $m+n+1$ pairwise transversal smooth branches, whose tangent lines are either $y = p_{i}x$ or $y = q_{j} x$. Namely, if $C$ and $D$ are two of such curves, and the parametrizations of their branches with respect to $x$ coincide up to their $(m+n-1)$-jet, then there exists a change of coordinates tangent to the identity, which sends the curve $C$ to the curve $D$ (see also \cite[Subsec. 2.3]{[V]}).

\end{remark}

We will now state Theorem \ref{Propo_Igualdad_Curvas_Amalgamado}. This theorem is a reformulation of Theorem \ref{teo:realizacion curvas como curvas de tangencias} and it will be proved by using Theorem \ref{Teorema_Limpieza de TransNorm_Intro_1} and Theorem \ref{Teorema_Parametrizacion_Polar_pi_qj} given in the previous sections. 

For this purpose, we consider the collection $\lambda = (\lambda_{1}, \ldots, \lambda_{n+1}) \in \C^{n+1}$ of Camacho--Sad indices of non dicritical foliations in $\mathcal{N}_{\mathsf{p}}(\mathbf{h})$ (this collection is fixed when the hidden holonomy $\mathbf{h}$ is fixed), and also the quadratic coefficients of the collection of involutions $\mathcal{I}$ of the dicritical foliations in $\mathcal{D}_{\mathsf{q}}(\mathcal{I})$. In the following theorem, assumptions of genericity  on the collections $\lambda$ and $\tau$ are required; these assumptions will be explicitely defined in Subsection \ref{genericity assumptions}.

\begin{theorem}\label{Propo_Igualdad_Curvas_Amalgamado}
Under genericity assumptions on the collection  $\lambda \in \C^{n+1}$ of Camacho--Sad indices and on the collection $\tau \in \C^{m}$ of quadratic coefficients of the involutions,  for each germ of analytic curve $C$ as in \eqref{Ec_Curvas_Parametrizaciones}, up to, perhaps, an analytic
change of coordinates tangent to the identity which modifies the quadratic coefficients $(a_{1}, a_{2}, \ldots, a_{m+n+1})$ of $C$ to suitable ones, the following property holds:

For each natural number $k_{0} \geqslant 1$, there exists $(\F, \G)$ in $\mathcal{N}_{\mathsf{p}}(\mathbf{h}) \times \mathcal{D}_{\mathsf{q}}(\mathcal{I})$ a foliation pair whose normalizing transformations have power series with $k_{0}$-jet \textnormal{($\mathbf{nt_{k_{0}}}$)} as in Theorem \ref{Teorema_Limpieza de TransNorm_Intro_1}, and  such that the branches of its curve of tangencies $\mathcal{P}(\F, \G)$ and the branches of the curve $C$ have parametrizations with the same $(k_{0}+1)$-jet. 
\end{theorem}

Note that Theorem \ref{teo:realizacion curvas como curvas de tangencias} is a consequence of Theorem \ref{Propo_Igualdad_Curvas_Amalgamado}. In fact, considering Remark \ref{Obs_DeterminacionFinita}, Theorem \ref{Propo_Igualdad_Curvas_Amalgamado} allows us to conclude that, for any germ of curve $C$ as in \eqref{Ec_Curvas_Parametrizaciones}, there exists a tangent to the identity change of coordinates $H$ which sends $C$ to the curve of tangencies of a foliation pair $(\F, \G) \in \mathcal{N}_{\mathsf{p}}(\mathbf{h}) \times \mathcal{D}_{\mathsf{q}}(\mathcal{I})$, i.e., $H(C) = \mathcal{P}(\F, \G)$. As a consequence, the curve $C$ is the curve of tangencies of the foliation pair $(H^{-1}(\F), H^{-1}(\G))$ in $\mathcal{N}_{\mathsf{p}}(\mathbf{h}) \times \mathcal{D}_{\mathsf{q}}(\mathcal{I})$.

We will now focus on the proof of Theorem \ref{Propo_Igualdad_Curvas_Amalgamado}. In order to prove it, in Subsection \ref{genericity assumptions}, we will clarify what are the genericity conditions that we assume in its statement. After that, in Subsection \ref{Demostracion Propo_Igualdad_Curvas_Amalgamado}, we will present the proof of Theorem \ref{Propo_Igualdad_Curvas_Amalgamado}.

\subsection{Genericity assumptions for collections $(\lambda,\tau )$ in Theorem \ref{Propo_Igualdad_Curvas_Amalgamado} }\label{genericity assumptions}

\subsection*{Vandermonde matrix and analytic equivalence of curves }
Let us consider the collections $\mathsf{p}= (p_{1}, \ldots, p_{n+1})$ and $\mathsf{q} = (q_{1}, \ldots, q_{m})$ of  $n+1$ and $m$ mutually distinct complex numbers such that, for any  $k=1,\dots,n+1$ and $j= 1,\dots,m$, $p_{k}\neq q_{j}$. For fixed $(\mathsf{p},\mathsf{q}) \in \C^{n+1} \times \C^{m}$ we introduce the corresponding Vandermonde matrix $\mathbf{V} = \mathbf{V} (\mathsf{p},\mathsf{q})$, with $n+1+m$ rows and $4$ columns: 
 
\begin{equation}
\label{Ec_Vandermonde}
 \mathbf{V} (\mathsf{p},\mathsf{q}):=\left[
\begin{array}{c c}
\begin{matrix}
1  & p_{1} & p_{1}^{2} & p_{1}^{3} \\
\vdots & \vdots & \vdots & \vdots  \\
1  & p_{n+1} & p_{n+1}^{2} & p_{n+1}^{3} 
\end{matrix}
\\ \midrule
\begin{matrix}
1  & q_{1} & q_{1}^{2} & q_{1}^{3} \\
\vdots & \vdots & \vdots & \vdots  \\
\; 1 \; & \; q_{m} \; & \; q_{m}^{2} \; & \; q_{m}^{3} \;  
\end{matrix}
\end{array}
\right]
\end{equation}



Let $\hat{C}_{\rho} = \{y = \rho x + \mu x^{2} + \ldots\}$ be a germ of analytic smooth branch in $(\C^{2}, 0)$. We consider $H$ an analytic change of coordinates in $(\C^{2}, 0)$ satisfying 
\begin{equation*}
H(x, y) 
= (x, y) \, + \, (\mathbf{c}, \mathbf{d}) \, + \, \cdots \, ,
\end{equation*}
where
\begin{equation*}
    \mathbf{c} = \mathbf{c}_{0} x^{1} + \mathbf{c}_{1} x y+ \mathbf{c}_{2} y^{2} 
\end{equation*}
and
\begin{equation*}
    \mathbf{d} = \mathbf{d}_{0} x^{1} + \mathbf{d}_{1} x y+ \mathbf{d}_{2} y^{2}\,.
\end{equation*}
Then, $\hat{C}_{\rho}$ and $H(\hat{C}_{\rho}) = \{y = \rho x + \tilde{\mu} x^{2} + \ldots\}$ have the same linear coefficient $\rho$, while the coefficients of their quadratic monomials $\mu$ and $\tilde{\mu}$ are related as follows
\begin{equation*}
\tilde{\mu} \, - \, \mu \, = \, 
\mathbf{d}_{0} \, + \, (\mathbf{d}_{1} - \mathbf{c}_{0}) \, \rho 
\, + \, (\mathbf{d}_{2} - \mathbf{c}_{1}) \,\rho^{2}
\, - \, \mathbf{c}_{2} \, \rho^{3} \, 
\end{equation*}
(see \cite{[Gra]}, Lemma $1$). For curves with several branches, this property has the following consequence: let $C$ be an analytic curve in $(\C^{2}, 0)$, whose branches are as in \eqref{Ec_Curvas_Parametrizaciones}; we consider the quadratic coefficients of its branches $a = (a_{1}, a_{2}, \ldots, a_{m+n+1})$. If $H$ is the change of coordinates described previously, then the collection of coefficients of the quadratic monomials of the curve $H(C)$ are equal to 
\begin{equation}
a \, + \, \mathbf{V} \left(\mathbf{d}_{0}, \, \mathbf{d}_{1} - \mathbf{c}_{0}, \, \mathbf{d}_{2} - \mathbf{c}_{1}, \, - \mathbf{c}_{2} \right) \,,
\end{equation}
where $\mathbf{V}$ is the Vandermonde matrix defined in \eqref{Ec_Vandermonde}. 

\begin{remark}\label{rem: esp afin coeficientes imagen}
    The affine space $a + \text{Im} (\mathbf{V})$ consists of the coefficients of the quadratic monomials of the curves strictly analytically equivalent to $C$. In other words, $a + \text{Im} (\mathbf{V})$ is the set of the coefficients of the quadratic monomials of the curves of the form $H(C)$, being $H$ an analytic change of coordinates tangent to the identity.
\end{remark}

\subsection*{Parametrizations of curves of tangencies in terms of matrices}
For fixed $(\mathsf{p},\mathsf{q}) \in \C^{n+1} \times \C^{m}$ we  consider for any foliation pair in $\mathcal{N}_{\p}(\mathbf{h}) \times \mathcal{D}_{\mathsf{q}}(\mathcal{I})$, the collections $\lambda = (\lambda_{1}, \ldots, \lambda_{n+1}) \in \C^{n+1}$  of Camacho--Sad indices, and $\tau = (\tau_{1}, \ldots, \tau_{m}) \in \C^{m}$ the collection of the quadratic coefficients of the involutions $\mathcal{I} = (I_{1}, \ldots, I_{m})$, i.e., $I_{j} (u) - q_{j}$ is equal to $-(u-q_{j}) + \tau_{j} (u-q_{j})^{2} + \cdots$.

 Note that, since the map $g_{j}(u)$ is defined as $(q_{j} - u)(I_{j}(u) - q_{j})$, its derivative $g_{j}^{(3)}(q_{j})$ is equal to $-3! \tau_{j}$.

We define $\theta_{\scriptscriptstyle{j, \, k}} := -\tfrac{3}{2} \tau_{j} + k \cdot (\varsigma_{\scriptscriptstyle{j, \, 1}})^{(1)} (q_{j})$, where $\varsigma_{j} = \varsigma_{\scriptscriptstyle{j\, 1}}(u) x + \varsigma_{\scriptscriptstyle{j\, 2}}(u) x^{2} + \cdots$ is the first coordinate of the biholomorphism $\zeta_{j}$ (see (3.8)). Then, for $k \geqslant 1$ we define the square matrix $\mathbf{A}_{k} = \mathbf{A}_{k} (\mathsf{p}, \mathsf{q}; \lambda, \tau)$ of dimension $m + n +1$ 
\begin{equation}
\label{Ec_MatrizAk}
\mathbf{A}_{k}:=
\left[
\begin{array}{c|c} 
\begin{matrix}
\mbox{\footnotesize $1 - k\lambda_{1}$} & 0 & \cdots & 0 \\
0 & \mbox{\footnotesize $1 - k\lambda_{2}$} & \cdots & 0 \\
\vdots & \vdots & \ddots & \vdots \\
0 & 0 & \cdots & \mbox{\footnotesize $1 - k\lambda_{n+1}$}
\end{matrix}
& 
\begin{matrix}
\tfrac{1}{p_{1} - q_{1}} 
& \tfrac{1}{p_{1} - q_{2}} & \cdots 
& \tfrac{1}{p_{1} - q_{m}}
\\
\tfrac{1}{p_{2} - q_{1}} 
& \tfrac{1}{p_{2} - q_{2}} & \cdots 
& \tfrac{1}{p_{2} - q_{m}} \\
\vdots & \vdots & \ddots & \vdots \\
\tfrac{1}{p_{n+1} - q_{1}} 
& \tfrac{1}{p_{n+1} - q_{2}} & \cdots 
& \tfrac{1}{p_{n+1} - q_{m}}
\end{matrix}
\\ \midrule 
\begin{matrix}
\hspace{0.4cm} 0 \hspace{0.4cm}  
& \hspace{0.4cm} 0 \hspace{0.4cm} & \cdots 
& \hspace{0.4cm} 0 \hspace{0.4cm} \\
0 & 0 & \cdots & 0 \\
\vdots & \vdots & \ddots & \vdots \\
0 & 0 & \cdots & 0
\end{matrix}
& 
\begin{matrix}
\mbox{\footnotesize $\theta_{\scriptscriptstyle{1, \, k}}$} 
& \tfrac{1}{q_{1} - q_{2}} & \cdots 
& \tfrac{1}{q_{1} - q_{m}}
\\
\tfrac{1}{q_{2} - q_{1}} 
& \mbox{\footnotesize $\theta_{\scriptscriptstyle{2, \, k}}$} & \cdots 
& \tfrac{1}{q_{2} - q_{m}} \\
\vdots & \vdots & \ddots & \vdots \\
\tfrac{1}{q_{m} - q_{1}} 
& \tfrac{1}{q_{m} - q_{2}} & \cdots 
& \mbox{\footnotesize $\theta_{\scriptscriptstyle{m, \, k}}$}
\end{matrix}
\end{array}
\right]
\end{equation}

\begin{remark}
\label{Obs_TeoMatricialBis}
The square matrices \eqref{Ec_MatrizAk} give us expressions of the curves of tangencies of the foliation pairs $(\F, \G)$ considered in Theorem \ref{Teorema_Parametrizacion_Polar_pi_qj}, as we shall see below.

The coefficients $\mathbf{c}_{\scriptscriptstyle{p_{i}, k}}$ (see (\ref{Pageref_Pik_i_j}) of the power series $\pi_{p_{i}} = p_{i} + \textstyle \sum_{r \geqslant 1}  \mathbf{c}_{\scriptscriptstyle{p_{i}, r}}\,x^{r}$ of the parametrizations by $x$ of the curve of tangencies $\CP_{p_{i}} (\tilde{\F}, \tilde{\G})$ satisfy,  for $1 \leqslant k \leqslant k_{0}$, 
\begin{equation*}
\mathbf{c}_{\scriptscriptstyle{p_{i}, k}}\, 
 = (1 - k\lambda_{i}) s_{\scriptscriptstyle{{i} , k}} 
\, - \,  \sum_{\jmb =1}^{m} \tfrac{z_{\scriptscriptstyle{{\jmb} , k}}}{2(p_{i}-q_{\jmb})} 
\,  + \,  \boldsymbol{\pi}_{\scriptscriptstyle{p_{i}\, , k}}  (p_{i})
 \, 
 \end{equation*}
(see property \emph{a)} in Theorem \ref{Teorema_Parametrizacion_Polar_pi_qj}). Analogously, the coefficients $\mathbf{c}_{\scriptscriptstyle{q_{j}, k}}$ (see (\ref{Pageref_Pik_i_j_bis})) of the power series $    \pi_{q_{j}} = q_{j} + \textstyle \sum_{r \geqslant 1}  \mathbf{c}_{\scriptscriptstyle{q_{j}, r}}\, x^{r}$ of the parametrizations by $x$ of the curve of tangencies $\CP_{q_{j}} (\tilde{\F}, \tilde{\G})$ satisfy, for $1 \leqslant k \leqslant k_{0}$,  
\begin{equation*}
\begin{split}
\hspace{1.0cm}\mathbf{c}_{\scriptscriptstyle{q_{j}, k}}\, 
 &= 
 - \tfrac{z_{\scriptscriptstyle{j, k}}}{2} \Bigl(k(\varsigma_{\scriptscriptstyle{j, 1}})^{(1)}(q_{j}) 
\, + \, \tfrac{g_{j}^{(3)}(q_{j})}{4}\Bigl)
\, - \, \sum_{\jmb \neq j}^{m} \tfrac{z_{\scriptscriptstyle{{\jmb} , k}}}{2(q_{j} - q_{\jmb})} 
\, + \, \boldsymbol{\pi}_{\scriptscriptstyle{q_{j}\, , k}}  (q_{j}) \\
&=
- \tfrac{z_{\scriptscriptstyle{j, k}}}{2} \ \theta_{\scriptscriptstyle{j, \, k}}
\, - \, \sum_{\jmb \neq j}^{m} \tfrac{z_{\scriptscriptstyle{{\jmb} , k}}}{2(q_{j} - q_{\jmb})} 
\, + \, \boldsymbol{\pi}_{\scriptscriptstyle{q_{j}\, , k}}  (q_{j})
\,,
\end{split}
\end{equation*}
where $\theta_{\scriptscriptstyle{j, \, k}}$ is, by definition,  $k(\varsigma_{\scriptscriptstyle{j, 1}})^{(1)}(q_{j}) 
\, - \, \tfrac{3}{2} \tau_{j}$ (see property \emph{b)} in Theorem \ref{Teorema_Parametrizacion_Polar_pi_qj}). Thus, the differences $\mathbf{c}_{\scriptscriptstyle{p_{\imb}, k}} \,  - \, \boldsymbol{\pi}_{\scriptscriptstyle{p_{\imb}, k}}(p_{\imb})$ and  $\mathbf{c}_{\scriptscriptstyle{q_{\jmb}, k}} \, - \, \boldsymbol{\pi}_{\scriptscriptstyle{q_{\jmb}, k}}(q_{\jmb})$  are linear functions in terms of the coefficients $s_{\scriptscriptstyle{{\imb}, \, k}}$, $z_{\scriptscriptstyle{{\jmb}, \, k}}$, where 
\begin{equation*}
    s_{i} = \sum_{r \geqslant 1} s_{\scriptscriptstyle{i, \, r}} x^{r}\,,\quad z_{j} = \sum_{r \geqslant 1} z_{\scriptscriptstyle{j, \, r}} x^{r}\,,\quad z_{\scriptscriptstyle{j, \, 1}} \in \C^{\ast}\,,
\end{equation*}
are analytic invariants of the pair $(\F, \G)$. More specifically, these relations are expressed by the square matrices $\mathbf{A}_{k}$ defined in \eqref{Ec_MatrizAk}:
\begin{equation*} 
\begin{split}
& \hspace{1.9cm}
\mathbf{A}_{k} \left( \left(s_{\scriptscriptstyle{{\imb}, \, k}}\right)_{\scriptscriptstyle{\imb = 1}}^{\scriptscriptstyle{n+1}}, \left(-\tfrac{1}{2} z_{\scriptscriptstyle{{\jmb}, \, k}}\right)_{\scriptscriptstyle{\jmb = 1}}^{\scriptscriptstyle{m}}\right) \\
& \, = \, 
\left(\left(\mathbf{c}_{\scriptscriptstyle{p_{\imb}, k}} \,  - \, \boldsymbol{\pi}_{\scriptscriptstyle{p_{\imb}, k}}(p_{\imb})\right)_{\scriptscriptstyle{\imb = 1}}^{\scriptscriptstyle{n+1}}, \, 
\left(\mathbf{c}_{\scriptscriptstyle{q_{\jmb}, k}} \, - \, \boldsymbol{\pi}_{\scriptscriptstyle{q_{\jmb}, k}}(q_{\jmb})\right)_{\scriptscriptstyle{\jmb = 1}}^{\scriptscriptstyle{m}} \right) \, .
\end{split} 
\end{equation*}

\end{remark}

\subsection*{Genericity assumptions for collections ($\lambda, \tau$)}

\begin{lemma}
\label{lem: abierto Zariski A}
For such fixed collection of pairwise different complex numbers $(\mathsf{p},\mathsf{q}) \in \C^{n+1} \times \C^{m}$, we consider a natural number $k_{0} \geqslant 1$.

There exists a Zariski open set $\mathcal{A}$ of $\C^{n+1}\times \C^{m}$, depending on $\q$ and $k_{0}$, in which any element $(\lambda, \tau) \in \mathcal{A}$ satisfies the following properties:
\begin{itemize}

\item[a)]  For each $1 \leqslant k \leqslant k_{0}$, the square matrix
$\mathbf{A}_{k} = \mathbf{A}_{k}(\mathsf{p},\mathsf{q}; \lambda, \tau)$ defined in \eqref{Ec_MatrizAk}, is an invertible matrix.

\item[b)] For any $a \in \C^{n+1} \times \C^{m}$, the affine space $a + \text{Im} (\mathbf{V})$ defined by the image of the Vandermonde matrix $\mathbf{V}$ given in \eqref{Ec_Vandermonde}, satisfies
\begin{equation*}
a \, + \, \text{Im} (\mathbf{V}) \,  
\nsubseteq
\, 
\bigcup_{j=1}^{m} \, \mathbf{A}_{1} (\C^{n+1} \times \C^{m}_{j}) \, ,
\end{equation*}  
where $\C^{m}_{j}$ consists of the elements $(\mathsf{z}_{1}, \ldots, \mathsf{z}_{m} ) \in \C^{m}$ whose $j$th coordinate is zero, i.e., $\mathsf{z}_{j} = 0$.
\end{itemize}
\end{lemma}

\begin{remark}\label{obs: abierto Zariski A y polinomio P}
    The Zariski open set $\mathcal{A}$ results the complement of the zeros of a polynomial
\begin{equation}\label{ec: mathcal P}
\mathcal{P}(k_{0}, \lambda, \tau ):=\mathbf{P}_{k_{0}} (\tau_{1}, \tau_{2}, \ldots, \tau_{m}) \, \cdot\prod_{\substack{1 \leqslant k \leqslant k_{0} \\ 1 \leqslant i \leqslant n+1 }} \, \bigl(1 \, - \, k\lambda_{i}\bigl) \, ,
\end{equation}
where $\mathbf{P}_{k_{0}} (\tau_{1}, \tau_{2}, \ldots, \tau_{m})$ is a nonzero polynomial in the variables $\tau_{1}, \tau_{2}, \ldots, \tau_{m}$, whose coefficients are rational functions in the coordinates $q_{j}$ of $\q$, with denominators equal to products of differences  $q_{j} - q_{l}$, $j \neq l$. A precise description of the polynomial $\mathbf{P}_{k_{0}} (\tau_{1}, \tau_{2}, \ldots, \tau_{m})$ is given in Subsection \ref{Descripcion polinomio cal P} 
\end{remark}

\subsection{Proof of Theorem \ref{Propo_Igualdad_Curvas_Amalgamado}}
\label{Demostracion Propo_Igualdad_Curvas_Amalgamado}
Theorem \ref{Propo_Igualdad_Curvas_Amalgamado} will follow from two results: Proposition \ref{Propo_Cuadratico} and Proposition \ref{Propo_Igualdad_Curvas}. 

Proposition \ref{Propo_Cuadratico} allows us to modify the collection $a = (a_{1}, a_{2}, \ldots, a_{m+n+1})$ of quadratic coefficients of an analytic curve $C$ by means of a matrix $\mathbf{A}_{1} = \mathbf{A}_{1} (\mathsf{p}, \mathsf{q}; \lambda, \tau)$ of dimension $m + n +1$. On the other hand, Proposition \ref{Propo_Igualdad_Curvas} states that, for curves having quadratic coefficients as in Proposition \ref{Propo_Cuadratico}, any finite jet of the parametrizations of its branches is realized by the parametrizations of the branches of the curve of tangencies of some foliation pair in $\mathcal{N}_{\mathsf{p}}(\mathbf{h}) \times \mathcal{D}_{\mathsf{q}}(\mathcal{I})$.

\begin{proposition} 
\label{Propo_Cuadratico}
For the fixed collection $(\mathsf{p},\mathsf{q}) \in \C^{n+1} \times \C^{m}$, let  $(\lambda, \tau)$ be an element of the Zariski open set $\mathcal{A}$ described in Subsection \ref{genericity assumptions}, for  $k_{0} \geqslant 1$.

Then any germ of curve $C$ with branches as in \eqref{Ec_Curvas_Parametrizaciones}, except for an analytic change of coordinates tangent to the identity, has a collection of quadratic coefficients $(a_{1}, a_{2}, \ldots, a_{m+n+1})$ in the image of $\C^{n+1} \times (\C^{\ast})^{m}$ under the matrix $\mathbf{A}_{1}$, defined in \eqref{Ec_MatrizAk}. That is, there exist $(\mathsf{s}, \mathsf{z}) \in \C^{n+1}\times (\C^{\ast})^{m}$ such that 
\begin{equation*}
\mathbf{A}_{1} \left(\mathsf{s}, \mathsf{z}\right)
= \left(a_{1}, a_{2}, \ldots, a_{m+n+1}\right) \, .
\end{equation*}
\end{proposition} 

\begin{proof}[Proof of Proposition \ref{Propo_Cuadratico}]
Let $C$ be a germ of analytic curve in $(\C^{2}, 0)$, with branches as in \eqref{Ec_Curvas_Parametrizaciones}. Let us consider $a = (a_{1}, a_{2}, \ldots, a_{m+n+1})$, the collection of coefficients of the quadratic monomials of its branches.

Let $\lambda = (\lambda_{1}, \ldots, \lambda_{n+1}) \in \C^{n+1}$ be the collection of Camacho--Sad indices, and $\tau = (\tau_{1}, \ldots, \tau_{1}) \in \C^{m}$ be the collection of quadratic coefficients of the collection of involutions $\mathcal{I}$. We assume that $(\lambda, \tau)$ is an element of the Zariski open set $\mathcal{A}$. Hence, the matrix $\mathbf{A}_{1}$ defined in \eqref{Ec_MatrizAk} is invertible. Moreover,     
\begin{equation*}
a \, + \, \text{Im} (\mathbf{V}) \,  
\nsubseteq
\, 
\bigcup_{j=1}^{m} \, \mathbf{A}_{1} (\C^{n+1} \times \C^{m}_{j}) \, ,
\end{equation*}  
where $\C^{m}_{j}$ consists of the elements of $\C^{m}$ whose $j$th coordinate is zero. Therefore, there exists $b = (b_{1}, b_{2}, \ldots, b_{m+n+1})$ in the affine space $a + \text{Im} (\mathbf{V})$ that is not an element of the union of subspaces $\cup_{j=1}^{m} \mathbf{A}_{1} (\C^{n+1} \times \C^{m}_{j})$. 

Since $\mathbf{A}_{1}$ is a surjective linear map, $b$ belongs to the image of $\C^{n+1} \times (\C^{\ast})^{m}$ under $\mathbf{A}_{1}$. The proof of Proposition \ref{Propo_Cuadratico} follows since, by Remark \ref{rem: esp afin coeficientes imagen}, if $b \in a + \text{Im} (\mathbf{V})$ then there exists $H$ an analytic change of coordinates tangent to the identity such that $H(C)$ has quadratic coefficients equal to $b$. This finishes the proof of Proposition \ref{Propo_Cuadratico}.
\end{proof}

\begin{proposition}
\label{Propo_Igualdad_Curvas}
For the fixed collection $(\mathsf{p},\mathsf{q}) \in \C^{n+1} \times \C^{m}$, let  $(\lambda, \tau)$ be an element of the Zariski open set $\mathcal{A}$ described in Subsection \ref{genericity assumptions}, for  $k_{0} \geqslant 1$. 

Let us consider a germ of curve $C$ with branches as in \eqref{Ec_Curvas_Parametrizaciones}, whose quadratic coefficients $(a_{1}, a_{2}, \ldots, a_{m+n+1})$ are in the image of $\C^{n+1} \times (\C^{\ast})^{m}$ under the invertible matrix $\mathbf{A}_{1}=\mathbf{A}_{1} \left(\mathsf{p}, \mathsf{q}; \lambda, \tau \right)$,
 \begin{equation*}
\mathbf{A}_{1} \left(\mathsf{s}, \mathsf{z}\right)
= \left(a_{1}, a_{2}, \ldots, a_{m+n+1}\right) \,, \quad \left(\mathsf{s}, \mathsf{z}\right) \in \C^{n+1}\times (\C^{\ast})^{m}.
\end{equation*} 
Then there exists a foliation pair $(\mathcal{F}, \mathcal{G}) \in \mathcal{N}_{\mathsf{p}}(\mathbf{h}) \times \mathcal{D}_{\mathsf{q}}(\mathcal{I})$, whose normalizing transformations have power series with $k_{0}$-jet $(\mathbf{nt_{k_{0}}})$ as in Theorem \ref{Teorema_Limpieza de TransNorm_Intro_1}, and such that the branches of its curve of tangencies, $\mathcal{P}(\mathcal{F}, \mathcal{G})$, and the branches of the curve $C$ have parametrizations with the same $(k_{0} +1)$-jet. 
\end{proposition}

\subsection*{Proof of Proposition \ref{Propo_Igualdad_Curvas}}

Let $C$ be a germ of analytic curve in $(\C^{2}, 0)$, whose branches are as in \eqref{Ec_Curvas_Parametrizaciones}. We denote by $\tilde{C}$ its first blow-up. Its branches in the coordinate chart $(x, u = y/x)$,
\begin{equation*}
\begin{split}
\tilde{C}_{p_{i}} = \{u = \tilde{\sigma}_{p_{i}} (x)\} \, , 
\hspace{1cm}
\tilde{C}_{q_{j}} = \{u = \tilde{\sigma}_{q_{j}} (x)\} \, ,
\end{split}
\end{equation*}
satisfy $\sigma_{p_{i}} (x) = x \tilde{\sigma}_{p_{i}} (x)$, $\sigma_{q_{j}} (x) = x \tilde{\sigma}_{q_{j}} (x)$. Let us consider the power series of the parametrizations $\tilde{\sigma}_{p_{i}}$, $\tilde{\sigma}_{q_{j}}$,
\begin{equation*}
\tilde{\sigma}_{p_{i}} (x) 
\, = \, p_{i} \, + \, \sigma_{\scriptscriptstyle{{p_{i}}, 1}} x 
\, + \, \sigma_{\scriptscriptstyle{{p_{i}}, 2}} x^{2} 
\, + \, \cdots \, , 
\hspace{0.5cm}
\tilde{\sigma}_{q_{j}} (x) 
\, = \, q_{j} \, + \, \sigma_{\scriptscriptstyle{{q_{j}}, 1}} x 
\, + \, \sigma_{\scriptscriptstyle{{q_{j}}, 2}} x^{2} 
\, + \, \cdots 
\end{equation*} 
where $\sigma_{\scriptscriptstyle{{p_{i}}, 1}} = a_{i}$, $\sigma_{\scriptscriptstyle{{q_{j}}, 1}} = a_{j+n+1}$ are the quadratic coefficients stated in the assumptions of Proposition \ref{Propo_Igualdad_Curvas}.

To achieve the proof, we need to evidence the existence of a foliation pair $(\F, \G)$, where  $\F\in \mathcal{N}_{\p}(\mathbf{h})$ and  $\G \in \mathcal{D}_{\mathsf{q}}(\mathcal{I})$, having the following properties:

\begin{itemize}
\item[$a)$] \label{Propo_PropA} The normalizing transformations of $(\F, \G)$ have power series with $k_{0}$-jet as in $(\mathbf{nt_{k_{0}}})$ (see Theorem \ref{Teorema_Limpieza de TransNorm_Intro_1} ).

\item[$b)$] \label{Propo_PropB} 
The branches of the curve of tangencies $\mathcal{P}(\tilde{\mathcal{F}}, \tilde{\mathcal{G}})$ of the first blow-up of $(\F, \G)$, have parametrizations
\begin{equation*}
\begin{split}
\CP_{p_{i}} (\tilde{\F}, \tilde{\G})
= \{u = \pi_{p_{i}} (x)\} \, , 
\hspace{1cm}
\CP_{q_{j}} (\tilde{\F}, \tilde{\G}) = \{u = \pi_{q_{j}} (x)\} \, ,
\end{split}
\end{equation*}
as in Theorem 4.1. Thus, these parametrizations have the same $k_{0}$-jet as the parametrizations of the branches of $\tilde{C}$. That is,
\begin{equation*}
\textnormal{j}^{k_{0}} (\pi_{p_{i}})
\, = \,
\textnormal{j}^{k_{0}} (\tilde{\sigma}_{p_{i}})\, ,
\hspace{1cm}
\textnormal{j}^{k_{0}} (\pi_{q_{j}})
\, = \,
\textnormal{j}^{k_{0}} (\tilde{\sigma}_{q_{j}}) \, ,
\end{equation*}
for $1 \leqslant i \leqslant n+1$, $1 \leqslant j \leqslant m$.
\end{itemize}

To prove the existence of a foliation pair $(\F, \G)\in \mathcal{N}_{\p}(\mathbf{h}) \times \mathcal{D}_{\mathsf{q}}(\mathcal{I})$ satisfying the properties $a)$ and $b)$, we will look for its local models. 

For this purpose, let us recall Theorem \ref{Teorema_Parametrizacion_Polar_pi_qj}, the corresponding foliation pair $(\F, \G)$ considered there, and also, the parametrizations by $x$ of the curves of tangencies $\CP_{p_{i}} (\tilde{\F}, \tilde{\G})$ and $\CP_{q_{j}} (\tilde{\F}, \tilde{\G})$, respectively,
\begin{equation*}
       \pi_{p_{i}} = p_{i} + \textstyle \sum_{r \geqslant 1}  \mathbf{c}_{\scriptscriptstyle{p_{i}, r}}\,x^{r} \, ,
\hspace{1cm}       
\pi_{q_{j}} = q_{j} + \textstyle \sum_{r \geqslant 1}  \mathbf{c}_{\scriptscriptstyle{q_{j}, r}}\, x^{r} \, .
\end{equation*}
By Remark \ref{Obs_TeoMatricialBis}, the coefficients of these parametrizations are expressed in terms of the coefficients $s_{\scriptscriptstyle{{\imb}, \, k}}$, $z_{\scriptscriptstyle{{\jmb}, \, k}}$  of the analytic invariants of the pair $(\F, \G)$: 
\begin{equation*}
s_{i} = \sum_{r \geqslant 1} s_{\scriptscriptstyle{i, \, r}} x^{r}\,,\quad z_{j} = \sum_{r \geqslant 1} z_{\scriptscriptstyle{j, \, r}} x^{r}\,,\quad z_{\scriptscriptstyle{j, \, 1}} \in \C^{\ast}\,. 
\end{equation*}
More precisely, for $1 \leqslant k \leqslant k_{0}$, the differences $\mathbf{c}_{\scriptscriptstyle{p_{\imb}, k}} \,  - \, \boldsymbol{\pi}_{\scriptscriptstyle{p_{\imb}, k}}(p_{\imb})$ and  $\mathbf{c}_{\scriptscriptstyle{q_{\jmb}, k}} \, - \, \boldsymbol{\pi}_{\scriptscriptstyle{q_{\jmb}, k}}(q_{\jmb})$  are expressed by the square matrices $\mathbf{A}_{k}$ defined in \eqref{Ec_MatrizAk}:
\begin{equation} 
\label{Ec_TeoMatricialBis}
\begin{split}
& \hspace{1.9cm}
\mathbf{A}_{k} \left( \left(s_{\scriptscriptstyle{{\imb}, \, k}}\right)_{\scriptscriptstyle{\imb = 1}}^{\scriptscriptstyle{n+1}}, \left(-\tfrac{1}{2} z_{\scriptscriptstyle{{\jmb}, \, k}}\right)_{\scriptscriptstyle{\jmb = 1}}^{\scriptscriptstyle{m}}\right) \\
& \, = \, 
\left(\left(\mathbf{c}_{\scriptscriptstyle{p_{\imb}, k}} \,  - \, \boldsymbol{\pi}_{\scriptscriptstyle{p_{\imb}, k}}(p_{\imb})\right)_{\scriptscriptstyle{\imb = 1}}^{\scriptscriptstyle{n+1}}, \, 
\left(\mathbf{c}_{\scriptscriptstyle{q_{\jmb}, k}} \, - \, \boldsymbol{\pi}_{\scriptscriptstyle{q_{\jmb}, k}}(q_{\jmb})\right)_{\scriptscriptstyle{\jmb = 1}}^{\scriptscriptstyle{m}} \right) \, ,
\end{split} 
\end{equation}

\begin{remark}
\label{Obs_DependenciaPi}
Let us consider the functional maps $\boldsymbol{\pi}_{\scriptscriptstyle{p_{i}\, , k}}$, $\boldsymbol{\pi}_{\scriptscriptstyle{q_{j}\, , k}}$ stated in Proposition \ref{Propo_Igualdad_Curvas}. Since they are identically zero for $k=1$, we focus on the case $2 \leqslant k \leqslant k_{0}$, $k_{0}\geq m+n+1$.

These functional maps $\boldsymbol{\pi}_{\scriptscriptstyle{p_{i}\, , k}}$, $\boldsymbol{\pi}_{\scriptscriptstyle{q_{j}\, , k}}$ are evaluated at complex numbers and functions. Among them, the only ones depending on the foliation pair $(\F, \G)$ are the coefficients of the $(k-1)$-jet of their analytic invariants $s_{\scriptscriptstyle{{\imb},k}}$, $z_{\scriptscriptstyle{{\jmb},k}}$, $1 \leqslant \imb \leqslant n+1$, $1 \leqslant \jmb \leqslant m$. This follows from the fact that the coefficients of the functions $\varepsilon_{\scriptscriptstyle{{\imb}}}$, $\varsigma_{\scriptscriptstyle{{\jmb}}}$ are completely determined by the foliation ${\F}_{\scriptscriptstyle \mu}$ (see (\ref{Ec_Series de Psi y Xi}) and (\ref{Ec_Series de Phi y Zeta}), and the Camacho--Sad indices $\lambda_{\scriptscriptstyle{{\imb}}}$ and the maps $g_{\scriptscriptstyle{{\jmb}}}$ (defined by the collection of involutions $\mathcal{I}$), are the same for any foliation pair in $\mathcal{N}_{\p}(\mathbf{h}) \times \mathcal{D}_{\mathsf{q}}(\mathcal{I})$.

\end{remark}

We now proceed  by induction to prove the existence of the foliation pair.

\begin{itemize}
    \item [i)] \emph{Base case.} By assumption, there exist a collection of $n+1$ complex numbers, $\mathsf{s} = (\mathsf{s}_{1}, \ldots, \mathsf{s}_{n+1}) \in \C^{n+1}$, and a collection of $m$ nonzero complex numbers $\mathsf{z} = (\mathsf{z}_{1}, \ldots, \mathsf{z}_{m}) \in (\C^{\ast})^{m}$ such that
\begin{equation*}
\begin{split}
\mathbf{A}_{1} \left(\mathsf{s}, \mathsf{z}\right)
= \left(a_{1}, a_{2}, \ldots, a_{m+n+1}\right)
= \left((\sigma_{\scriptscriptstyle{{p_{\imb}}, 1}})_{\scriptscriptstyle{{\imb} = 1}}^{\scriptscriptstyle{n+1}}, (\sigma_{\scriptscriptstyle{{q_{\jmb}}, 1}})_{\scriptscriptstyle{{\jmb} = 1}}^{\scriptscriptstyle{m}}\right) \, ,
\end{split}
\end{equation*}
The previous equalities are equivalent to the equality \eqref{Ec_TeoMatricialBis} for the case $k = 1$, by taking $s_{\scriptscriptstyle{i, 1}} := \mathsf{s}_{i} \in \C$ and $z_{\scriptscriptstyle{j, 1}} := - 2 \mathsf{z}_{j} \in \C^{\ast}$. Indeed, this follows from the fact that the constants $\boldsymbol{\pi}_{\scriptscriptstyle{p_{i}, \, 1}} (p_{i})$, $\boldsymbol{\pi}_{\scriptscriptstyle{q_{j}, \, 1}} (q_{j})$ are equal to zero, as stated in Theorem \ref{Teorema_Parametrizacion_Polar_pi_qj}.

\item[ii)] \emph{Induction step.} We prove, for $1 \leqslant i \leqslant n+1$, $1 \leqslant j \leqslant m$, the existence of complex numbers
\begin{equation*}
s_{\scriptscriptstyle{i, 1}}, s_{\scriptscriptstyle{i, 2}}, \ldots, s_{\scriptscriptstyle{i, k_{0}}} \in \C  \,  \,
\text{and} \, \,
z_{\scriptscriptstyle{j, 1}}, z_{\scriptscriptstyle{j, 2}}, \ldots, z_{\scriptscriptstyle{j, k_{0}}} \in \C  \, ,
 \, \,
\text{with} \, \,
z_{\scriptscriptstyle{j, 1}} \neq 0 \, ,
\end{equation*}
such that, for $1 \leqslant k \leqslant k_{0}$, 
\begin{equation} 
\label{Ec_PropoMatrizAk}
\begin{split}
& \hspace{1.9cm}
\mathbf{A}_{k} \left( \left(s_{\scriptscriptstyle{{\imb}, \, k}}\right)_{\scriptscriptstyle{\imb = 1}}^{\scriptscriptstyle{n+1}}, \left(-\tfrac{1}{2} z_{\scriptscriptstyle{{\jmb}, \, k}}\right)_{\scriptscriptstyle{\jmb = 1}}^{\scriptscriptstyle{m}}\right) \\
& \, = \, 
\left(\left(\sigma_{\scriptscriptstyle{p_{\imb}, k}} \,  - \, \boldsymbol{\pi}_{\scriptscriptstyle{p_{\imb}, k}}(p_{\imb})\right)_{\scriptscriptstyle{\imb = 1}}^{\scriptscriptstyle{n+1}}, \, 
\left(\sigma_{\scriptscriptstyle{q_{\jmb}, k}} \, - \, \boldsymbol{\pi}_{\scriptscriptstyle{q_{\jmb}, k}}(q_{\jmb})\right)_{\scriptscriptstyle{\jmb = 1}}^{\scriptscriptstyle{m}} \right) \, .
\end{split} 
\end{equation}
where $\boldsymbol{\pi}_{\scriptscriptstyle{p_{\imb}, k}}(p_{\imb})$, $\boldsymbol{\pi}_{\scriptscriptstyle{q_{\jmb}, k}} (q_{\jmb})$ are the complex numbers described in Theorem \ref{Teorema_Parametrizacion_Polar_pi_qj}. The proof by induction of such complex numbers is a direct consequence of the following properties:

\begin{itemize}
\item[1.]  According to Remark \ref{Obs_DependenciaPi}, the complex numbers $\boldsymbol{\pi}_{\scriptscriptstyle{p_{\imb}, k}}(p_{\imb})$, $\boldsymbol{\pi}_{\scriptscriptstyle{q_{\jmb}, k}} (q_{\jmb})$ are determined by the complex numbers $s_{\scriptscriptstyle{\imb, 1}}, \ldots, s_{\scriptscriptstyle{\imb, k-1}}$, and by the complex numbers $z_{\scriptscriptstyle{\jmb, 1}}, \ldots, z_{\scriptscriptstyle{\jmb, k-1}}$.

\item[2.] Since $(\lambda, \tau)$ belongs to de Zariski open set $\mathcal{A}$, the matrices $\mathbf{A}_{1}, \ldots, \mathbf{A}_{k_{0}}$ are invertible (see Lemma \ref{lem: abierto Zariski A}).

\end{itemize}

Therefore, given these complex numbers, we define holomorphic transformations  $s_{i}, z_{j} \colon (\C, 0) \rightarrow (\C, 0)$ having $k_{0}$-jet equal to 
\begin{equation*}
\textnormal{j}^{k_{0}} (s_{i}) 
= s_{\scriptscriptstyle{i, 1}} x + s_{\scriptscriptstyle{i, 2}} x^{2} + \ldots + s_{\scriptscriptstyle{i, k_{0}}} x^{k_{0}} \, ,
\end{equation*}
\begin{equation*}
\textnormal{j}^{k_{0}} (z_{j}) 
= z_{\scriptscriptstyle{i, 1}} x + z_{\scriptscriptstyle{i, 2}} x^{2} + \ldots + z_{\scriptscriptstyle{i, k_{0}}} x^{k_{0}} \, ,
\end{equation*}
for $1 \leqslant i \leqslant n+1$, $1 \leqslant j \leqslant m$. From these maps we define the holomorphic foliation $\G_{p_{i}}$ in $(\C^{2}, D_{p_{i}})$ having first integral $u + s_{i} (x)$, and the holomorphic foliation $\G_{q_{j}}$ in $(\C^{2}, D_{q_{j}})$ having first integral $g_{j} (u) + z_{j} (x)$. As a consequence of Theorem \ref{teo:Construccion de PdF a partir de PL}, on realization of local models and normalizing transformations, there exists a foliation pair $(\F, \G)$, where $\F \in \mathcal{N}_{\p}(\mathbf{h})$ and $ \G \in  \mathcal{D}_{\mathsf{q}}(\mathcal{I})$, such that $\left(\mathcal{F}_{p_{i}}^l, \G_{p_{i}}\right)$ is the local model at the singularity $(0, p_{i})$, and $\left((x = \text{ct}), \G_{q_{j}}\right)$ is the local model at the tangency point $(0,q_{j})$. That is, the maps $s_{i}$, $z_{j}$ are the \textit{analytic invariants of the foliation pair $(\F, \G)$} (see Theorem \ref{Teo_Invariantes de clasificacion}).

Except for an analytic change of coordinates applied to the pair $(\F, \G)$, we may assume that the $y$-axis is an invariant branch of the dicritical foliation $\G$. In this way, the normalizing transformations of the pair $(\F, \G)$, except for an analytic change of coordinates, have power series with $k_{0}$-jet as in $(\mathbf{nt_{k_{0}}})$ in Theorem  \ref{Teorema_Limpieza de TransNorm_Intro_1}. That is, the foliation pair $(\F, \G)$ satisfies the property $a)$ stated at the beginning of the proof. The fact that the property $b)$  is also satisfied, follows directly from Remark \ref{Obs_TeoMatricialBis}, since the coefficients of the analytic invariants $s_{i}$, $z_{j}$ of the pair $(\F, \G)$ satisfy the equalities \eqref{Ec_PropoMatrizAk}.
\end{itemize}

  
\subsection{Description of the Polynomial $\mathcal{P}$ determining the Zariski open $\mathcal{A}$} \label{Descripcion polinomio cal P}

Recall the description of the polynomial (\ref{ec: mathcal P}) given in Remark \ref{obs: abierto Zariski A y polinomio P}.
\begin{equation*}
\mathcal{P}(k_{0}, \lambda, \tau ):=\mathbf{P}_{k_{0}} (\tau_{1}, \tau_{2}, \ldots, \tau_{m}) \, \cdot\prod_{\substack{1 \leqslant k \leqslant k_{0} \\ 1 \leqslant i \leqslant n+1 }} \, \bigl(1 \, - \, k\lambda_{i}\bigl) \, .
\end{equation*}

To prove the existence of this polynomial satisfying the properties a) and b) in Lemma \ref{lem: abierto Zariski A} we focus our attention in the polynomial  $\mathbf{P}_{k_{0}}$. We will define it as a product $\mathbf{P}_{k_{0}} = \mathbf{R} \, \mathbf{Q}_{k_{0}} $ of polymonials in $\tau_{1}, \tau_{2}, \ldots, \tau_{m}$, with coefficients given by rational functions in $q_{j} - q_{l}$, $j \neq l$.

\emph{Requirements to define the polynomials $\mathbf{Q}_{k_{0}}$ and $\mathbf{R}$ in order to satisfy properties a) and b) in Lemma \ref{lem: abierto Zariski A}}.

\emph{About property a)}. Let $\mathbf{A}_{k} = \mathbf{A}_{k} (\mathsf{p}, \mathsf{q}; \lambda, \tau)$ be the square matrix of dimension $m + n +1$   defined in \eqref{Ec_MatrizAk}, for $  1 \leqslant k \leqslant k_{0}$. Its determinant is expressed by the relation 
\begin{equation*}
\text{det} (\mathbf{A}_{k}) = \text{det} (\mathbf{\tilde{A}}_{k})
\prod_{i = 1}^{n+1} (1 - k \lambda_{i}) \, ,
\end{equation*}
where $\mathbf{\tilde{A}}_{k} = \left[\alpha_{j \, l}\right]$ is the square matrix of dimension $m$, such that $\alpha_{j \, l} = \tfrac{1}{q_{j} - q_{l}}$, for $j \neq l$, where $\alpha_{j \, j} = \theta_{\scriptscriptstyle{j, k}} = - \tfrac{3}{2} \tau_{j} + k (\varsigma_{\scriptscriptstyle{j, 1}})^{(1)} (q_{j})$. By the properties of the determinants, it can be proved that $\text{det} (\mathbf{\tilde{A}}_{k})$ is a non zero polynomial in $\tau_{1}, \tau_{2}, \ldots, \tau_{m}$. Namely,
\begin{equation*}
\text{det} (\mathbf{\tilde{A}}_{k}) 
= \left(- \tfrac{3}{2}\right)^{m} \, \tau_{1} \,  \tau_{2}  \cdots \tau_{m} \, + \, \cdots \, ,
\end{equation*}
where the multiple dots represent a polynomial in 
 $\tau_{1}, \tau_{2}, \ldots, \tau_{m}$ of degree less or equal to   $m-1$ and simple factors. The coefficients are rational functions in the coordinates $q_{j}$ of $\q$, with denominators given by products of the differences $q_{j} - q_{l}$, $j \neq l$.

We define the polynomial
\begin{equation}
\label{Ec_Polinomio_Qk}
\mathbf{Q}_{k_{0}} \left(\tau_{1}, \tau_{2}, \ldots, \tau_{m}\right) \, := \,
\textnormal{det} (\mathbf{A}_{1}) \, \textnormal{det} (\mathbf{A}_{2}) \, \cdots \, \textnormal{det} (\mathbf{A}_{k_{0}}) \, .
\end{equation}
Note that choosing $\tau_{1}, \tau_{2}, \ldots, \tau_{m} \in \C$ such that  $\mathbf{Q}_{k_{0}}$ is noy identically zero, the property a) takes place.

\emph{About property b)}. We begin with the following remark
\begin{remark}
 It is enough to prove that \begin{equation}
\label{Ec_Propiedad_ii}
\mathbf{A}_{1} \left(\C^{n+1} \times \C_{j}^{m}\right) \, + \, \textnormal{Im} \, (\mathbf{V}) 
\, = \, 
\C^{n+1} \times \C^{m} \,,
\end{equation}
for $1 \leqslant j \leqslant m$, where $\C^{m}_{j}$ consists of the elements $(\mathsf{z}_{1}, \ldots, \mathsf{z}_{m} ) \in \C^{m}$ whose $j$-th coordinate is zero, i.e., $\mathsf{z}_{j} = 0$. In fact, since  $\mathbf{A}_{1} \left(\C^{n+1} \times \C_{j}^{m}\right)$ is a subespace of $\C^{n+1} \times \C^{m}$ having dimension at most $m + n$, then the equality (\ref{Ec_Propiedad_ii}) implies that  $\text{Im} (\mathbf{V})$ is not contained in  $\mathbf{A}_{1} \left(\C^{n+1} \times \C_{j}^{m}\right)$.  Hence, $\textnormal{Im} \, (\mathbf{V})$ satisfies 
\begin{equation*}
\textnormal{Im} \, (\mathbf{V}) \, 
\nsubseteq  
\, \bigcup_{j = 1}^{m} \, 
\mathbf{A}_{1} \left(\C^{n+1} \times \C_{j}^{m}\right)\,.
\end{equation*}
Thus, for any $a \in \C^{n+1} \times \C^{m}$, 
\begin{equation*}
a \, + \, \text{Im} (\mathbf{V}) \,  
\nsubseteq
\, 
\bigcup_{j=1}^{m} \, \mathbf{A}_{1} (\C^{n+1} \times \C^{m}_{j}) \, .
\end{equation*}
\end{remark}

\emph{Proof of equality \eqref{Ec_Propiedad_ii}}.
Let $\mathbf{\Lambda} = \mathbf{\Lambda} (\q; \tau)$ be the $m\times (m+4)$ matrix, with $m$ rows $\mathbf{r}_k$ and $m+4$ columns $\mathbf{c}_l$

\begin{equation*}
\mathbf{\Lambda} = \mathbf{\Lambda} (\q; \tau)=\left[
\begin{array}{c|c}
\begin{matrix}
\mbox{\footnotesize $\theta_{\scriptscriptstyle{1, 1}}$} 
& \tfrac{1}{q_{1} - q_{2}} & \cdots 
& \tfrac{1}{q_{1} - q_{m}}
\\
\tfrac{1}{q_{2} - q_{1}} 
& \mbox{\footnotesize $\theta_{\scriptscriptstyle{2, 1}}$} & \cdots 
& \tfrac{1}{q_{2} - q_{m}} \\
\vdots & \vdots & \ddots & \vdots \\
\tfrac{1}{q_{m} - q_{1}} 
& \tfrac{1}{q_{m} - q_{2}} & \cdots 
& \mbox{\footnotesize $\theta_{\scriptscriptstyle{m, 1}}$}
\end{matrix} 
&
\begin{matrix}
1 & q_{1} & q_{1}^{2}  & q_{1}^{3} \\
1 & q_{2} & q_{2}^{2}  & q_{2}^{3} \\
\vdots & \vdots & \ddots & \vdots \\ 
1 & q_{m} & q_{m}^{2}  & q_{m}^{3}
\end{matrix}
\end{array}
\right] \,.
\end{equation*}
We stress that for $\lambda_{i} \neq 1$, $1 \leqslant i \leqslant n+1$ equality \eqref{Ec_Propiedad_ii} takes place, if and only if, 

\begin{equation}
\label{Ec_Propiedad_ii_Bis}
\mathbf{\Lambda} \left(\C^{m}_{j} \times \C^{4}\right)
\, = \, 
\C^{m} \, .
\end{equation}

For $m \leqslant 4$ the equality \eqref{Ec_Propiedad_ii_Bis} takes place since Vandermonde $\left[q_{r}^{s}\right]$ matrix, having $m$ rows, $1 \leqslant r \leqslant m$, and  $4$ columns, $0 \leqslant s \leqslant 3$, is surjective.  

For $m \geqslant 5$. Let $\mathcal{J} = \{j_{1}, j_{2}, j_{3}, j_{4}\, \, |\, j_{1}<j_{2} < j_{3} < j_{4}\} \, \subset \, \{1, 2, \ldots, m\}$. We define the $m\times m$ matrix $\mathbf{\Lambda}_{\mathcal{J}}$ by removing from matrix $\mathbf{\Lambda}$ the columns $\mathbf{c}_k$ corresponding to the indices $j_k\in \mathcal{J}$:  $\mathbf{c}_{\scriptscriptstyle{j_{1}}}$, $\mathbf{c}_{\scriptscriptstyle{j_{2}}}$, $\mathbf{c}_{\scriptscriptstyle{j_{3}}}$, $\mathbf{c}_{\scriptscriptstyle{j_{4}}}$.

Note that
\begin{equation*}
\textnormal{Im} \, \left(\mathbf{\Lambda}_{\mathcal{J}}\right) 
\, \subseteq \, 
\mathbf{\Lambda} \left(\C^{m}_{j_{i}} \times \C^{4}\right) \, ,
\hspace{0.5cm}
\textnormal{for}
\hspace{0.3cm}
i = 1, 2, 3, 4 \, .
\end{equation*}

If we rearrange the rows $j_{1}, j_{2}, j_{3}, j_{4}$ of the matrix $\mathbf{\Lambda}_{\mathcal{J}}$ so that they occupy, respectively, the rows $m-3, m-2, m-1, m $, the resulting matrix has the form and properties stated in the following lemma. 

\begin{lemma}
\label{Lema_Antisimetrica_Vandermonde}
Let $y_{1}, y_2, \ldots, y_{s+4}$ be $s+4$ pairwise distinct complex numbers. Let $\mathbf{\tilde{\Lambda}} = \mathbf{\tilde{\Lambda}} (y_{i})_{i}$ be the  $(s+4)\times (s+4)$ matrix,
\begin{equation*}
\mathbf{\tilde{\Lambda}} = \mathbf{\tilde{\Lambda}} (y_{i})_{i}=\left[ 
\begin{array}{c|c}
\begin{matrix}
\theta_{1} & 
\tfrac{1}{y_{1} - y_{2}} & 
\cdots & \tfrac{1}{y_{1} - y_{s}}
\\
\tfrac{1}{y_{2} - y_{1}} 
& \theta_{2} & 
\cdots & \tfrac{1}{y_{2} - y_{s}}
\\
\vdots & \vdots & \ddots & \vdots
\\
\, \, \, \tfrac{1}{y_{s} - y_{1}} \, \, \,  & 
\, \, \, \tfrac{1}{y_{s} - y_{2}} \, \, \,  & 
\cdots 
& \, \, \, \theta_{s} \, \, \, 
\end{matrix}
&
\begin{matrix}
1 & y_{1} & y_{1}^{2} & y_{1}^{3}
\\
1 & y_{2} & y_{2}^{2} & y_{2}^{3}
\\
\vdots & \vdots & \vdots & \vdots
\\
\, \, \, 1 \, \, \, 
& \, \, \, y_{s} \, \, \,
& \, \, \, y_{s}^{2} \, \, \, 
& \, \, \, y_{s}^{3} \, \, \,
\end{matrix} 
\\ \midrule 
\begin{matrix}
\tfrac{1}{y_{s+1} - y_{1}} & 
\tfrac{1}{y_{s+1} - y_{2}} & 
\cdots & 
\tfrac{1}{y_{s+1} - y_{s}} 
\\
\vdots & \vdots & \ddots & \vdots
\\
\tfrac{1}{y_{s+4} - y_{1}} & 
\tfrac{1}{y_{s+4} - y_{2}} & 
\cdots & 
\tfrac{1}{y_{s+4} - y_{s}}
\end{matrix}
&
\begin{matrix}
1 & y_{s+1} & y_{s+1}^{2} & y_{s+1}^{3}
\\
\vdots & \vdots & \vdots & \vdots
\\
1 & y_{s+4} & y_{s+4}^{2} & y_{s+4}^{3}
\end{matrix}
\end{array}
\right].
\end{equation*}

Then the determinant of matrix  $\mathbf{\tilde{\Lambda}}$ is given by 
 
\begin{equation*}
\det \,(\mathbf{\tilde{\Lambda}})=\prod_{s+1 \leqslant i < j \leqslant s+4} (y_{j} - y_{i}) \, 
\theta_{1} \cdot \theta_{2} \cdots \theta_{s} \, 
+ \, \cdots 
\end{equation*}
where the multiple points represent a polynomial of degree less than or equal to $s-1$ in $\theta_{1}, \theta_{2},  \ldots, \theta_{s}$, with monomials  
\begin{equation*}
\theta_{j_{1}} \cdot \theta_{j_{2}} \cdots \theta_{j_{k}} \, ,
\hspace{0.5cm}
\hspace{0.3cm}
1 \leqslant j_{1} < j_{2} < \cdots < j_{k} \leqslant s \, ,
\hspace{0.3cm}
0 \leqslant k < s \, . 
\end{equation*}
The coefficients are rational functions on   $y_{1}, y_{2}, \ldots, y_{s+4}$, whose denominators are products of the differences $y_{j} - y_{i}$,  $1 \leqslant i, j \leqslant s+4$, $i \neq j$.

\end{lemma}

\begin{proof}
    
The proof of Lemma \ref{Lema_Antisimetrica_Vandermonde} relies on the  Leibniz formula for determinants and on the explicit expression of the determinants of Vandermonde square matrices.

  By Lemma \ref{Lema_Antisimetrica_Vandermonde} we get that    $\det(\mathbf{\Lambda}_{\mathcal{J}})$, is a polynomial of degree $m-4$ in $\tau_{1}, \tau_{2}, \ldots, \tau_{m}$ variables, except for  $\tau_{j_{1}}, \tau_{j_{2}}, \tau_{j_{3}}, \tau_{j_{4}}$, where the monomials have simple factors. Namely,
  
  \begin{equation*}
\textnormal{det} (\mathbf{\Lambda}_{\mathcal{J}})
\, = \, 
\mathbf{a}_{\mathcal{J}} \, \left(- \tfrac{3}{2}\right)^{m-4} \tfrac{\tau_{1} \, \tau_{2} \, \tau_{3} \, \ldots \, \tau_{m}}{\tau_{j_{1}} \tau_{j_{2}} \tau_{j_{3}} \tau_{j_{4}}} 
\, + \, 
\cdots \,,
\end{equation*}
\\
where $\mathbf{a}_{\mathcal{J}}$ equals to either $\prod_{1 \leqslant i < l \leqslant 4} \left(q_{j_{l}} - q_{j_{i}}\right)$, or 
$-\prod_{1 \leqslant i < l \leqslant 4} \left(q_{j_{l}} - q_{j_{i}}\right)$, and the multiple points denote a polynomial of degree less or equal to $m-5$ whose coefficients are rational functions in the  constants $q_{1},\cdots, q_{m}$, and with denominators equal to the product of the differences $q_{j} - q_{l}$, $j \neq l$.


Let   $1 \leqslant s_{0}$ and $0 \leqslant r_{0} < 4$ be such that $m = 4 s_{0} + r_{0}$. For $1 \leqslant s \leqslant s_{0}$ we define the subset   
\begin{equation*}
\mathcal{J}_{s} := \{ 4s - 3, 4s - 2, 4s - 1, 4s\} \, \subset \, \{1, 2, \ldots, m\}\,,
\end{equation*}
and
\begin{equation*}
    \mathcal{J}_{s_{0}+1} := \{ m - 3,  m - 2, m - 1, m\}\,.
\end{equation*}
 Since for any $j \in \{1, 2, \ldots, m\}$ there exists a natural number $1 \leqslant s \leqslant s_{0} + 1$ such that $j \in \mathcal{J}_{s}$.
Hence, 
\begin{equation} 
\label{Ec_Propiedad_ii_Prima}
\textnormal{Im} \, \left(\mathbf{\Lambda}_{\mathcal{J}_{s}}\right) 
\, \subseteq \, 
\mathbf{\Lambda} \left(\C^{m}_{j} \times \C^{4}\right) \, .
\end{equation}
We define the polynomial $\mathbf{R}$ in  the variables $\tau_{1}, \tau_{2}, \ldots, \tau_{m}$ by 
\begin{equation}
\label{Ec_Polinomio_R}
\mathbf{R} \left(\tau_{1}, \tau_{2}, \ldots, \tau_{m}\right) \, := \, 
\textnormal{det} (\mathbf{\Lambda}_{\mathcal{J}_{1}}) \,
\textnormal{det} (\mathbf{\Lambda}_{\mathcal{J}_{2}}) \, \cdots \,
\textnormal{det} (\mathbf{\Lambda}_{\mathcal{J}_{s_{0}}}) \, 
\textnormal{det} (\mathbf{\Lambda}_{\mathcal{J}_{s_{0} + 1}}) \, . 
\end{equation}
Its coefficients are rational functions in the coordinates  $q_{j}$ of $\q$, whose denominators given by products of the differences $q_{j} - q_{l}$, $j \neq l$. 

Assume that $\tau_{1}, \tau_{2}, \ldots, \tau_{m} \in \C$ are chosen in such a way that the polynomial $\mathbf{R}$ does not vanish. Then, for every $1 \leqslant s \leqslant s_{0} + 1$, the matrix $\mathbf{\Lambda}_{\mathcal{J}_{s}}$ is invertible. Together with the inclusion \eqref{Ec_Propiedad_ii_Prima}, implies the property  \eqref{Ec_Propiedad_ii_Bis}.
 
The existence of the Polynomial $\mathcal{P}$ determining the Zariski open $\mathcal{A}$ is proved.
\end{proof}

\subsubsection*{Conclusion}
In this work we gave an exhaustive description of the curves of tangencies associated with pairs of foliations determined by germs of dicritical and non dicritical vector fields satisfying some genericity assumptions. To this purpose we used local models and analytic normalizing transformations.  Moreover, for each natural number $k$ we gave  \emph{k-normal forms} for the normalizing transformations. These normal forms were used to give parametrizations, up to a finite jet, of the branches of the curves of tangencies.  We proved as well, that under genericity assumptions on the classes of non dicritical and dicritical foliations, any germ of analytic curve having pairwise transversal smooth branches may be realized as curve of tangencies of a  --\emph{non dicritical and dicritical}-- foliation pair.

In a work in progress we relate these results to Thom's analytical classification invariants given in \cite{[ORV 1]} and \cite{[ORV 2]}. Namely, using the collection of curves of tangency we give a geometric interpretation of the finite collection of parametric analytical invariants appearing in the formal normal forms of the non dicritical and dicritical cases, respectively. 

Furthermore, we are interested in using these results in the study of foliations with singular points with greater degeneration.

We declare that our investigations are original and have not been published or submitted elsewhere. We have no conflicts of interest to disclose.

\subsection*{Acknowledgements}
The authors express their gratitude to L. Rosales-Ortiz for the careful elaboration of the figures in this text.

\begin{thebibliography}{CCD13}

\bibitem[CS82]{[Cam-Sad]} C. Camacho, P. Sad, \textit{Invariant Varieties through Singularities of Holomorphic Vector Fields}, Ann. of Math. 115 (3); 1982; p. 579-595.

\bibitem[GO]{[G-O]} X.Gomez-Mont, L.Ortiz-Bobadilla, \textit{Sistemas din\'amicos holomorfos en superficies}, Aportaciones Matem\'aticas, Investigaci\'on 3, segunda edici\'on, UNAM, 2004.


\bibitem[Gra]{[Gra]} J.-M. Granger, \textit{Sur un espace de modules de germe de courbe plane}, Bull. Sc. Math., $2^{e}$ s\'erie, $103$; 1979; pp. 3-16.
 

\bibitem[Gr62]{[Grauert]} H. Grauert, \textit{\"Uber Modifikationen und exzeptionelle anlytische Mengen}, Math. Ann., $146$; 1962; p.331-368.

\bibitem[JOV]{[JOV]} Jaurez-Rosas J., Ortiz-Bobadilla L., Voronin S.M., \textit{Local models and curves of tangencies of  foliation pairs}, preprint.

\bibitem[KP02]{[Krantz-Parks]} S. Krantz, H. Parks, \textit{A Primer of Real Analytic Functions}, second edition, Birkh\"auser, 2002.

\bibitem[Li87]{[LinsNeto]} Lins-Neto A., \textit{Construction of singular holomorphic vector fields and foliations in dimension two}. J. Differential Geom. $26$, Number $1$; 1987; pp. 1–31. 


\bibitem[ORV 1]{[ORV 1]} L. Ortiz-Bobadilla, E. Rosales-Gonz\'alez, S. M. Voronin, \textit{Rigidity Theorems for Generic Holomorphic Germs of Dicritic Foliations and Vector Fields in $(\C^{2}, 0)$}, Moscow Mathematical Journal, Vol. $5$, Number $1$; 2005; pp. 171-206.

\bibitem[ORV 2]{[ORV 2]} Ortiz-Bobadilla L., Rosales-Gonz\'alez E.,   Voronin S.M., \textit{Thom's Problem for the Orbital Analytic Classification of Degenerate Singular Points of Holomorphic Vector Fields in the Plane},
 Moscow Mathematical Journal, Vol. $12$, Number $4$; 2012; pp. 825-862.

\bibitem[P]{[P]} Poincaré, H, \emph{Sur les propriétés des fonctions définies par les équations aux différences partielles}, Thèses présentées à la Faculté des sciences de Paris, 1er août $1879$, Paris, Gauthier-Villars, $93$ pages. Œuvres, tome I, pp. \textsc{XLIX-CXXXI}.

\bibitem[V]{[V]} Voronin, S.M. {\it Orbital analytic equivalence of degenerate singular points of holomorphic vector fields on the complex
plane}, Tr. Mat. Inst. Steklova {\bf 213} (1997), Differ. Uravn.
  s Veshchestv. i Kompleks. Vrem., 35--55 (Russian).
\end{thebibliography}
\end{document}